\documentclass[10pt]{amsart}
\usepackage{amsmath, amsfonts, amscd, amssymb, amsthm, bm}
\usepackage{color}
\usepackage{psfrag}
\usepackage[breaklinks]{hyperref}
\usepackage{verbatim}

\newcommand{\C}{\mathbb{C}}
\newcommand{\R}{\mathbb{R}}
\newcommand{\Z}{\mathbb{Z}}
\newcommand{\N}{\mathbb{N}}
\newcommand{\RP}{\mathbb{RP}}
\newcommand{\DD}{\mathbb{D}}
\newcommand{\ep}{\varepsilon}
\newcommand{\dbar}{\bar\partial}
\newcommand{\tl}{\tilde}
\newcommand{\wt}{\widetilde}
\newcommand{\X}{X_{\lambda}}
\newcommand{\Xf}{X_{f}}
\newcommand{\A}{\mathbf{A}}

\newcommand{\winfty}{\wind_{\infty}}
\newcommand{\B}{\mathcal{B}}
\newcommand{\M}{\mathcal{M}}
\newcommand{\J}{\mathcal{J}}
\newcommand{\F}{\mathcal{F}}

\newcommand{\dz}{\partial_{z}}
\newcommand{\dx}{\partial_{x}}
\newcommand{\dy}{\partial_{y}}
\newcommand{\dth}{\partial_{\theta}}
\newcommand{\drh}{\partial_{\rho}}
\newcommand{\dph}{\partial_{\phi}}
\newcommand{\dR}{\partial_{R}}
\newcommand{\dTh}{\partial_{\Theta}}
\newcommand{\pathcz}{\mu_{cz}}
\newcommand{\wind}{\operatorname{wind}}
\newcommand{\Hom}{\operatorname{Hom}}
\newcommand{\End}{\operatorname{End}}

\newcommand{\ind}{\operatorname{ind}}
\newcommand{\inum}{\operatorname{int}}
\newcommand{\ain}{i_{\infty}}
\newcommand{\lne}{\sigma^{-}_{max}}
\newcommand{\vspan}{\operatorname{span}}

\def\ndbar#1{\dbar_{N}^{\nabla}(#1)}
\def\bp#1{\left(#1\right)}
\def\bbr#1{\left[#1\right]}
\def\br#1{\left\{#1\right\}}
\def\fl#1{\lfloor#1\rfloor}

\def\abs#1{\left|#1\right|}

\newtheorem{lemma}{Lemma}[section]
\newtheorem{theorem}[lemma]{Theorem}
\newtheorem{corollary}[lemma]{Corollary}
\newtheorem{proposition}[lemma]{Proposition}

\theoremstyle{definition}
\newtheorem{definition}[lemma]{Definition}
\newtheorem*{ack}{Acknowledgements}
\newtheorem{remark}[lemma]{Remark}

\numberwithin{equation}{section}

\begin{document}
\title[Connected Sum Foliations I]{Connected sums and finite energy foliations I: Contact connected sums}
\author[J.\ W.\ Fish]{Joel W.\ Fish}
\address{Department of Mathematics \\ University of Massachusetts Boston \\ Boston, MA 02125 \\ U.S.A.}
\email{\href{mailto:joelfish@umb.edu}{joelfish@umb.edu}}
\urladdr{\url{http://www.joelfish.com}}
\thanks{J.\ W.\ Fish's research was partially supported by NSF grant DMS-0844188 and the Ellentuck Fund.}
\author[R.\ Siefring]{Richard Siefring}
\address{Fakult\"at f\"ur Mathematik \\ Ruhr-Universit\"at Bochum \\
     44780 Bochum \\
    Germany}
\email{\href{mailto:richard.siefring@ruhr-uni-bochum.de}{richard.siefring@ruhr-uni-bochum.de}}
\urladdr{\url{http://homepage.ruhr-uni-bochum.de/richard.siefring}}
\thanks{Richard Siefring is supported by DFG grant BR 5251/1-1.}
\date{\today}

\begin{abstract}
We consider a $3$-manifold $M$ equipped with nondegenerate contact form $\lambda$
and compatible almost complex structure $J$.
We show that if the data $(M, \lambda, J)$ admits a stable finite energy foliation, then for a
generic choice of distinct points $p$, $q\in M$,
the manifold $M'$
formed by taking the contact connected sum at $p$ and $q$ admits a
nondegenerate contact
form $\lambda'$ and compatible almost complex structure $J'$ so that the data
$(M', \lambda', J')$ also admits a stable finite energy foliation.
Along the way, we develop some general theory for the study of finite energy foliations.
\end{abstract}

\maketitle
\tableofcontents

\pagebreak

\section{Introduction and the main result}

Let $(M, \xi=\ker\lambda)$ be a closed contact $3$-manifold
equipped with a nondegenerate contact form $\lambda$.
Recall a complex structure $J$ on $\xi$ is said to be
compatible with the data $(M, \lambda)$ if
$d\lambda(\cdot, J\cdot)$ is a bundle metric on $\xi$.
We denote the set of complex structures on $J$ compatible with $(M, \lambda)$
by $\J(M, \lambda)$.  Given a $J\in\J(M, \lambda)$ we can extend it in the usual
way to an $\R$-invariant almost complex structure $\tl J$ on $\R\times M$
by requiring
\[
\tl J\partial_{a}=\X \qquad\text{ and }\qquad \tl J|_{\pi_{M}^{*}\xi}=\pi_{M}^{*}J
\]
where $\partial_{a}$ is the coordinate field along $\R$ and
$\pi_{M}:\R\times M\to M$ is the canonical projection onto the second factor.

Given data $(M, \lambda, J)$ where $(M, \lambda)$ is a $3$-manifold $M$
with contact form $\lambda$ and $J\in\J(M, \lambda)$ is a compatible complex multiplication on
$\xi=\ker\lambda$, a finite-energy pseudoholomorphic map
in $\R\times M$ is a quadruple
$(\Sigma, j, \Gamma, \tl u)$ where
$(\Sigma, j)$ is a compact Riemann surface, $\Gamma\subset\Sigma$ is a finite set,
and $\tl u=(a, u):\Sigma\setminus\Gamma\to \R\times M$ is a map
satisfying
\[
\tl J\circ d\tl u=d\tl u\circ j
\]
and
\[
0<E(\tl u)<\infty
\]
where the energy $E(\tl u)$ of the map is defined by
\[
E(\tl u)=\sup_{\varphi\in\Xi}\int_{\Sigma\setminus\Gamma}\tl u^{*}d(\varphi\lambda)
\]
where $\Xi$ is defined by
\[
\Xi=\br{\varphi\in C^{\infty}(\R, [0, 1])\,|\, \varphi'(x)\ge 0}.
\]
A finite-energy pseudoholomorphic curve in $\R\times M$
is then an equivalence class $C=[\Sigma, j, \Gamma, \tl u]$ under
the equivalence relation of holomorphic reparametrization.
Since the energy of a pseudoholomorphic map is invariant under holomorphic reparametrization of 
the domain, pseudoholomorphic curves have a well-defined energy.
A now well-known result of Hofer \cite{hofer1993} tells us that near the
(nonremovable) punctures, finite-energy
pseudoholomorphic curves are asymptotic to periodic orbits of the Reeb vector field.

A \emph{finite energy foliation} $\F$ for the data $(M, \lambda, J)$
is
a collection of connected finite-energy pseudoholomorphic curves
with uniformly bounded energies
whose images form
a smooth foliation of $\R\times M$.
We define the energy $E(\F)$ of a foliation $\F$ to be
the supremum of the energies of the curves in the foliation,
that is
\[
E(\F)=\sup_{C\in\F}E(C).
\]
A finite energy foliation $\F$ for the data $(M, \lambda, J)$ is said to be \emph{stable} if:
\begin{enumerate}
\item For any $C\in\F$, $C$ is either a trivial cylinder over a periodic orbit or
the Fredholm index $\ind(C)$ (see Section \ref{s:fredholm} below)
is either $1$ or $2$.
\item For any two curves $C_{1}$, $C_{2}\in\F$ with $\ind(C_{i})\in\br{1, 2}$,
the holomorphic intersection number from \cite{siefring2011}
(see Section \ref{s:intersections} below) $C_{1}*C_{2}$ vanishes.
\end{enumerate}
The word ``stable'' here is meant to connote the fact that both the existence of
such a finite energy foliation and its basic structure will persist under suitable
sufficiently small perturbations of the data $(\lambda, J)$.
As we will discuss in Section \ref{s:foliating-curves} below,
a stable finite energy foliation necessarily
consists of only punctured spheres and is
invariant under the $\R$-action on $\R\times M$ given by shifting
the $\R$-coordinate.
The $\R$-invariance
in turn lets us conclude that the projections of the curves in the foliation
to the $3$-manifold $M$ are embedded, transverse to the flow of the Reeb vector field, and
foliate the complement of a finite collection of periodic orbits in $M$.

The study of finite energy foliations was initiated by Hofer, Wysocki, and Zehnder in
\cite{hwz:convex} in which they use the existence of a finite energy foliation
to construct a global surface of section of disk type for a
$3$-dimensional strictly convex energy surface in $(\R^{4}, \sum_{i=1}^{2}dx_{i}\wedge dy_{i})$.
This work was extended in \cite{hwz:foliations} where the same authors show that 
any nondegenerate star-shaped hypersurface in $\R^{4}$ admits a 
stable finite energy foliation.  Using this fact, they then show that
there exists a Baire set of star-shaped hypersurfaces in $\R^{4}$ so that any
given hypersurface in this set has either precisely two or infinitely many periodic orbits.

Recently, Bramham has introduced the use of finite energy foliations
to the study of area-preserving maps of the disk \cite{bramham2008, bramham-foliations}.
Using the foliations that he constructs in \cite{bramham-foliations}, Bramham
proves in \cite{bramham2015-1} that 
every smooth, irrational pseudorotation of the $2$-disk
is the uniform limit of a sequence of maps which are each conjugate to a rotation
about the origin.
In \cite{bramham2015-2}, these foliations are again used to prove
there is a dense subset $\mathcal{L}_{*}\subset\mathcal{L}$ of the 
Liouville numbers so that a pseudorotation of the disk
with rotation number in $\mathcal{L}_{*}$ has a sequence of iterates which converge
uniformly to the identity map and thus such a pseudorotation can't exhibit
strong mixing.
A discussion of further applications of finite energy foliations to the study of disk maps
can be found in the survey \cite{bramhamhofer}.

The existence of finite energy foliations has also had applications in contact
and symplectic topology.
Among these are
Hind's work on Lagrangian unknottedness in Stein surfaces \cite{hind2012}
and Wendl's work on fillabilty of contact $3$-manifolds \cite{wendl2010-fillable}.
Further work either addressing existence of finite energy foliations
or in which the existence of finite energy foliations
play a role in dynamical or contact/symplectic topological results can be found in
\cite{abbas2011,
abb-cie-hof,
contrerasoliveira,
depaulosalamao,
etnyreghrist,
hind2003,
hind2004,
hindlisi,
hwz:s3-1,
hwz:s3-2,
hryniewicz2012,
hryniewiczsalamao2011,
hryniewicz2014,
hutchingstaubes,
latschevwendl,
niederkruegerwendl,
wendl2005,
wendl2008,
wendl2010-openbook,
wendl2013}.

In the present series of papers we develop abstract tools for extending previously
known existence results for stable finite energy foliations.
One motivation for this work comes from  the study of the planar, circular, restricted
three-body problem.
Albers, Frauenfelder, van Koert and Paternain show in
\cite{afkp} that near the two massive primaries,
the regularized energy levels below and slightly above the first Lagrange point
are diffeomorphic respectively to two copies of $\RP^{3}$ and the connected
sum of two copies of $\RP^{3}$.
In \cite{affhk}, Albers, Frauenfelder, Hofer, van Koert, and the first author apply
techniques from \cite{hwz:convex} to construct finite energy foliations for
many mass ratios and regularized energy levels below the first Lagrange point.
Since many classical techniques to study the restricted three-body problem fail above the first
Lagrange point, it's of interest to know whether the existence of
finite energy foliations for regularized energy surfaces below the first Lagrange point
can be used to deduce the existence of 
finite energy foliations for regularized energy surfaces above the first Lagrange point.

The results from \cite{afkp,affhk} and the associated problem of attempting to construct a
finite energy foliation for regularized energy surfaces above the first Lagrange point
naturally lead to the general question of whether
the existence of finite energy foliations persist under
the formation of a contact connected sums
as in \cite{meckert1982,weinstein1991}.
Our main theorem, which we state now, answers this question by showing
that finite energy foliations do indeed persist after forming the
contact connected sum of a contact manifold.
 
\begin{theorem}\label{t:0-surg-main}
Let $(M, \xi)$ be a contact $3$-manifold with contact structure induced by
a nondegenerate contact form $\lambda$,
and let $J\in\J(M, \lambda)$ be a complex multiplication for which
the triple $(M, \lambda, J)$ admits a stable finite energy foliation $\F$
of energy $E(\F)$.
Then, there exists an open, dense set
$\mathcal{U}\subset M\times M\setminus\Delta(M)$ so that for
any $(p, q)\in\mathcal{U}$ the contact manifold $(M', \xi')$
obtained by performing a contact
connected sum at $(p, q)$ as in \cite{meckert1982,weinstein1991}
admits a nondegenerate contact form $\lambda'$ with $\xi'=\ker\lambda'$,
a compatible $J\in\J(M', \lambda')$ and
a stable finite energy foliation $\F'$ for the data $(M', \lambda', J')$ with 
energy $E(\F')=E(\F)$.
\end{theorem}

We briefly discuss some of the key steps of the proof of this theorem.
Given a finite energy foliation $\F$ for the data $(M, \lambda, J)$
we choose any two distinct points $p$ and $q\in\M$
lying on distinct (up to $\R$-translation) index-$2$ leaves of the foliation.
We then form the connected sum $M'$ by $S^{2}$-compactifying
$M\setminus\br{p, q}$ and gluing along the newly created boundary.
We denote the new manifold by $M'$, the induced inclusion
$M\setminus\br{p, q}\hookrightarrow M'$ by $i$, and the embedded sphere
$M'\setminus i(M\setminus\br{p, q})$ by $S$.
It is well known from
\cite{meckert1982,weinstein1991}
that the gluing can be done in such a way that
$M'$ is a smooth  manifold and the induced contact structure
continues smoothly across $S$.
We show that, in addition, we can find a contact form $\lambda'$ and
compatible $J'$ which agree respectively with $\lambda$ and $J$
outside of any desired sufficiently-small
neighborhood $U$ of $S$ so that
there is precisely one simple periodic orbit $\gamma_{0}\subset S$ contained in $U$ and
so that $\gamma_{0}$ divides
$S$ into two disks, each of which is the projection to 
$M'$ of an index-$1$ $\tl J'$-holomorphic plane.
Using the fact that curves in a stable finite energy foliation must satisfy
so-called automatic transversality conditions, we investigate the boundaries
of the moduli spaces of curves surrounding the neighborhood $U$ of $S$.
Using intersection theory arguments, and specifically a result concerning the direction
of approach of a curve to an orbit with even Conley--Zehnder index,
we show that these families of curves
converge to height-$2$ pseudoholomorphic buildings with one of the planes
in $S$ as one of the nontrivial components and that the resulting collection of curves
forms a finite energy foliation.

Since the finite energy foliation we construct on the connected sum always contains
a pair of rigid (i.e.\ index-$1$) planes asymptotic to same periodic orbit, it is natural to ask
the question of whether the operation can be reversed anytime one has a foliation
with a similar configuration of curves in it.
We show in \cite{fs2} that this in fact can be done.  Specifically, assuming the data
$(M, \lambda, J)$ admits a finite energy foliation $\F$
containing two distinct (up to $\R$-translation)
index-$1$ planes asymptotic to the same periodic orbit, the manifold $M'$ obtained
by doing surgery on the $2$-sphere formed by the orbit and 
the projections of the planes to $M$ admits a contact form $\lambda'$
and a compatible $J'\in\J(M', \lambda')$ so that the data $(M', \lambda', J')$ admits
a finite energy foliation $\F'$ with $E(\F')=E(\F)$.
We further show in \cite{fs3} that a Weinstein cobordism connecting
the two contact manifolds admits a finite energy foliation which is
asymptotic to the foliations $\F$ and $\F'$ on the boundaries.

Our main result in this paper
can be combined with previous existence results for finite energy foliations
to produce new finite energy foliations.
We recall that for contact manifolds whose contact structures are supported by
planar open book decompositions, Abbas \cite{abbas2011} and Wendl \cite{wendl2010-openbook}
have constructed finite energy foliations consisting of curves which
correspond via the projection
$\R\times M\to M$ to pages of
the open book decomposition.  Moreover, it is known that performing a contact connected sum
at two points in a contact manifold supported by a planar open book decomposition produces
a contact manifold whose contact structure
is also supported by a planar open book decomposition with one
additional binding component.\footnote{
This can be seen from the following argument which was explained to the second author by
O.\ van Koert.
Forming a contact connected sum
at two points in a given connected manifold $(M, \xi)$ is the same
as forming a contact connected sum of $(M, \xi)$ with
$(S^{1}\times S^{2}, \xi_{0})$ where $\xi_{0}$ is the contact structure
arising as the kernel of the $S^{1}$-invariant contact form
$\lambda_{0}=\cos\theta\,dt+\sin^{2}\theta\,d\phi$ with $t\in\R/\Z$ the coordinate along $S^{1}$, and
with
$\phi\in\R/2\pi\Z$ and $\theta\in[0, \pi]$, respectively, the polar and azimuthal coordinates on $S^{2}$.
Forming the book connected sum (see e.g.\ Section 4.5 in \cite{quinn}
or Section 5.2.3 in \cite{vankoert-thesis})
of an open book decomposition for $(M, \xi)$ with the open book decomposition 
for $(S^{1}\times S^{2}, \xi_{0})$ by cylinders of the form $S^{1}\times \br{\phi=c}$
yields an open book decomposition for the contact connected sum
$(M, \xi)\#(S^{1}\times S^{2}, \xi_{0})$ 
with pages having the same genus as those of the given open book decomposition
for $(M, \xi)$
and one additional boundary component.
}
Combining these facts with our construction leads to the following theorem.
\begin{theorem}
Let $(M_{0}, \xi_{0})$ be a contact manifold with contact structure supported
by a planar open book decomposition with $b$ binding components, and let
$(M_{k}, \xi_{k})$ denote the contact manifold obtained from performing
$k$ successive contact connected sums on $(M_{0}, \xi_{0})$.
Then for any integer $\ell\in [1, k]$ there exist
a contact form $\lambda_{\ell}$ with $\ker\lambda_{\ell}=\xi_{k}$,
a compatible complex structure $J_{\ell}$ on $\xi_{k}$,
and a stable finite energy foliation $\F_{\ell}$ for the data $(M, \lambda_{\ell}, J_{\ell})$ consisting of:
\begin{itemize}
\item $b+k-\ell$ trivial cylinders over elliptic orbits $\br{\gamma^{e}_{1}, \dots, \gamma^{e}_{b+k-\ell}}$,
\item $\ell$ trivial cylinders over even orbits $\br{\gamma^{h}_{1}, \dots, \gamma^{h}_{\ell}}$,
\item $\ell$ pairs of index-$1$ families of planes, with one such pair of families asymptotic
to each of the periodic orbits $\gamma^{h}_{i}$,
\item $\ell$ pairs of index-$1$ families of curves, $1$ pair for each hyperbolic orbit
$\gamma^{h}_{i}$, having $b+k-\ell$ positive punctures with the collection of elliptic orbits
$\br{\gamma^{e}_{1}, \dots, \gamma^{e}_{b+k-\ell}}$
as asymptotic limits and one negative puncture with $\gamma^{h}_{i}$ as an asymptotic limit, and
\item $2\ell$ families of index-$2$ curves in which each curve has 
$b+k-\ell$ positive punctures with the collection of elliptic orbits
$\br{\gamma^{e}_{1}, \dots, \gamma^{e}_{b+k-\ell}}$
as asymptotic limits.
\end{itemize}
\end{theorem}
\begin{proof}
Given an integer $\ell\in[1, k]$ we consider the contact
manifold $(M_{k-\ell}, \xi_{k-\ell})$ obtained by performing
$k-\ell$ successive contact connected sums on $(M_{0}, \xi_{0})$.
According to the observation above,
since $(M_{0}, \xi_{0})$ is supported by a planar open book decomposition
with $b$ binding components,
$(M_{k-\ell}, \xi_{k-\ell})$ is supported by a planar open book decomposition with
$b+k-\ell$ binding components.
The constructions of Abbas \cite{abbas2011} and Wendl \cite{wendl2010-openbook}
then provide a finite energy foliation whose leaves
are all either trivial cylinders over elliptic orbits
corresponding to bindings of the open book decomposition
or index-$2$ curves having only positive punctures which project to $M$
to correspond with the pages of the open book decomposition.

Given this, we then carry out our construction $\ell$ times on the given open book decomposition.
Since varying the points at which the connected sum occurs leads to contactomorphic contact manifolds,
we are free to choose the points at each step to lie on an index-$2$ curve in the foliation.
As is shown in Section \ref{s:main-proof}, each time we apply our construction
we add one even orbit to the foliation with two planes (modulo the $\R$-action) asymptotic to it and
two index-$1$ curves (modulo the $\R$-action) asymptotic to that orbit with a negative puncture.
Moreover the
four different height-$2$ holomorphic buildings that can be formed by these curves are
each a boundary component of one of the surrounding index-$2$ curves, all of which
have the same number of punctures and asymptotic behavior as the curves in the original foliation.
Since these curves form $1$-dimensional families modulo the $\R$-action, and each such family has
$2$ boundary components, there must be exactly $2\ell$ of them.
\end{proof}
\noindent
This theorem shows that manifolds of the form $(M_{k}, \xi_{k})$ satisfying the
hypotheses above admit at least $k+1$ different stable finite energy foliations:
the holomorphic open book decomposition $\F_{0}$ constructed by
Abbas \cite{abbas2011} and Wendl \cite{wendl2010-openbook}
and the $k$ foliations $\F_{\ell}$ constructed by the theorem.
An interesting question for further investigation is whether or not the
various foliations on the manifold $(M_{k}, \xi_{k})$ are,
in language introduced in \cite{wendl2005}, concordant; that is whether or not
two given foliations $\F_{i}$ and $\F_{j}$ on $(M_{k}, \xi_{k})$
can be related by a non-$\R$-invariant pseudoholomorphic foliation
$\F_{ij}$ of
$\R\times M_{k}$ which is asymptotic
at the positive and negative ends to $\F_{i}$ and $\F_{j}$ respectively.

Our construction can also be used in
combination with the results of Hofer, Wysocki, and Zehnder \cite{hwz:foliations} to 
 make local changes to the structure of a
given finite energy foliation on a manifold.
Indeed,
we recall that forming the contact connected sum of a contact manifold $(M, \xi)$ with
a copy of tight $S^{3}$ produces a contact structure on $M$ contactomorphic to the original.
Given then data $(M, \lambda, J)$ admitting a foliation, we can
apply our construction to the given foliation on $M$ and any
of the foliations on tight $S^{3}$ constructed by Hofer, Wysocki, and Zehnder in \cite{hwz:foliations}
to produce a new foliation on $M$ for a contact form inducing a contact structure contactomorphic to the
original.	
For example, by applying our construction to
any number of copies of tight $S^{3}$ equipped with an open book decomposition consisting
of pseudoholomorphic
planes and a planar contact manifold equipped with a holomorphic planar open book decomposition
constructed by Abbas \cite{abbas2011} and Wendl \cite{wendl2010-openbook},
one can construct a stable finite energy foliation for a given planar contact manifold having
any desired number of even Conley--Zehnder index orbits appearing as asymptotic limits of
leaves of the foliation.

We make two observations about the new foliations produced
by forming a connected sum
of a foliation $\F$ for data $(M, \lambda, J)$ with a 
Hofer--Wysocki--Zehnder foliation $\F_{HWZ}$ on tight $S^{3}$.
The first is that the new foliation
$\F'=\F\#\F_{HWZ}$
can, with some additional work,
be shown to be
concordant to the original foliation $\F$.
A second observation about this construction is that,
since Hofer--Wysocki--Zehnder foliations exist for any nondegenerate contact form on
$S^{3}$, there is a good deal of freedom in the choice of contact form and almost complex
structure on the $S^{3}$ portion of the connected sum.
Indeed, forming the connected sum of a foliation with a Hofer--Wysocki--Zehnder foliation
creates a $3$-ball in the manifold bounded by a $2$-sphere
composed of a pair of pseudoholomorphic planes.  The foliation created will persist under arbitrary
perturbations of the contact form and almost complex structure which are compactly supported in
the interior of this three-ball, provided the resulting contact form is nondegenerate and the
almost complex structure is regular.  Thus, the existence of a foliation for a given set of
data $(M, \lambda, J)$ is a property which is persistent under a large class of
localized perturbations to the data $(\lambda, J)$.

In regards to the original motivation for this work from the three-body problem,
we point out that our main theorem here does not immediately
imply the existence of a finite energy foliation for level sets of the Hamiltonian having energies just
above the first Lagrange point.
However, given the existence of finite energy foliations for level sets below the first Lagrange point,
our result allows one to construct a Hamiltonian on the same phase space
having an interval of regular level sets
which admit finite energy foliations and 
which are homotopic to a level set of the
original Hamiltonian for the three-body problem having energy just above the first Lagrange point.
This fact along with some deformation results for finite energy foliations
currently being developed by the second author
would then allow one to construct a finite energy foliation for any small
nondegenerate perturbation of a level set of the original Hamiltonian.
In some cases these foliations could then be used to construct foliations for
level sets of the Hamiltonian 
via a limiting argument as in \cite{hwz:convex}.

Finally, we remark that for simplicity of presentation, and for the convenience of being
able to quote results from other papers, we have chosen to focus on the
case of contact manifolds equipped with a nondegenerate contact form
and we don't impose any conditions on the rates of convergence of curves in the foliation
to their asymptotic limits.
However,
with the use of the in-progress work \cite{siefringwendl},
which generalizes the intersection theory
of \cite{siefring2011} to include exponential weights and Morse--Bott nondegenerate orbits,
it is straightforward to adapt our arguments to
somewhat more general situations.  The essential point 
is that appropriate generalizations of the necessary results from 
intersection theory \cite{siefring2011} and Fredholm theory \cite{hwz:prop3,wendl2010-automatic}
are true provided the curves in question approach their asymptotic limits exponentially fast.
Given this, one can consider $\R$-invariant finite energy foliations in a manifold with a
degenerate contact form, provided all curves converge exponentially to their asymptotic limits.
The arguments in the proofs of
\cite[Theorem 2.4]{siefring2011} and
\cite[Theorem 2.6]{siefring2011}
show that, for appropriate choices of exponential weights, the exponentially
weighted version of the $*$-product, developed in \cite{siefringwendl},
must vanish between any two of the nontrivial curves of such a foliation.
Such a foliation would then be called a \emph{weighted, stable finite energy foliation}
if all weighted Fredholm indices of nontrivial curves are $1$ or $2$.
Given results that will be proved in \cite{siefringwendl},
it is straightforward to adapt our arguments to work for weighted, stable finite
energy foliations.

\begin{remark}
In the recent work \cite{depaulosalamao}, de Paulo and Salom\~ao study
Hamiltonians $H$ on $\R^{4}$ having a saddle-center equilibrium point 
lying on a strictly convex singular subset $S_{0}\subset H^{-1}(0)$.  They show
that for all sufficiently small positive energies $E$, there is a subset
$S_{E}\subset H^{-1}(E)$ diffeomorphic to the closed three ball
so that the symplectization $\R\times S_{E}$ admits a finite energy foliation.
The structure of the finite energy foliation that they construct is the same
as that which would result from our construction when taking a connected
sum with $S^{3}$ equipped with one of the pseudoholomorphic open book decompositions
constructed by Hofer, Wysocki, and Zehnder in \cite{hwz:convex}.
\end{remark}

\begin{ack}
We'd like to thank Peter Albers for posing the question that lead to this series of papers.
We'd also like to thank the Max Planck Institute for Mathematics in the Sciences, and
in particular J\"urgen Jost and Matthias Schwarz, for providing a supportive research environment.
\end{ack}

\subsection{Outline of the paper}
While most earlier results on finite energy foliations deal with relatively concrete constructions,
the results we prove in the present series of papers deal with finite energy foliations abstractly.
The proofs of our results thus require us to develop some general theory for finite energy foliations.
To assist in this, we review some background about Reeb dynamics and pseudoholomorphic
curves in Sections \ref{s:reeb-dynamics} and \ref{s:curves}, with a special focus
on giving precise statements and references for facts that we will need in this paper and its sequels.

We start by recalling relevant facts about contact geometry and Reeb dynamics in
Section \ref{s:reeb-dynamics}, primarily focusing on material concerning properties of the
Conley--Zehnder index from \cite{hwz:prop2}.
Then in Section \ref{s:curves} we review background on finite-energy pseudoholomorphic
curves.  First, in Section \ref{s:asymptotics} we recall the basic asymptotic convergence
to a periodic orbit, established by Hofer in \cite{hofer1993}, as well as the refined
relative asymptotic formula of the second author from \cite{siefring2008}.
Then in Section \ref{s:compactness}, we recall the compactification of the space
of finite-energy pseudoholomorphic curves.  Of particular relevance here
is the work of Wendl \cite{wendl2010-compactness} which focuses on what sort of
limiting objects can arise as sequences of so-called nicely-embedded curves.
After that, in Section \ref{s:intersections}, we review results related to the intersection product
for finite-energy pseudoholomorphic curves introduced by the second author in \cite{siefring2011}.
An adaptation of a result from \cite{siefring2011} concerning the direction of approach
of curves to even orbits
will be key for the proof of our main theorem.
Finally, in Section \ref{s:fredholm} we recall facts about the Fredholm theory
of embedded finite-energy pseudoholomorphic curves from \cite{hwz:prop3} and review
so-called automatic transversality conditions \cite{hofer-lizan-sikorav, hwz:prop3, abb-cie-hof, wendl2010-automatic}
which give topological criteria that guarantee the moduli space of curves is a smooth
manifold of dimension equal to the Fredholm index.

General discussion of finite energy foliations begins in Section \ref{s:foliating-curves}.
After giving a definition of stable finite energy foliations
we establish some basic properties of stable finite energy foliations that follow from this definition.
We then discuss some facts about the structure of the moduli spaces of curves which appear
in finite energy foliations.
In Section \ref{s:connect-sum} we show that contact connected sums can be
formed in a way which gives us properties necessary to prove our main theorem.
In order to focus on the main ideas, some of the more straightforward but tedious
computations needed to support claims in this section are delayed to Appendix \ref{a:details}.
Finally in Section \ref{s:main-proof}, we give the proof of our main theorem.

\section{Background in contact geometry and Reeb dynamics}\label{s:reeb-dynamics}

In this section we review some basic notions from contact geometry and Reeb dynamics that we will need,
and fix some notation.  Much of the material from this section,
particularly that material pertaining to the Conley--Zehnder index of periodic orbits, is adapted from
\cite{hwz:prop2}.

Let $M$ be a closed, oriented $3$-manifold.  Recall that a \emph{contact form} on $M$ is a
$1$-form $\lambda$ for which
\begin{equation}\label{e:contact-condition}
\text{$\lambda\wedge d\lambda$ is a volume form on $M$.} 
\end{equation}
This condition implies that there is a unique vector field $\X$, called the
\emph{Reeb vector field} associated to $\lambda$, satisfying the conditions
\begin{equation}\label{e:reeb-conditions}
i_{\X}\lambda=1 \qquad\text{ and }\qquad i_{\X}d\lambda=0.
\end{equation}
The \emph{contact structure} $\xi$ determined by $\lambda$ is defined by
$\xi=\ker\lambda$.  As a result of condition \eqref{e:contact-condition} the contact
structure is necessarily a $2$-plane bundle transverse to $\X$, and
$d\lambda$ restricts to a nondegenerate form on $\xi$.  The contact form $\lambda$ thus
determines a splitting
\begin{equation}\label{e:splitting}
TM=\R\X\oplus(\xi, d\lambda)
\end{equation}
of the tangent space $TM$ of $M$ into a framed line bundle and a symplectic $2$-plane bundle.
Moreover, the defining conditions \eqref{e:reeb-conditions} for $\X$ used with
the formula $L_{X}=i_{X}\circ d+d\circ i_{X}$ imply that
\[
L_{\X}\lambda=0 \qquad\text{ and }\qquad L_{\X}d\lambda=0
\]
and thus the flow of $\X$ preserves the splitting \eqref{e:splitting}.

It will be convenient for our purposes here to think of periodic orbits
of the Reeb vector field as maps from the circle $S^{1}=\R/\Z$.
In particular,
for $T> 0$
we consider a $T$-periodic orbit to be
a map $\gamma:S^{1}\to M$ satisfying
\[
\dot\gamma(t)=T\cdot\X(\gamma(t)).
\]
An unparametrized periodic orbit is a collection of parametrized orbits
that differ by reparametrization via the $S^{1}$-action on the domain.
We will generally use the same notation for a parametrized orbit and its associated unparametrized
orbit, allowing the context or specific language to distinguish between the two.

Let $\psi_{\cdot}:\R\times M\to M$ be the flow generated by the Reeb vector field $\X$, that is
\[
\dot\psi_{t}(x)=\X\circ\psi_{t}(x),
\]
and let $\gamma:S^{1}\to M$ be a parametrized $T$-periodic orbit.
Since the flow of $\X$ preserves the splitting \eqref{e:splitting}, we obtain
for any $t\in S^{1}$ a symplectic map
\[
d\psi_{T}(\gamma(t))\in Sp(\xi_{\gamma(t)}, d\lambda),
\]
and, since
the group property of the flow and its linearization can be used to 
show that
\begin{equation}\label{e:return-map}
d\psi_{T}(\gamma(t))=[d\psi_{-tT}(\gamma(t))]^{-1}d\psi_{T}(\gamma(0))d\psi_{-tT}(\gamma(t)),
\end{equation}
the spectrum of $d\psi_{T}(\gamma(t))$ is independent of $t\in S^{1}$.
We will thus say that an unparametrized $T$-periodic orbit $\gamma$ is \emph{nondegenerate}
if for a representative parametrization $\gamma:S^{1}\to M$,
the map
$d\psi_{T}(\gamma(0))$ does not have $1$ in the spectrum.
A contact form $\lambda$ on $M$ is said to \emph{nondegenerate} if all periodic orbits are nondegenerate.

A nondegenerate $T$-periodic orbit $\gamma$ is said to be:
\begin{itemize}
\item elliptic if $d\psi_{T}(\gamma(t))$ has complex eigenvalues, or
\item hyperbolic if $d\psi_{T}(\gamma(t))$ has real eigenvalues.
\end{itemize}
Moreover, $\gamma$ is said to be:
\begin{itemize}
\item odd if $\gamma$ is elliptic, or if $\gamma$ is hyperbolic and $d\psi_{T}(\gamma(t))$ has negative eigenvalues, or 
\item even $\gamma$ is hyperbolic and $d\psi_{T}(\gamma(t))$ has positive eigenvalues.
\end{itemize}
As a result of \eqref{e:return-map} the designation of a nondegenerate orbit as even/odd, positive/negative
is a well-defined property associated to the unparametrized orbit.
The parity of a periodic orbit as defined here agrees with the parity of the orbit's
Conley--Zehnder index, which we will now define.

Given a trivialization of the contact structure along a nondegenerate periodic orbit, one can assign
a number, called the Conley--Zehnder index, to the orbit which can be thought of as a measure of
the winding with respect to the given trivialization of the linearized flow along the orbit
\cite{conleyzehnder, SalamonZehnder, RobbinSalamon, hwz:prop2}.
We review the key information now.
As a starting point we recall information about
the Maslov index and Conley--Zehnder index for, respectively, loops and paths in
$Sp(1)=Sp(\R^{2}, \omega_{0}=dx\wedge dy)$.
We first recall that the fundamental group
$\pi_{1}(Sp(1))$ of the symplectic group is
isomorphic to $\Z$ (see e.g.\ \cite[Section 1.2.1]{abbondandolo}).
The \emph{Maslov index} of a (homotopy class of) loop(s) of matrices in
$Sp(1)$ based at the identity
is, by definition, the isomorphism
\[
m:\pi_{1}(Sp(1))\to\Z
\]
determined by assigning a value of $1$ to the (homotopy class of the) loop
\[
t\in S^{1}=\R/\Z\mapsto
\begin{bmatrix}
\cos 2\pi t & -\sin 2\pi t \\ \sin 2\pi t & \cos 2\pi t
\end{bmatrix}
\]
which is a generator of $\pi_{1}(Sp(1))$.
Given this, we can define the Conley--Zehnder index for (homotopy classes of) paths
in $Sp(1)$ that start at the identity and end at a matrix without $1$ in the spectrum
via the following axiomatic characterization from \cite[Theorem 3.2]{hwz:prop2}.
\begin{theorem}\label{t:conley-zehnder-axioms}
Let
\begin{equation}\label{e:cz-symplectic-paths}
\Sigma(1)=
\br{\Psi\in C^{0}([0, 1], Sp(1)) \,|\, \Psi(0)=I \text{ and } \det(\Psi(1)-I)\ne 0}
\end{equation}
denote the space of continuous paths
in $Sp(1)$ which start at the identity and end at a matrix without $1$ in the spectrum.
There exists a unique map
\[
\pathcz:\Sigma(1)\to\Z,
\]
called the Conley--Zehnder index, determined by the following axioms:
\begin{enumerate}
 \item Homotopy invariance: The Conley--Zehnder index of a path in $\Sigma(1)$ is invariant under homotopies
 of paths in $\Sigma(1)$.
 \item Maslov compatibility: If $\Psi\in\Sigma(1)$ and $g:[0, 1]\to Sp(1)$ is a loop based at the identity, then
 \[
 \pathcz(g\Psi)=2m(g)+\pathcz(\Psi)
 \]
 where $g\Psi\in\Sigma(1)$ is the path defined by $(g\Psi)(t)=g(t)\Psi(t)$.
 \item Inverse axiom: If $\Psi\in\Sigma(1)$ and $\Psi^{-1}\in\Sigma(1)$ is the inverse path defined
 by $\Psi^{-1}(t)=[\Psi(t)]^{-1}$, then
 \[
 \pathcz(\Psi^{-1})=-\pathcz(\Psi).
 \]
\end{enumerate}
\end{theorem}

Now, let $\gamma:S^{1}\to M$ by a nondegenerate, $T$-periodic orbit, and let
$\Phi:S^{1}\times \R^{2}\to \gamma^{*}\xi$
by a symplectic trivialization, that is, assume that
\[
d\lambda_{\gamma(t)}\bp{\Phi(t)\cdot, \Phi(t)\cdot}
=
dx\wedge dy
\]
for all $t\in S^{1}$.
Again, recalling that $L_{\X}d\lambda=0$, the flow $\psi_{t}$ of $\X$
gives for any $t\in\R$ a symplectic map
\[
d\psi_{tT}(\gamma(0)):(\xi_{\gamma(0)}, d\lambda_{\gamma(0)})\to (\xi_{\psi_{tT}(\gamma(0))}, d\lambda_{\psi_{tT}(\gamma(0))})=(\xi_{\gamma(t)}, d\lambda_{\gamma(t)})
\]
and thus the map
\begin{equation}\label{e:triv-path}
t\in[0, 1]\to \Phi^{-1}(t)d\psi_{tT}(\gamma(0))\Phi(0)
\end{equation}
gives a path of matrices in $Sp(1)$ starting at the identity and ending
at
\[
\Phi^{-1}(1)d\psi_{T}(\gamma(0))\Phi(0)=\Phi^{-1}(0)d\psi_{T}(\gamma(0))\Phi(0)
\]
which doesn't have $1$ in the spectrum by the assumption that $\gamma$ is nondegenerate.
We define the \emph{Conley--Zehnder index} $\mu^{\Phi}(\gamma)$ of the orbit $\gamma$
relative to the trivialization $\Phi$ by
\begin{equation}\label{e:cz-orbit-definition}
\mu^{\Phi}(\gamma):=\pathcz(\Psi)
\end{equation}
with $\Psi\in\Sigma(1)$ the path \eqref{e:triv-path}, and $\pathcz(\Psi)$ the Conley--Zehnder index of
the path $\Psi$ as characterized in Theorem \ref{t:conley-zehnder-axioms}.
We note that, as a result of the homotopy invariance axiom from Theorem \ref{t:conley-zehnder-axioms},
the Conley--Zehnder index of an orbit is invariant under homotopies of the trivialization.
Furthermore, the homotopy invariance axiom can be used to show that the Conley--Zehnder index 
relative to a given trivialization is independent of the choice of parametrization of the orbit.
Finally, we note that, as result of the Maslov compatibility axiom, the parity of
the Conley--Zehnder index of an orbit does not depend on the choice of trivialization.
Further, this parity can be shown to agree with that defined above in terms of the eigenvalues of
the linearized flow.

We will need the characterization of the Conley--Zehnder index in terms
of the spectrum of
a certain self-adjoint operator acting on sections of the contact structure along the orbit
from \cite{hwz:prop2}.
Let $\gamma$ be a parametrized $T$-periodic orbit, and let
$h:S^{1}\to \xi$ be a smooth section of the contact structure along $\gamma$, 
i.e.\ $h(t)\in \xi_{\gamma(t)}$ for all $t\in S^{1}$.
We observe that since $h$ is defined along a flow line of $\X$, it has a well-defined
Lie derivative $L_{\X}h$ defined by
\begin{equation}\label{e:lie-derivative}
L_{\X}h(t)=\left.\frac{d}{ds}\right|_{s=0}d\psi_{-s}(\gamma(t+s/T))h(t+s/T)
\end{equation}
and, since the flow $\psi_{t}$ of $\X$ preserves the splitting \eqref{e:splitting},
$L_{\X}h$ is also a section of the contact structure along $\gamma$.
Given any symmetric connection $\nabla$ on $TM$,
we use that $\dot\gamma(t)=T\cdot\X(\gamma(t))$ to write
\[
T\cdot L_{\X}h=\nabla_{t}h-T\nabla_{h}\X.
\]
Thus $\nabla_{t}\cdot-T\nabla_{\cdot}\X$ give a first-order differential operator
on $C^{\infty}(\gamma^{*}\xi)$ which is independent of choice of symmetric connection.

Next, recall that given a symplectic vector bundle $(E, \omega)$ a complex
structure $J$ on $E$ is said to be compatible with $\omega$ if the bilinear form
\[
g_{J}(\cdot, \cdot)=\omega(\cdot, J\cdot)
\]
is a metric on $E$.  It is a well known fact the space of compatible almost complex structure on
a given symplectic vector bundle is nonempty and contractible in the $C^{\infty}$ topology
(see e.g.\ Proposition 5 and discussion thereafter in Section 1.3 of \cite{hoferzehnder}).
Recalling that $(\xi, d\lambda)$ is a symplectic vector bundle, we define the set
$J(M, \lambda)\subset\End(\xi)$ to be the set of complex structures on $\xi$ compatible with $d\lambda$.
Given a $T$-periodic orbit $\gamma$ and a $J\in J(M, \lambda)$, 
we define the \emph{asymptotic operator} $\A_{\gamma, J}$
associated to $\gamma$ and $J$
by
\begin{equation}\label{e:asymptotic-operator-definition}
\A_{\gamma, J}h=-J(\nabla_{t}h-T\nabla_{h}\X),
\end{equation}
and note that, by the discussion of the previous paragraph, $\A_{\gamma, J}$
gives a first-order differential operator on $C^{\infty}(\gamma^{*}\xi)$ which is independent of the choice
of symmetric connection used to define it.

We define an inner product $\langle\cdot, \cdot\rangle_{J}$
on $C^{\infty}(\gamma^{*}\xi)$ by
\[
\langle h, k \rangle_{J}
=\int_{S^{1}}d\lambda_{\gamma(t)}\bp{h(t), J(\gamma(t))k(t)}\,dt.
\]
Recalling that $L_{\X}d\lambda=0$, we have for any $h$, $k\in C^{\infty}(\gamma^{*}\xi)$
that
\[
\frac{d}{dt}\bbr{d\lambda_{\gamma(t)}\bp{h(t), k(t)}}
=d\lambda_{\gamma(t)}\bp{T(L_{\X}h)(t), k(t)}+ d\lambda_{\gamma(t)}\bp{h(t), T(L_{\X} k)(t)}.
\]
Using that compatibility of $J$ with $(\xi, d\lambda)$ implies that
$d\lambda(J\cdot, J\cdot)=d\lambda$ on $\xi\times\xi$,
we can integrate the above equation to give
\[
\langle h, \A_{\gamma, J}k\rangle_{J}=\langle\A_{\gamma, J}h, k\rangle_{J}.
\]
Thus $\A_{\gamma, J}$ is formally self-adjoint, and induces
a self-adjoint operator
\[
\A_{\gamma,J}:D(\A_{\gamma,J})=H^{1}(\gamma^{*}\xi)
\subset L^{2}(\gamma^{*}\xi, \langle\cdot, \cdot\rangle_{J})
\to L^{2}(\gamma^{*}\xi, \langle\cdot, \cdot\rangle_{J}).
\]
Since for any value in the resolvent set of $\A_{\gamma, J}$,
the associated resolvent operator factors
through the compact embedding
$H^{1}(\gamma^{*}\xi)\hookrightarrow L^{2}(\gamma^{*}\xi)$, we know
from the spectral theorem for compact self-adjoint operators that 
the spectrum of $\A_{\gamma, J}$ consists of real, isolated eigenvalues of finite multiplicity
accumulating only at $\pm\infty$.

We recall the observation from \cite{hwz:prop1} that
$\ker\A_{\gamma, J}$ is trivial if and only if $\gamma$ is a nondegenerate orbit.
Indeed, if $h$ is a section of $\gamma^{*}\xi$ in the kernel of $\A_{\gamma, J}$
then $L_{\X}h=0$ and thus $d\psi_{tT}h(t_{0})=h(t_{0}+t)$ for any $t\in\R$
and $t_{0}\in S^{1}$.  In particular $d\psi_{T}(\gamma(t_{0}))h(t_{0})=h(t_{0}+1)=h(t_{0})$
so $d\psi_{T}(\gamma(t_{0}))$ has $1$ as an eigenvalue and $\gamma$
must be a degenerate orbit.
Conversely, if the orbit is degenerate, then $d\psi_{T}(\gamma(t_{0}))$ has $1$ as an eigenvalue.
Letting $v_{0}\in\xi_{\gamma(0)}\setminus\br{0}$ be a vector with
$d\psi_{T}(\gamma(0))v_{0}=v_{0}$,
the map $v:\R\to\xi$ defined by $v(t)=d\psi_{tT}(\gamma(0))v_{0}\in\xi_{\gamma(t)}$ will be
$1$-periodic and
satisfy $L_{\X}v=0$, thus determining a section of $\gamma^{*}\xi$
in the kernel of $\A_{\gamma, J}$.

In a unitary trivialization of $(\gamma^{*}\xi, d\lambda, J)$ --- 
that is, a symplectic trivialization $\Phi:S^{1}\times\R^{2}\to\gamma^{*}\xi$
of $(\gamma^{*}\xi, d\lambda)$
satisfying
\[
\Phi\circ J_{0}=J\circ \Phi
\]
with 
\[
J_{0}=
\begin{bmatrix}
0 & -1 \\ 1 & 0
\end{bmatrix}
\]
--- the operator $\A_{\gamma, J}$ takes the form
\[
\Phi^{-1}\circ \A_{\gamma, J}\circ \Phi=-i\tfrac{d}{dt}-S(t)
\]
with $S(t)$ a symmetric matrix.
An eigenvector of $\A_{\gamma,J}$ satisfies a linear, first-order ordinary differential equation
 and therefore never vanishes
since it doesn't vanish identically.
Hence every eigenvector gives a map from $S^{1}\to\R^{2}\setminus\{0\}$
and thus has a well-defined winding number.
Since $-i\frac{d}{dt}-S(t)$ is a compact perturbation of $-i\frac{d}{dt}$,
it can be shown using perturbation theory in \cite{kato} that
that the winding is monotonic in the eigenvalue
and that to any $k\in\Z$
the span of the set of eigenvectors having winding $k$ is two dimensional.
These facts are proved in Section 3 of \cite{hwz:prop2}, and we restate them here as a lemma.

\begin{lemma}\label{l:operator-spectrum}
Let $\gamma$ be a $T$-periodic orbit of $\X$, let $\A_{\gamma, J}$
denote the asymptotic operator of $\gamma$, and let
$\mathfrak{T}(\gamma^{*}\xi)$ denote
the set of homotopy classes of symplectic trivializations of $(\gamma^{*}\xi, d\lambda)$.  There exists a map
$w:\sigma(\A_{\gamma, J})\times\mathfrak{T}(\gamma^{*}\xi)\to\Z$ which satisfies
\begin{enumerate}
\item If $e:S^{1}\to \gamma^{*}\xi$ is an
eigenvector of $\A_{\gamma, J}$ with eigenvalue $\lambda$, then $w(\lambda, [\Phi])=\wind(\Phi^{-1}e)$, 
that is, $w(\lambda, [\Phi])$ measures the winding with respect to $\Phi$ of any eigenvector of
$\A_{\gamma, J}$ with eigenvalue $\lambda$.

\item For any fixed $[\Phi]\in\mathfrak{T}(\gamma^{*}\xi)$ we have that
\[
w(\lambda, [\Phi])<w(\mu, [\Phi])\Rightarrow \lambda<\mu,
\]
that is, the winding of eigenvectors of $\A_{\gamma, J}$ is (not necessarily strictly)
monotonic in the eigenvalue.

\item If $m(\lambda)=\dim\ker(\A_{\gamma, J}-\lambda)$ denotes the multiplicity of $\lambda$ as an eigenvalue we have for every
$k\in\Z$ and $[\Phi]\in\mathfrak{T}(\gamma^{*}\xi)$ that
\[
\sum_{\{\lambda \,|\, w(\lambda, [\Phi])=k\}}m(\lambda)=2,
\]
that is, the span of the set of eigenvectors of $\A_{\gamma, J}$
with any given winding has dimension $2$.
\end{enumerate}
\end{lemma}

We now describe the characterization of the Conley--Zehnder in terms of the asymptotic operator
from \cite{hwz:prop2}.
Given a $T$-periodic orbit $\gamma$ and a $J\in\J(M, \lambda)$
let $\lne(\gamma)\in\sigma(\A_{\gamma, J})$ denote the largest negative eigenvalue of $\A_{\gamma, J}$.
Given a trivialization $\Phi$ of $\gamma^{*}\xi$, we define
\begin{equation}\label{e:alpha}
\alpha^{\Phi}(\gamma)=w(\lne(\gamma); [\Phi])
\end{equation}
so that $\alpha^{\Phi}(\gamma)$ is the winding relative to $\Phi$ of any eigenvector of
$\A_{\gamma, J}$
having the largest possible negative eigenvalue.  We define the parity of $p(\gamma)$ of
$\gamma$ by
\begin{equation}\label{e:parity}
p(\gamma)=
\begin{cases}
0  & \text{ if $\exists\mu\in\sigma(\A_{\gamma, J})\cap\R^{+}$ with $w(\mu,[\Phi])=\alpha^{\Phi}(\gamma)$ }\\
1  & \text{ otherwise }
\end{cases}
\end{equation}
i.e.\ the parity is $0$ if there is a positive eigenvalue with eigenvectors having winding equal
to that of those eigenvectors having largest negative eigenvalue, and the parity is $1$ otherwise.
The following theorem then gives a formula for the Conley--Zehnder index
of $\gamma$ in terms of the quantities $\alpha^{\Phi}$ and $p$ just defined.

\begin{theorem}[Hofer--Wysocki--Zehnder \cite{hwz:prop2}]\label{t:spectral-conley-zehnder}
Let $\gamma$ be a $T$-periodic orbit of the Reeb vector field $\X$,
let $\Phi:S^{1}\times\R^{2}\to\gamma^{*}\xi$ be a symplectic trivialization of
$(\gamma^{*}\xi, d\lambda)$ and let $\alpha^{\Phi}(\gamma)$ and $p(\gamma)$
be as defined in \eqref{e:alpha}--\eqref{e:parity} above.  Then
the Conley--Zehnder index of $\gamma$ relative to $\Phi$ is given by the formula
\[
\mu^{\Phi}(\gamma)=2\alpha^{\Phi}(\gamma)+ p(\gamma).
\]
\end{theorem}

Finally we close this section by stating a formula for how the Conley--Zehnder index of
iterates of an orbit.  This lemma follow from facts about $Sp(1)$ which can be found
in e.g.\ \cite[Appendix 8.1]{hwz:foliations} or \cite[Section 1.2]{abbondandolo}.

\begin{lemma}\label{l:cz-iteration}
Let $\gamma$ be a periodic orbit of the Reeb vector field $\X$
and let $\Phi:S^{1}\times\R^{2}\to\gamma^{*}\xi$ be a symplectic trivialization.
Assume that for each positive integer $m$, the periodic orbit
$\gamma^{m}$ defined by $\gamma^{m}(t)=\gamma(mt)$
is nondegenerate.
Then:
\begin{itemize}
\item If $\gamma$ is a hyperbolic orbit
\[
\mu^{\Phi}(\gamma^{m})=m\mu^{\Phi}(\gamma).
\]
\item If $\gamma$ is an elliptic orbit, there exists an irrational number $\theta$ so that
\[
\mu^{\Phi}(\gamma^{m})=2\fl{m\theta}+1.
\]
\end{itemize}
\end{lemma}

\section{Background on pseudoholomorphic curves}\label{s:curves}

In this section we review some basic facts about punctured pseudoholomorphic curves.
First, in Section \ref{s:asymptotics} we review the basic set-up and
review some facts about the asymptotic behavior of finite-energy curves
from \cite{hofer1993,hwz:prop1,mora,siefring2008}.
Next, in Section \ref{s:compactness}, we recall the compactification of the space
of finite-energy curves \cite{behwz}, focusing on
a result from \cite{wendl2010-compactness}
concerning
the extra properties
that can be
proved about the compactification when restricting attention to sequences of
curves which project to embeddings in the $3$-manifold.
In Section \ref{s:intersections} we recall facts about the intersection theory of
finite-energy curves from \cite{siefring2011}.
Of particular importance here is a slight generalization of a result from \cite{siefring2011}
concerning curves which approach an even orbit in the same direction.
Finally, in Section \ref{s:fredholm},
we recall facts about the Fredholm theory of embedded finite-energy curves from
\cite{hwz:prop3}.

\subsection{Basic set-up and asymptotic behavior}\label{s:asymptotics}
Let $(M, \lambda)$ be $3$-manifold equipped with a nondegenerate contact from,
and recall from the previous section that we defined $\J(M, \lambda)$ to be
the collection of complex structures on the contact structure $\xi$ compatible with $d\lambda$.
Given a $J\in\J(M, \lambda)$ we can extend it
in the usual manner
to an $\R$-invariant almost complex structure $\tl J$ on
$\R\times M$ by requiring
\begin{equation}\label{e:tilde-J-def}
\tl J\partial_{a}=\X \qquad\text{ and }\qquad \tl J|_{\pi_{M}^{*}\xi}=\pi_{M}^{*}J
\end{equation}
where $a$ is the coordinate in $\R$, and
$\pi_{M}:\R\times M\to M$ is the canonical projection onto the second factor.
We consider quintuples
$(\Sigma, j, \Gamma, a, u)$ where
\begin{itemize}
\item $(\Sigma, j)$ is a compact Riemann surface,
\item $\Gamma\subset\Sigma$ is a finite set called the set of \emph{punctures}, and
\item $\tl u:=(a, u):\Sigma\setminus\Gamma\to\R\times M$ is a smooth map.
\end{itemize}
We define the energy of such a quintuple by
\begin{equation}\label{e:energy-def}
E(\tl u)=\sup_{\varphi\in\Xi}\int_{\Sigma\setminus\Gamma}\tl u^{*}d(\varphi\lambda)
\end{equation}
where $\Xi$ is defined by
\[
\Xi=\br{\varphi\in C^{\infty}(\R, [0, 1])\,|\, \varphi'(x)\ge 0}.
\]
The data $(\Sigma, j, \Gamma, a, u)$ is said to be a \emph{finite-energy pseudoholomorphic map}
if the map $\tl u$ has finite energy and is $\tl J$-holomorphic, that is, if
\begin{equation}\label{e:tilde-j-hol-eqn}
\tl J\circ d\tl u=d\tl u\circ j
\end{equation}
and
\[
E(\tl u)<\infty.
\]
A \emph{finite-energy pseudoholomorphic curve} is an equivalence class
$C=[\Sigma, j, \Gamma, a, u]$
of finite-energy pseudoholomorphic maps $(\Sigma, j, \Gamma, a, u)$ under the equivalence relation
of holomorphic reparametrization of the domain.
For a given $3$-manifold $M$ equipped with a nondegenerate contact form
$\lambda$, and compatible $J\in\J(M, \lambda)$, we will denote the moduli space
of finite-energy $\tl J$-holomorphic curves by $\M(\lambda, J)$.

If $(\Sigma, j, \Gamma, \tl u=(a, u))$ is a $\tl J$-holomorphic map, then we can use the  
$\R$-invariance of $\tl J$ defined by \eqref{e:tilde-J-def} to conclude
that the map $\tl u_{c}:=(a+c, u)$ obtained by translating
the $\R$-coordinate by a constant is also a $\tl J$-holomorphic map, and it is moreover
easily shown that $E(\tl u)=E(\tl u_{c})$.
Thus there is an $\R$-action
on the space of finite-energy
$\tl J$-holomorphic curves given by translating the $\R$-coordinate by a constant
and, in fact, the $M$-component $u$ of $\tl u=(a, u)$ determines the $\R$-component $a$ up
to a constant.
To see this, we define
$\pi_{\xi}:TM=\R\X\oplus\xi\to\xi$ to be the projection of $TM$ onto $\xi$ determined by the splitting
\eqref{e:splitting}.
It then follows from the definition of $\tl J$ that the equation \eqref{e:tilde-j-hol-eqn}
is equivalent to the pair
of equations
\begin{subequations}\label{e:J-hol-eqn-projected}
\begin{align}
u^{*}\lambda\circ j&=da \label{e:J-hol-eqn-R-comp}  \\
J\circ \pi_{\xi}\circ du&=\pi_{\xi}\circ du\circ j  \label{e:J-hol-eqn-M-comp}
\end{align}
\end{subequations}
and from the first of these equations it's clear that the map $u$ determines
$da$, and thus $a$ up to a constant.
We will define a \emph{projected (finite-energy) pseudoholomorphic map} to be
a quintuple $(\Sigma, j, \Gamma, da, u)$ 
satisfying equations \eqref{e:J-hol-eqn-projected}
for which the associated map $\tl u=(a, u)$ to $\R\times M$ has finite energy.
A  \emph{projected (finite-energy) pseudoholomorphic curve} is then an equivalence
class $C=[\Sigma, j, \Gamma, da, u]$
of projected pseudoholomorphic maps
under the equivalence relation of holomorphic reparametrization of the domain.
For a given $3$-manifold $M$ equipped with a nondegenerate contact form
$\lambda$, and compatible $J\in\J(M, \lambda)$, we will denote the moduli space
of projected, finite-energy $\tl J$-holomorphic curves by $\M(\lambda, J)/\R$.

In his work on the Weinstein conjecture \cite{hofer1993}, Hofer showed that 
near the nonremovable punctures of a finite-energy pseudoholomorphic curves, there are
sequences of loops whose images under $u$ converge to periodic orbits of the Reeb vector field.
In the case that the periodic orbit of the Reeb vector field is nondegenerate, then more can be said
about this convergence.
Suppose that $\lambda$ is a nondegenerate contact form and
$(\Sigma, j, \Gamma, a, u)$ is a finite-energy pseudoholomorphic map.
Then, for each puncture $z_{0}\in\Gamma$ there are three possibilities:
\begin{enumerate}
\item Removable punctures: The map $\tl u=(a, u)$ is bounded near $z_{0}$, in which case
$\tl u$ admits a smooth, $\tl J$-holomorphic extension over the puncture.
\item Positive punctures: The function $a$ is bounded from below near $z_{0}$ but not from above.
In this case there exists a nondegenerate periodic orbit $\gamma$ with period $T\le E(\tl u)$
and a holomorphic coordinate system
\[
\phi:[R, \infty)\times S^{1}\subset \R\times S^{1} \approx \C/i\Z\to\Sigma\setminus\br{z_{0}}
\]
on a punctured neighborhood of $z_{0}$
so that 
the maps
$\tl v_{c}:[R, \infty)\times S^{1}\to\R\times M$ defined by
\[
\tl v_{c}(s, t)=(a(s+c/T, t)-c, u(s+c/T, t))
\]
converge in $C^{\infty}([R, \infty)\times S^{1}, \R\times M)$ as $c\to\infty$ to the map
\[
(s, t)\mapsto (Ts, \gamma(t)).
\]
\item Negative punctures: The function $a$ is bounded from above near $z_{0}$ but not from below.
In this case there exists a nondegenerate periodic orbit $\gamma$ with period $T\le E(\tl u)$
and a holomorphic coordinate system
\[
\phi:(-\infty, -R]\times S^{1}\subset \R\times S^{1} \approx \C/i\Z\to\Sigma\setminus\br{z_{0}}
\]
on a punctured neighborhood of $z_{0}$
so that 
the maps
$\tl v_{c}:(-\infty, -R]\times S^{1}\to\R\times M$ defined by
\[
\tl v_{c}(s, t)=(a(s-c/T, t)+c, u(s-c/T, t))
\]
converge in $C^{\infty}((-\infty, -R]\times S^{1}, \R\times M)$ as $c\to\infty$ to the map
\[
(s, t)\mapsto (Ts, \gamma(t)).
\]
\end{enumerate}
We will henceforth assume that all removable punctures have been removed, and thus 
that all punctures are either positive or negative punctures at which the curves in question
are asymptotic to cylinders over periodic orbits.

We will need more precise information about the asymptotic behavior of curves near a puncture,
in particular that convergence is exponential in nature and
the finer behavior of the map (and differences between two maps)
can be described
in terms of eigenvectors of the asymptotic operator associated to the orbit.
Before stating the appropriate result, we first establish some language.
Let $(\Sigma, j, \Gamma, a, u)$ be a pseudoholomorphic map
and assume that $\tl u=(a, u)$ has a positive puncture at $z_{0}\in \Gamma$
where $\tl u$ is asymptotic to a cylinder over the $T$-periodic orbit $\gamma$.
A map $U:[R, \infty)\times S^{1}\to\gamma^{*}\xi$ with $U(s, t)\in\xi_{\gamma(t)}$ for all
$(s, t)\in [R, \infty)\times S^{1}$ is called an
\emph{asymptotic representative of $\tl u$ near $z_{0}$} if there exists
a map $\phi:[R, \infty)\times S^{1}\to\Sigma\setminus\br{z_{0}}$ with
$\lim_{s\to\infty}\phi(s, t)=z_{0}$ so that
\[
\tl u\circ \phi(s, t)=\bp{Ts, \exp_{\gamma(t)}U(s, t)}
\]
where $\exp$ is the exponential map of some metric on $M$.\footnote{
In \cite{siefring2008}
a specific metric is used in the definition of asymptotic representative
but that specific choice of metric is not essential for
Theorem \ref{t:relative-asymptotics} to remain true.
}
Asymptotic representatives at negative punctures are defined
similarly but as maps from negative half-cylinders
of the form $(-\infty, -R]\times S^{1}$.
The following theorem, concerning the asymptotic behavior of
differences of asymptotic representatives,
is proved in \cite{siefring2008}.

\begin{theorem}\label{t:relative-asymptotics}
Let $U$, $V:[R, \infty)\times S^{1}\to\gamma^{*}\xi$
be smooth maps with
$U(s, t)$, $V(s, t)\in\xi_{\gamma(t)}$
representing positive pseudoholomorphic half-cylinders
(or, respectively, 
let $U$, $V:(-\infty, -R]\times S^{1}\to\gamma^{*}\xi$
be smooth maps with
$U(s, t)$, $V(s, t)\in\xi_{\gamma(t)}$
representing negative pseudoholomorphic half-cylinders).
Then either $U-V$ vanishes identically or
\[
U(s, t)-V(s, t)=e^{\sigma s}[e(t)+r(s, t)]
\]
where
\begin{itemize}
\item $\sigma$ is a negative (resp.\ positive) eigenvalue of 
the asymptotic operator $\A_{\gamma, J}$ (defined in \eqref{e:asymptotic-operator-definition}),
\item $e\in\ker(\A_{\gamma, J}-\sigma)\setminus\br{0}$ is an eigenvector of $\A_{\gamma, J}$ with eigenvalue $\sigma$, and
\item $\nabla_{s}^{i}\nabla_{t}^{j}r(s, t)\to 0$ 
as $s\to\infty$ (resp.\ $s\to-\infty$) exponentially for all $(i, j)\in\N^{2}$.
\end{itemize}
\end{theorem}

The special case of this theorem
where $V\equiv 0$ recovers the asymptotic results for single half-cylinders from
\cite{hwz:prop1, mora}.
As is shown in \cite{hwz:prop2}, the asymptotic formula for a single half-cylinder allows one
to assign a local invariant to each puncture, known as the asymptotic winding.
Indeed, as a result of this formula,
the $M$-portion $u$ of a given pseudoholomorphic map
$(\Sigma, j, \Gamma, a, u)$
can be written near some given puncture $z_{0}\in\Gamma$
\[
u\circ \psi_{z_{0}}(s, t)=\exp_{\gamma(t)}U_{z_{0}}(s, t)
\]
with the asymptotic representative $U$ satisfying a formula of the form
\[
U(s, t)=e^{\sigma s}[e(t)+r(s, t)]
\]
with $\sigma$, $e$, and $r$ satisfying the conditions listed above.
Since eigenvectors of the asymptotic operator $\A_{\gamma, J}$ are nowhere vanishing,
the fact that $r$ converges to $0$ as $|s|\to\infty$ implies that
$U(s, t)$ is nonvanishing for all sufficiently large $|s|$, or equivalently that
in some neighborhood of the puncture,
$u$ does not intersect its asymptotic limit
$\gamma$.
Choosing a trivialization of $\gamma^{*}\xi$, we define the asymptotic winding of $\tl u$
at $z_{0}$ by
\[
\winfty^{\Phi}(u; z_{0})=\wind(\Phi^{-1}U_{z_{0}}(s, \cdot))
\]
with the right-hand side being well defined and independent of all sufficiently large $|s|$.
Using the asymptotic results of
\cite{hwz:prop1} and the characterization of the Conley--Zehnder index in terms
of the spectrum of $\A_{\gamma, J}$ from \cite{hwz:prop2}
(reviewed as Theorem \ref{t:spectral-conley-zehnder} above),
the following inequality for the asymptotic winding is deduced in \cite{hwz:prop2}.

\begin{theorem}\label{t:wind-infinity-bound}
Let $C=[\Sigma, j, \Gamma, a, u]\in\M(\lambda, J)$ and let $z\in\Gamma$.
Then
\begin{equation}\label{e:wind-infinity-inequality}
\pm_{z}\winfty^{\Phi}(u; z)\le \fl{\pm_{z}\mu^{\Phi}(\tl u; z)/2}
\end{equation}
where $\pm_{z}$ is the sign of the puncture $z$.
\end{theorem}

\subsection{Compactness}\label{s:compactness}
It is shown in \cite{behwz} that the space of punctured
pseudoholomorphic curves
with energy below any given value
can be compactified
by including more general objects, known as pseudoholomorphic buildings.
In \cite{wendl2010-compactness}, it's shown
that the space of curves which project to embeddings in the $3$-manifold $M$
can be compactified by considering only those buildings
whose components are 
either
pairwise disjoint or identical when projected to the $3$-manifold $M$ and are all
either trivial cylinders or project to embeddings in $M$.  We recall the result here.

We start with some definitions.
First, for $i\in \br{1, \dots, k}$, consider a collection of (possibly disconnected)
pseudoholomorphic curves
$C_{i}=[\Sigma_{i}, j_{i}, \Gamma_{i}, a_{i}, u_{i}]\in\M(\lambda, J)$
and write $\Gamma_{i}=\Gamma^{+}_{i}\cup\Gamma^{-}_{i}$
to indicate the decomposition into positive and negative punctures.
Assume there are bijections
$I_{i}:\Gamma^{-}_{i}\to\Gamma^{+}_{i+1}$ between the negative punctures
of one curve and the positive punctures of the next in the sequence.
We say that the data $(C_{1}, \dots, C_{k}; I_{1}, \dots, I_{k-1})$
form a \emph{height-$k$ non-nodal pseudoholomorphic building}
when pairs of punctures identified via the bijections $I_{i}$ have the same asymptotic limit.
We will denote such a height-$k$ pseudoholomorphic building by
\[
C_{1}\odot_{I_{1}} \dots \odot _{I_{k-1}}C_{k}
\]
or simply
\[
C_{1}\odot\dots \odot C_{k}
\]
when the specific bijections are not important, and we will refer to the curves
$C_{i}$ as the \emph{levels} of the building.
Given a height-$k$ pseudoholomorphic building $C_{1}\odot_{I_{1}} \dots \odot_{I_{k-1}} C_{k}$
with $C_{i}=[\Sigma_{i}, j_{i}, \Gamma_{i}, a_{i}, u_{i}]$,
we can circle-compactify each of the domain surfaces $\Sigma_{i}\setminus\Gamma_{i}$
at the punctures and glue these compactified surfaces together 
along circles corresponding to punctures identified via the bijections $I_{i}$ to form a 
a topological surface with boundary.
Due to the asymptotic behavior of the curves, this identification can be done in
such a way that the maps $u_{i}$ extend to the circle-compactifications
and glue together to give 
a continuous map from the glued surface into $M$.
In the event that any of the levels are asymptotic to multiply covered periodic orbits,
the operation of gluing the circle-compactified surfaces is only
uniquely determined when further
choices, namely that of so-called asymptotic markers, are made.
The specifics won't be important here,
so we won't address this issue any further.

The structure of a non-nodal pseudoholomorphic building
$C_{1}\odot_{I_{1}} \dots \odot_{I_{k-1}} C_{k}$
can be encoded in a graph
with a vertex for each smooth connected component of the domains
of the levels $C_{i}$
and an edge for
each pair of punctures identified via the bijections $I_{i}$.
We will say that a non-nodal
pseudoholomorphic building is \emph{connected} if the corresponding graph is connected.
This is equivalent to requiring the surface obtained from circle-compactifying and gluing the
levels, as described in the previous paragraph, to be connected.
The \emph{arithmetic genus} of a connected pseudoholomorphic building is
the genus of the glued surface.  The arithmetic genus can be computed in terms of
the graph modeling the building by the formula
\[
g=\#E-\#V+\sum_{v_{i}\in V}g(v_{i})+1
\]
where $\#E$ is the number of edges, $\#V$ is the number of vertices,
and $g(v_{i})$ is the genus of a given smooth connected surface in the building corresponding
to the vertex $v_{i}$ (see \cite[Equation (6)]{behwz}).
In particular, a connected pseudoholomorphic building has arithmetic genus zero
precisely when each component has genus $0$ and $\#E=\#V-1$ or,
equivalently, precisely when each component has genus zero and
the modeling graph is a tree.

Following \cite{wendl2010-compactness},
we refer to a connected pseudoholomorphic curve
$C=[\Sigma, j, \Gamma, a, u]$ as a \emph{nicely-embedded pseudoholomorphic curve}
if the map $u:\Sigma\setminus\Gamma\to M$ is an embedding, that is, if the curve projects
to an embedding in the $3$-manifold $M$.
We will say that a non-nodal pseudoholomorphic building is \emph{nicely embedded}
if:\footnote{
The definition in
\cite{wendl2010-compactness} also includes a condition on some of the periodic orbits which
connect the levels, but this condition 
(in fact a slightly stronger condition)
is a consequence of the above two conditions.  See
Lemma \ref{l:bidirectional-orbits} below.
}
\begin{enumerate}
\item Each $C\in\M(\lambda, J)$ occurring as a connected component
of the building is either nicely embedded or a trivial cylinder (i.e.\ a curve of the form
$\R\times\gamma$ for some periodic orbit $\gamma$).

\item If $C$ and $D\in\M(\lambda, J)$ occur as connected components of the building,
the projections of $C$ and $D$ to $M$ are either identical or disjoint.
\end{enumerate}
We will call a nicely-embedded, non-nodal pseudoholomorphic building, \emph{stable}
if no level consists entirely of trivial cylinders.\footnote{
The general definition of stable from \cite{behwz} allows for levels
which contain only trivial cylinders or constant maps provided the domains
of these maps are stable curves, i.e.\  twice the genus plus the number of special points
(marked points and nodes) is greater than or equal to $3$.  Since
we only consider buildings with no nodes or marked points here, our simpler definition
is equivalent.
}
The following theorem is proved as the main theorem in \cite{wendl2010-compactness}

\begin{theorem}\label{t:spec-cpct}\cite[Theorem 1]{wendl2010-compactness}
Let $C_{k}\in\M(\lambda, J)$ be a sequence of nicely-embedded pseudoholomorphic curves
with uniformly bounded energy.  Then there is a subsequence which converges in the sense
of \cite{behwz} to a stable, nicely-embedded pseudoholomorphic building.
\end{theorem}

For our purposes, the complete definition of SFT-convergence from \cite{behwz} is not
necessary, but we will need the following facts which we state as a proposition.

\begin{proposition}\label{p:compactness-local-C-infinity}
Assume a sequence $C_{k}=[\Sigma_{k}, j_{k}, \Gamma_{k}, a_{k}, u_{k}]\in\M(\lambda, J)$
converges in the sense of \cite{behwz} to a non-nodal pseudoholomorphic building
$C_{\infty, 1}\odot \dots \odot C_{\infty, \ell}$
with
$C_{\infty, i}=[\Sigma_{\infty, i}, j_{\infty, i}, \Gamma_{\infty, i}, a_{\infty, i}, u_{\infty, i}]$.
Then:
\begin{enumerate}
\item  There exist  sequences of embeddings
$\psi_{k, i}:\Sigma_{\infty, i}\setminus\Gamma_{\infty, i}\to\Sigma_{k}\setminus\Gamma_{k}$
and sequences of constants $c_{k, i}$ so that
\[
a_{k}\circ\psi_{k, i}+c_{k, i}\to a_{\infty, i}
\text{ in } C^{\infty}_{loc}(\Sigma_{\infty, i}\setminus\Gamma_{\infty, i}, \R) 
\]
and
\[
u_{k}\circ\psi_{k, i}\to u_{\infty, i}
\text{ in } C^{\infty}_{loc}(\Sigma_{\infty, i}\setminus\Gamma_{\infty, i}, M).
\]

\item There exists a punctured surface $\Sigma_{\infty}\setminus\Gamma_{\infty}$
and a $k_{0}$ so that for all $k\ge k_{0}$,
$\Sigma_{k}\setminus\Gamma_{k}$ is diffeomorphic to
$\Sigma_{\infty}\setminus\Gamma_{\infty}$.  Moreover, there exists a sequence of
diffeomorphisms
$\psi_{k}:\Sigma_{\infty}\setminus\Gamma_{\infty}\to\Sigma_{k}\setminus\Gamma_{k}$
so that the maps $u_{k}\circ \psi_{k}$ converge in
$C^{0}(\Sigma_{\infty}\setminus\Gamma_{\infty}, M)$. 
\end{enumerate}

\end{proposition}

Finally, we will need to know what sorts of periodic orbits can appear in the SFT-limit
of sequences of nicely-embedded curves.  We start with a definition.
In the following definition, we will use the notation $\gamma^{m}$ to denote the
$m$-fold cover a periodic orbit $\gamma$.

\begin{definition}\label{d:bidirectional-limit}
Let $\gamma$ be a simply covered orbit and let $m_{+}$ and $m_{-}$ be positive integers.
We say that $(\gamma, m_{+}, m_{-})$ is a
\emph{bidirectional asymptotic limit} of a given non-nodal pseudoholomorphic building, if
there are (possibly identical) nontrivial components $C_{+}$, $C_{-}$ in the building
so that $\gamma^{m_{+}}$ is a positive asymptotic limit of
$C_{+}$ and $\gamma^{m_{-}}$ is a negative asymptotic limit of $C_{-}$.
\end{definition}
We remark that
nontrivial breaking orbits as defined in \cite{wendl2010-compactness}
always give rise to a bidirectional limit, but the converse is not true.

\begin{lemma}\label{l:bidirectional-orbits}
Let $(\gamma, m_{+}, m_{-})$ be a bidirectional limit
of a nicely-embedded pseudoholomorphic building.
Then either:
\begin{itemize}
\item $\gamma$ is even and $m_{+}=m_{-}=1$, or
\item $\gamma$ is odd, hyperbolic and $m_{+}=m_{-}=2$.
\end{itemize}
\end{lemma}

\begin{proof}
This is equivalent to Proposition 4.4 in \cite{wendl2010-compactness}
which references \cite{siefring2011} for proof.
While this result is easily deduced from facts in \cite{siefring2011},
this fact is not stated explicitly,
so we outline the proof here.

Let $C_{+}=[\Sigma_{+}, j_{+}, \Gamma_{+}, a_{+}, u_{+}]$ and
$C_{-}=[\Sigma_{-}, j_{-}, \Gamma_{-}, a_{-}, u_{-}]$ be nontrivial components of the building
so that $C_{+}$ has $\gamma^{m_{+}}$ as an asymptotic limit at a puncture
$z_{+}\in\Gamma_{+}$ and
$C_{-}$ has $\gamma^{m_{-}}$ as an asymptotic limit at a puncture $z_{-}\in\Gamma_{-}$.
The assumption that $C_{\pm}$ are nicely embedded and
have either identical or disjoint projections to $M$ imply
via, as appropriate,
either condition 2(c) in \cite[Theorem 2.4]{siefring2011}/Theorem \ref{t:no-isect}
or
condition 3(b) in \cite[Theorem 2.6]{siefring2011}/Theorem \ref{t:embedded-projection}
that\footnote{
The sign difference between the equations given here and those in
\cite{siefring2011} are due to a convention difference for computing
Conley--Zehnder indices and $\winfty$.  Here we compute both by always traversing
an orbit in the direction of the Reeb vector field, while in \cite{siefring2011}
both are computed by traversing the orbit in a direction determined by
the boundary of the $S^{1}$-compactified surface, which means negative asymptotic limits
are traversed in the direction opposite of the flow of the Reeb vector field.
}
\begin{equation}\label{e:wind-cz-pos-neg}
\frac{\winfty^{\Phi}(u_{-}; z_{-})}{m_{-}}
=
\frac{-\fl{-\mu^{\Phi}(\gamma^{m_{-}})/2}}{m_{-}}
=
\frac{\fl{\mu^{\Phi}(\gamma^{m_{+}})/2}}{m_{+}}
=
\frac{\winfty^{\Phi}(u_{+}; z_{+})}{m_{+}},
\end{equation}
while condition 4(c) of \cite[Theorem 2.6]{siefring2011}/Theorem \ref{t:embedded-projection}
tells us that
\begin{equation}\label{e:wind-gcd}
\gcd(m_{+}, \winfty^{\Phi}(u_{+}, z_{+}))=\gcd(m_{-}, \winfty^{\Phi}(u_{-}, z_{-}))=1
\end{equation}
for any trivialization $\Phi$ of $\xi|_{\gamma}$.
However, it's proved in \cite[Theorem 2.4]{siefring2011}
using the
iteration formulas for the Conley--Zehnder index
(Lemma \ref{l:cz-iteration})
that
\[
\frac{-\fl{-\mu^{\Phi}(\gamma^{m_{-}})/2}}{m_{-}}
=
\frac{\fl{\mu^{\Phi}(\gamma^{m_{+}})/2}}{m_{+}}
\]
if and only if
$\gamma^{m_{+}}$ and $\gamma^{m_{-}}$ are both even orbits.
This is equivalent to requiring either that
$\gamma$ is even, or $\gamma$ is odd hyperbolic and $m_{+}$ and $m_{-}$ are both even.
In either case,
we can use that the Conley--Zehnder index iterates linearly for
hyperbolic orbits (Lemma \ref{l:cz-iteration}).
In the case that $\gamma$ is even
we then have from \eqref{e:wind-cz-pos-neg} that
\begin{align*}
\gcd(m_{\pm}, \winfty^{\Phi}(u_{\pm}, z_{\pm}))
&=\gcd(m_{\pm}, \fl{\pm\mu^{\Phi}(\gamma^{m_{\pm}})/2}) \\
&=\gcd(m_{\pm}, \fl{\pm m_{\pm }\mu^{\Phi}(\gamma)/2}) \\
&=\gcd(m_{\pm}, m_{\pm }\mu^{\Phi}(\gamma)/2) \\
&=m_{\pm}\gcd(1, \mu^{\Phi}(\gamma)/2) \\
&=m_{\pm}
\end{align*}
so we must have $m_{\pm}=1$ for \eqref{e:wind-gcd} to hold.
On the other hand, if $\gamma$ is odd hyperbolic
and $m_{\pm}=2n_{\pm}$ are even,
we have from \eqref{e:wind-cz-pos-neg} that
\begin{align*}
\gcd(m_{\pm}, \winfty^{\Phi}(u_{\pm}, z_{\pm}))
&=\gcd(m_{\pm}, \fl{\pm\mu^{\Phi}(\gamma^{m_{\pm}})/2}) \\
&=\gcd(2n_{\pm}, \fl{\pm\mu^{\Phi}(\gamma^{2n_{\pm}})/2}) \\
&=\gcd(2n_{\pm}, \fl{\pm 2n_{\pm}\mu^{\Phi}(\gamma)/2}) \\
&=\gcd(2n_{\pm}, \fl{\pm n_{\pm}\mu^{\Phi}(\gamma)}) \\
&=\gcd(2n_{\pm}, n_{\pm}\mu^{\Phi}(\gamma)) \\
&=n_{\pm}\gcd(2, \mu^{\Phi}(\gamma)) \\
&=n_{\pm}
\end{align*}
where we've used the $\mu^{\Phi}(\gamma)$ is odd in the last line.
This lets us conclude that \eqref{e:wind-gcd} holds precisely when $n_{\pm}=1$
and hence $m_{\pm}=2n_{\pm}=2$.
This completes the proof.
\end{proof}

\subsection{Intersection theory}\label{s:intersections}
Here we review some facts about the intersection 
theory of punctured pseudoholomorphic curves from \cite{siefring2011}.

We continue to assume $\lambda$ is a nondegenerate contact form on $M$
and $J\in\J(M, \lambda)$ is a compatible almost complex structure.
Let $C_{1}=[\Sigma, j, \Gamma, a, u]$ and $C_{2}=[\Sigma', j', \Gamma', b, v]$ be
pseudoholomorphic curves.
We write $\Gamma=\Gamma_{+}\cup\Gamma_{-}$ and $\Gamma'=\Gamma'_{+}\cup\Gamma'_{-}$
to indicate the signs of the punctures.
We assume that at
$z\in\Gamma$, $\tl u=(a, u)$ is asymptotic to $\gamma_{z}^{m_{z}}$
where $\gamma_{z}$ is a simply-covered, unparametrized periodic orbit, $m_{z}$ is a positive integer,
and $\gamma_{z}^{m_{z}}$ denotes the $m_{z}$-fold cover of $\gamma_{z}$.
Similarly we assume that
at $w\in\Gamma'$, $\tl v=(b, v)$ is asymptotic to $\gamma_{w}^{m_{w}}$ with $\gamma_{w}$ simply covered.
We let $\Phi$ denote a choice of trivialization of the contact structure
along all simply covered periodic orbits of $\X$
with covers appearing as asymptotic limits of $C_{1}$ or $C_{2}$.
We define a map
$\tl v_{\Phi}$ by perturbing the $M$-portion $v$ of the map slightly near the
ends by the flow of a section of the contact structure defined near the orbits
which has zero winding relative to the trivialization $\Phi$.  It can be shown
that for suitably small such perturbations, the algebraic intersection number $\inum(\tl u, \tl v_{\Phi})$
of the maps $\tl u$ and $\tl v_{\Phi}$ is well defined and depends only on the homotopy
classes of the maps $\tl u$, $\tl v$ and the trivialization $\Phi$.
We thus define the
\emph{relative intersection number} $i^{\Phi}(C_{1}, C_{2})$
of $C_{1}$ and $C_{2}$ relative to the trivialization $\Phi$
by
\[
i^{\Phi}(C_{1}, C_{2})=\inum(\tl u, \tl v_{\Phi}).
\]
For more background on the definition and properties of the relative intersection number,
see \cite[Section 2.4]{hutchings2002} or \cite[Section 4.1.1]{siefring2011}

Given the relative intersection number of two curves, we define the
\emph{holomorphic intersection number}\footnote{
This is called the \emph{generalized intersection number} in \cite{siefring2011}.
}
by\footnote{
This definition appears slightly different from that given in \cite{siefring2011}
since there Conley--Zehnder indices
of orbits at negative punctures are computed by traversing the orbit backwards.
}
\begin{equation}\label{e:gin-def}
\begin{aligned}
C_{1}*C_{2}
=i^{\Phi}(C_{1}, C_{2})
&+\sum_{\stackrel{(z, w)\in\Gamma_{+}\times\Gamma_{+}'}{\gamma_{z}=\gamma_{w}}}
m_{z}m_{w}\max\br{\tfrac{\fl{\mu^{\Phi}(\gamma_{z}^{m_{z}})/2}}{m_{z}}, \tfrac{\fl{\mu^{\Phi}(\gamma_{w}^{m_{w}})/2}}{m_{w}}}\\
&+\sum_{\stackrel{(z, w)\in\Gamma_{-}\times\Gamma_{-}'}{\gamma_{z}=\gamma_{w}}}
m_{z}m_{w}\max\br{\tfrac{\fl{-\mu^{\Phi}(\gamma_{z}^{m_{z}})/2}}{m_{z}}, \tfrac{\fl{-\mu^{\Phi}(\gamma_{w}^{m_{w}})/2}}{m_{w}}}.
\end{aligned}
\end{equation}
We note that the sums here are taken over all pairs of ends with the same sign which are
asymptotic to coverings of the same underlying simply covered orbit;
the quantities in these sums correspond to the negation of the minimum number of
intersections that must appear between a pair of such ends when one is perturbed in the
prescribed direction (see Section 3.2 and specifically Corollary 3.21 in \cite{siefring2011}). 
As our notation indicates, the holomorphic intersection product of two
curves is independent of the choice of trivialization used to define
the quantities on the right hand side of \eqref{e:gin-def}.
For proof of this fact and the other basic properties of the holomorphic intersection number
collected in the following theorem, we refer the reader to \cite{siefring2011}.

\begin{theorem}[Properties of the generalized intersection number]\label{t:hin-prop}
Let $(M, \lambda, J)$ be a nondegenerate contact manifold with compatible $J\in\J(M, \lambda)$,
and let $\M(\lambda, J)$ denote the
moduli space of finite-energy pseudoholomorphic curves in $M$.
\begin{enumerate}
\item  If $C=[\Sigma, j, \Gamma, \tl u]$, $D=[\Sigma', j', \Gamma', \tl v]\in\M(\lambda, J)$
are pseudoholomorphic curves
then the generalized intersection
number $C*D$ depends only on the relative homotopy classes of the maps $\tl u$ and $\tl v$.

\item For any $C$, $D\in\M(\lambda, J)$
\[
C*D=D*C.
\]

\item If $C_{1}$, $C_{2}$, $D\in\M(\lambda, J)$
then
\[
(C_{1}+C_{2})*D=C_{1}*D+C_{2}*D
\]
where ``$+$'' on the left hand side denotes the disjoint union of the curves $C_{1}$ and $C_{2}$.

\item
If $C_{1}\odot C_{2}$ and $D_{1}\odot D_{2}$ are asymptotically cylindrical buildings
then
\[
(C_{1}\odot C_{2})*(D_{1}\odot D_{2})\ge C_{1}*D_{1}+C_{2}*D_{2}.
\]
Moreover, strict inequality occurs if and only if there is a periodic orbit $\gamma$ so
that $C_{1}$ has a negative puncture asymptotic to $\gamma^m$, $D_{1}$
has a negative puncture asymptotic
to $\gamma^n$, and both $\gamma^m$ and $\gamma^n$ are odd orbits.
\end{enumerate}
\end{theorem}

One of the main motivations for the definition of the holomorphic intersection number
is that certain well-known theorems concerning the homological intersection number of
holomorphic curves generalize nicely to facts about the holomorphic intersection number,
albeit with an additional complication.
The first such theorem is a generalization of the fact that for a pair of closed
pseudoholomorphic curves
having no common components (i.e.\ no components having identical image)
the homological intersection number is nonnegative, and equal to zero if and only if the two curves
are disjoint.  For punctured curves a statement almost as strong can be made,
but we have to allow for the possibility that intersections disappear at the punctures
when the curves have ends approaching the same orbit.  In this case, the disappearance 
of intersections is traded for a higher degree of ``tangency at infinity'' with this notion being
made precise in terms of the asymptotic relative asymptotic formula from \cite{siefring2008} reviewed
above as Theorem \ref{t:relative-asymptotics}.  The total measure of ``tangency at infinity'' between
two curves $C$, $D$ without common components
is called the \emph{total asymptotic intersection number} and denoted
$\delta_{\infty}(C, D)$.  For a precise definition of the total asymptotic intersection number, and
for proof and further discussion, we refer the reader to \cite[Theorem 4.4/2.2]{siefring2011}.

\begin{theorem}\label{t:intpositivity}
Let $C$, $D\in\M(\lambda, J)$ 
be pseudoholomorphic curves.
If $C$ and $D$ have no common components then
\begin{equation}\label{e:intpositivity}
C*D=\inum(C, D)+\delta_\infty(C, D)
\end{equation}
where $\inum(C, D)$ is the algebraic intersection number of $C$ and $D$, and
$\delta_{\infty}(C, D)$ is the asymptotic intersection index of $C$ and $D$.
In particular
\[
C*D\ge\inum(C, D)\ge 0,
\]
and
\[
C*D=0
\]
if and only if $C$ and $D$ don't intersect, and the total asymptotic intersection index vanishes,
i.e.\
$\delta_\infty(C, D)=0$.
\end{theorem}

It will be of use to us here to be able to identify situations in which the holomorphic intersection
number of two curves vanishes.  A set of necessary and sufficient conditions is proved
is proved in \cite[Corollary 5.9]{siefring2011}.  We quote that result here with appropriate
adjustments to the notation and conventions.

\begin{theorem}\label{t:gin-zero}
Let $C=[\Sigma, j, \Gamma, \tl u=(a, u)]$
and $D=[\Sigma', j', \Gamma', \tl v=(b, v)]\in\M(\lambda, J)$
be pseudoholomorphic curves,
and assume that no component of $C$ or $D$ lies in a trivial cylinder.
Then the following are equivalent:
\begin{enumerate}
\item\label{i:gin-zero-1} The generalized intersection number $C*D=0$.

\item\label{i:gin-zero-2} All of the following hold:
     \begin{enumerate}
     \item The map $u$ does not intersect any of the positive asymptotic limits of $v$.
     \item The map $v$ does not intersect any of the negative asymptotic limits of $u$.
     \item Let $\gamma$ be a periodic orbit so that at $z\in\Gamma$, $\tl u$ is asymptotic
     to $\gamma^{m_z}$ and at $w\in\Gamma'$, $\tl v$ is asymptotic to $\gamma^{m_w}$,
     and let $\Phi$ be a trivialization of $\xi|_{\gamma}$.  Then:
          \begin{enumerate}
          \item If $z$ and $w$ are both positive punctures,
          \[
          \winfty^{\Phi}(\tl u; z)=\fl{\mu^{\Phi}(\gamma^{m_{z}})/2}
          \]
          and
          \begin{equation}\label{e:gin-zero-cz-ineq}
          \tfrac{\fl{\mu^{\Phi}(\gamma^{m_z})/2}}{m_z}\ge\tfrac{\fl{\mu^{\Phi}(\gamma^{m_w})/2}}{m_w}.
          \end{equation}
          \item If $z$ and $w$ are both negative are both negative punctures
          \[
          -\winfty^{\Phi}(\tl v; w)=\fl{-\mu^{\Phi}(\gamma^{m_{w}})/2}
          \]
          and
          \[
          \tfrac{\fl{-\mu^{\Phi}(\gamma^{m_w})/2}}{m_{w}}\ge\tfrac{\fl{-\mu^{\Phi}(\gamma^{m_z})/2}}{m_{z}}.
          \]
          \item\label{item:gin-zero-pos-neg} If $z$ is a negative puncture and $w$ is a positive puncture,
          \[
          -\winfty^{\Phi}(\tl u; z)-\fl{-\mu^{\Phi}(\gamma^{m_{z}})/2}=\winfty^{\Phi}(\tl v; w)-\fl{\mu^{\Phi}(\gamma^{m_{w}})/2}=0
          \]
          and
          $\gamma^{m_z}$ and $\gamma^{m_w}$ are both even orbits; or equivalently
          \[
          \tfrac{\winfty^{\Phi}(\tl u; z)}{m_{z}}=\tfrac{\winfty^{\Phi}(\tl v; w)}{m_{w}}.
          \]
     \end{enumerate}
     \end{enumerate}

\item\label{i:gin-zero-3} All of the following hold:
     \begin{enumerate}
     \item The map $u$ does not intersect any of the asymptotic limits of $v$.
     \item The map $v$ does not intersect any of the asymptotic limits of $u$.
     \item If $\gamma$ is a periodic orbit so that at $z\in\Gamma$, $\tl u$ is asymptotic
     to $\gamma^{m_z}$ and at $w\in\Gamma'$, $\tl v$ is asymptotic to $\gamma^{m_w}$, then
     \[
     \pm_{z}\winfty^{\Phi}(\tl u; z)-\fl{\pm_{z}\mu^{\Phi}(\gamma^{m_{z}})/2}=\pm_{w}\winfty^{\Phi}(\tl v; w)-\fl{\pm_{w}\mu^{\Phi}(\gamma^{m_{w}})/2}=0.
     \]
     Further
     \begin{enumerate}
          \item if $\gamma$ is elliptic, then $z$ and $w$ are either both positive punctures,
          or both negative punctures, and
          \[
          \tfrac{\fl{\pm_{z}\mu^{\Phi}(\gamma^{m_z})/2}}{m_z}=\tfrac{\fl{\pm_{w}\mu^{\Phi}(\gamma^{m_w})/2}}{m_w}.
          \]
          \item if $\gamma$ is odd, hyperbolic then either $m_z$ and $m_w$ are both even, or
          the punctures have the same sign and
          $m_z=m_w$.
     \end{enumerate}
\end{enumerate}
\end{enumerate}
\end{theorem}

It is observed in the discussion following Corollary 5.9 in \cite{siefring2011} 
that if holomorphic intersection number
of two connected curves vanishes, then the projections
of those curves to the $3$-manifold are either disjoint or identical.
Indeed, assume that the projections of $C$ and $D$ have neither identical nor disjoint image.
Then, arguing as in \cite{hwz:prop2}, an intersection point between the projections can be
seen as an intersection between one curve $C$ and an $\R$-shift of the other
$c\cdot D$.  Thus $C*(c\cdot D)>0$ by Theorem \ref{t:intpositivity}.  But
homotopy invariance of the $*$-product from Theorem \ref{t:hin-prop} then tells us the
$C*D=C*(c\cdot D)>0$, and we can conclude that $C*D=0$ if and only if the projections of
$C$ and $D$ to $M$ have either identical or disjoint image.
While the converse of this statement is not true,
a fairly complete set of necessary and sufficient conditions
for the projections of two curves to the $3$-manifold to not intersect is given in
\cite[Theorem 2.4/5.12]{siefring2011}.
We recall that theorem here with appropriate adjustments to the notation and conventions.

\begin{theorem}\label{t:no-isect}
Let $[\Sigma, j, \Gamma, \tl u=(a, u)]$ and $[\Sigma', j', \Gamma', \tl v=(b, v)]\in\M(\lambda, J)$
be pseudoholomorphic curves,
and assume that no component of $\tl u$ or $\tl v$ lies in a trivial cylinder,
and that the projected curves $u$ and $v$ do not have identical image on any component of their domains.
Then the following are equivalent:
\begin{enumerate}
\item The projected curves $u$ and $v$ do not intersect.

\item\label{i:no-isect-2} All of the following hold:
     \begin{enumerate}
     \item The map $u$ does not intersect any of the positive asymptotic limits of $v$.
     \item The map $v$ does not intersect any of the negative asymptotic limits of $u$.
     \item If $\gamma$ is a periodic orbit so that at $z\in\Gamma$, $\tl u$ is asymptotic
     to $\gamma^{m_z}$ and at $w\in\Gamma'$, $\tl v$ is asymptotic to $\gamma^{m_w}$, then:
          \begin{enumerate}
          \item If $z$ and $w$  are either both positive punctures or both negative punctures then
          \[
          \tfrac{\winfty(\tl u; z)}{m_z}\ge\tfrac{\winfty(\tl v; w)}{m_w}.
          \]
          \item If $z$ is a negative puncture and $w$ is a positive puncture then
          \[
          \tfrac{\winfty^\Phi(\tl u; z)}{m_z}=\tfrac{-\fl{-\mu^\Phi(\gamma^{m_z})/2}}{m_z}
          =\tfrac{\fl{\mu^\Phi(\gamma^{m_w})/2}}{m_w}=\tfrac{\winfty^\Phi(\tl v; w)}{m_w}
          \]
          (this is only possible if $\gamma^{m_z}$ and $\gamma^{m_w}$ are both even orbits).
    \end{enumerate}
     \end{enumerate}
\item\label{i:no-isect-3} All of the following hold:
     \begin{enumerate}
     \item The map $u$ does not intersect any of the asymptotic limits of $v$.
     \item The map $v$ does not intersect any of the asymptotic limits of $u$.
     \item If $\gamma$ is a periodic orbit so that at $z\in\Gamma$, $\tl u$ is asymptotic
     to $\gamma^{m_z}$ and at $w\in\Gamma'$, $\tl v$ is asymptotic to $\gamma^{m_w}$, then
          \[
          \tfrac{\winfty(\tl u; z)}{m_z} = \tfrac{\winfty(\tl v; w)}{m_w}.
          \]
     \end{enumerate}
\end{enumerate}
\end{theorem}

The following corollary will be of use in the proof of our main theorem.

\begin{corollary}\label{c:no-isect-gin-zero}
Let $C=[\Sigma, j, \Gamma, \tl u=(a, u)]$
and $D=[\Sigma', j', \Gamma', \tl v=(b, v)]\in\M(\lambda, J)$
be connected pseudoholomorphic curves.
Assume that $C$ or $D$ are not trivial cylinders, that
$C$ and $D$ have distinct projections to $M$, and that
at every puncture of $C$ and $D$ the winding bound from
\eqref{e:wind-infinity-inequality} is achieved.
Then $C*D=0$ if and only if the projections of $C$ and $D$ to $M$
are disjoint.
\end{corollary}

\begin{proof}
In the event that the winding bound from \eqref{e:wind-infinity-inequality} is achieved
at each puncture,
i.e.\ that
\[
\pm_{z}\winfty^{\Phi}(\tl u; z)=\fl{\pm_{z}\mu^{\Phi}(\gamma^{m_{z}})/2}
\]
for all $z\in\Gamma$ and
\[
\pm_{w}\winfty^{\Phi}(\tl v; w)=\fl{\pm_{w}\mu^{\Phi}(\gamma^{m_{w}})/2}
\]
for all $w\in\Gamma'$, then
condition \eqref{i:no-isect-2} in Theorem \ref{t:no-isect} and
condition \eqref{i:gin-zero-2} in Theorem \ref{t:gin-zero}
reduce to the same thing.
Thus the two theorems together imply that $C*D=0$ if and only if
the projection of $C$ and $D$ to $M$ are disjoint.
\end{proof}

In addition to Theorem \ref{t:intpositivity}, a second fact which motivates
the definition of the holomorphic
intersection number is that, like the homological self-intersection number of a closed curve,
the holomorphic self-intersection number identifies
those relative homotopy classes of simple curves
(i.e.\ those that don't factor through a branched cover) which
must be embedded \cite[Theorem 2.3]{siefring2011}.
This result can be combined with Theorem \ref{t:no-isect}
and some results
and techniques
from \cite{hwz:prop2} to state
a fairly exhaustive set of necessary and sufficient
conditions of the vanishing of the holomorphic self-intersection number of a curve.
We will
summarize the information we need from
this result below after recalling the definitions of some relevant invariants
associated to a punctured pseudoholomorphic curve.

Let $C=[\Sigma, j, \Gamma, a, u]\in\M(\lambda, J)$ be a pseudoholomorphic curve,
and assume that at $z\in\Gamma$ the map $u$ is asymptotic to a cover
of the periodic orbit $\gamma_{z}$.
Let $\Phi$ denote a trivialization of $\xi$ in a neighborhood of each periodic orbit $\gamma$
appearing as an asymptotic limit of $C$, and note that $\Phi$ induces a trivialization
of $u^{*}\xi$ in a neighborhood of each puncture.
We then define the \emph{total Conley--Zehnder index} of the curve
$C$ to be
\begin{equation}\label{e:total-cz-index}
\mu(C)=2c^{\Phi}_{1}(u^{*}\xi)+\sum_{z\in\Gamma^{+}}\mu^{\Phi}(u; z)-\sum_{z\in\Gamma^{-}}\mu^{\Phi}(u; z)
\end{equation}
where $c_{1}^{\Phi}(u^{*}\xi)$ is the relative first Chern number, defined to be
algebraic count of zeroes of a section of $u^{*}\xi$
which is nonzero and constant in the trivialization $\Phi$ in a neighborhood of each puncture
(see \cite[Section 2.2]{hutchings2002} or \cite[Section 4.2.1]{siefring2011} for more details
on the properties of the relative first chern number).
As a result of the respective change-of-trivialization formulas for the Conley--Zehnder index
and $c_{1}^{\Phi}(u^{*}\xi)$, it follows that the total Conley--Zehnder index of a curve is independent
of any choice of trivialization.
We then the define the \emph{index} $\ind(C)$ of $C$ by
\begin{equation}\label{e:index-formula}
\ind(C)=\mu(C)-\chi(\Sigma)+\#\Gamma
\end{equation}
where $\chi(\Sigma)$ is the Euler characteristic of $\Sigma$ and $\#\Gamma$ is the number of
puncture of the curve.  This index represents the Fredholm index of
the operator describing the local deformations of the curve $C$ 
(the relevant facts from \cite{hwz:prop3}
are reviewed in Section \ref{s:fredholm} below).

We now have the following theorem, summarizing relevant information from
\cite[Corollary 5.17]{siefring2011} and the discussion thereafter.

\begin{theorem}\label{t:self-gin-zero}
Let
$C=[\Sigma, j, \Gamma, \tl u=(a, u)]\in\M(\lambda, J)$
be a simple, connected pseudoholomorphic curve,
and assume that
$C*C=0$ and that
$C$ does not lie in a trivial cylinder.
Then:
\begin{enumerate}
\item The map $\tl u:\Sigma\setminus\Gamma\to\R\times M$ is an embedding.

\item The map $u:\Sigma\setminus\Gamma\to M$ is embedding, everywhere transverse to
Reeb flow, which is disjoint from all the asymptotic limits of $u$.

\item For each $z\in\Gamma$, the bound from \eqref{e:wind-infinity-inequality} is achieved, i.e.\
\[
\pm_{z}\winfty^{\Phi}(\tl u; z)= \fl{\pm_{z}\mu^{\Phi}(\tl u; z)/2}
\]
where $\pm_{z}$ denotes the sign of the puncture $z$.

\item The index $\ind(C)$ satisfies
\[
\ind(C)-\chi(\Sigma)+\#\Gamma_{even}=0
\]
where $\chi(\Sigma)$ is the Euler characteristic of the surface $\Sigma$ and $\#\Gamma_{even}$
is the number of punctures of $C$ asymptotic to even periodic orbits.
\end{enumerate}
\end{theorem}

We recall from the discussion following Corollary 5.17 in \cite{siefring2011}
that  for a connected curve $C=[\Sigma, j, \Gamma, a, u]$
satisfying the hypotheses of the previous result,
if $C*C=0$ then the projection of the curve to $M$ is an embedding.
Indeed, the result shows that $u$ must be an immersion which doesn't intersect
any of its asymptotic limits.
Moreover, since for the $\R$-translates  $c\cdot C=[\Sigma, j, \Gamma, a+c, u]$, we have
\[
0\le\inum(C, c\cdot C)\le C*(c\cdot C)=C*C=0,
\]
it follows from positivity of intersections
that $\tl u$ doesn't intersect any of its $\R$-translates, and hence that the projection $u$ is injective.
As observed in \cite{hwz:prop2}, the asymptotic behavior of
the curve then allows us to conclude that
$u$ is an embedding.
As with the discussion of intersections of curves with distinct projections to the three-manifold,
the converse is not true: it is possible for curves to project
to embeddings in $M$ but have positive self-intersection number.
However, various sets of necessary and sufficient conditions for the projection of a curve
of $M$ to be embedded are given in \cite[Theorem 2.6/5.20]{siefring2011}.
We quote that result
here, making appropriate adjustments to notation and conventions.

\begin{theorem}\label{t:embedded-projection}
Let $[\Sigma, j, \Gamma, \tl u=(a,u)]\in\M(\lambda, J)$
be a connected, simple pseudoholomorphic curve, and assume
that $\tl u$ does not have image contained in a trivial cylinder.
Then the following are equivalent:
\begin{enumerate}
\item The projected map $u:\Sigma\setminus\Gamma\to M$ is an embedding.

\item\label{i:embedded-projection-2} The algebraic intersection number $\inum(\tl u, \tl u_{c})$ between $\tl u$ and an $\R$-translate
     $\tl u_{c}=(a+c, u)$ is zero for all $c\in\R\setminus\br{0}$.

\item\label{i:embedded-projection-3} All of the following hold:
     \begin{enumerate}
          \item $u$ does not intersect any of its asymptotic limits.
          \item If $\gamma$ is a periodic orbit so that $u$ is asymptotic at $z\in\Gamma$ to $\gamma^{m_z}$
          and $u$ is asymptotic at $w\in\Gamma$ to $\gamma^{m_w}$, then
          \[
          \tfrac{\winfty(\tl u; z)}{m_z}=\tfrac{\winfty(\tl u; w)}{m_w}.
          \]
     \end{enumerate}

\item\label{i:embedded-projection-4} All of the following hold:
     \begin{enumerate}
          \item The map $\tl u$ is an embedding.
          \item The projected map $u$ is an immersion which is everywhere transverse to $\X$
          \item For each $z\in \Gamma$, we have
          \[
          \gcd(m_z, \winfty(\tl u; z))=1.
          \]
          \item If $\gamma$ is a simple periodic orbit so that
          $u$ is asymptotic at $z$ to $\gamma^{m_z}$,
          $u$ is asymptotic at $w\ne z$ to $\gamma^{m_w}$, and the punctures have the same signs,
          then the
          relative asymptotic intersection number
          (see Lemma 3.19 and discussion following in \cite{siefring2011} for definition)
          of the ends $[\tl u; z]$ and $[\tl u; w]$ satisfies
          \[
          \ain^\Phi([\tl u; z], [\tl u; w])=-m_{z}m_{w}\max\br{\tfrac{\pm_{z}\winfty^\Phi(\tl u; z)}{m_z}, \tfrac{\pm_{w}\winfty^\Phi(\tl u; w)}{m_w}}.
          \]
     \end{enumerate}
\end{enumerate}
\end{theorem}

The following corollary will be of use in the proof of our main theorem.

\begin{corollary}\label{c:embedding-gin-zero}
Let $C=[\Sigma, j, \Gamma, \tl u=(a, u)]\in\M(\lambda, J)$
be a connected pseudoholomorphic curve.
Assume $C$  is not a trivial cylinder and that
at every puncture of $C$ the winding bound from
\eqref{e:wind-infinity-inequality} is achieved.
Then $C*C=0$ if and only if the 
map $u:\Sigma\setminus\Gamma\to M$ is an embedding.
\end{corollary}

\begin{proof}
In the event that the winding bound from \eqref{e:wind-infinity-inequality} is achieved
at each puncture,
i.e.\ that
\[
\pm_{z}\winfty^{\Phi}(\tl u; z)=\fl{\pm_{z}\mu^{\Phi}(\gamma^{m_{z}})/2}
\]
for all $z\in\Gamma$ then
the special case of 
condition \eqref{i:gin-zero-2} in Theorem \ref{t:gin-zero} in which $\tl u=\tl v$
and
condition \eqref{i:embedded-projection-3} in Theorem \ref{t:embedded-projection}
reduce to the same thing.
Thus the two theorems together imply that $C*C=0$ if and only if
the map $u:\Sigma\setminus\Gamma\to M$ is an embedding.
\end{proof}

It is shown in \cite{siefring2011} that
when two curves have ends approaching coverings of the same
hyperbolic orbit, there is a ``direction-of-approach'' condition that will guarantee
the two curves have positive holomorphic intersection number.
We review the relevant definitions and results here.

We will first define what it means for two pseudoholomorphic half-cylinders to approach
an orbit in the same (or opposite) direction.  The definition we use here will apply to
any nondegenerate periodic orbit and, when the orbit is even, will be stricter than 
the definition used in \cite{siefring2011}.  In exchange for using a slightly stricter definition
we will be able to make a slightly stronger conclusion via essentially the same argument used in
\cite{siefring2011}.

Let $\tl u$, $\tl v:[R, \infty)\times S^1\to\R\times M$ be positive pseudoholomorphic
half-cylinders asymptotic to the same nondegenerate periodic orbit $\gamma$,
and assume that the 
the asymptotic formulas for
asymptotic representatives
$U$, $V$ of $\tl u$ and $\tl v$ respectively are given by
\begin{gather*}
U(s, t)=e^{\lambda_{u} s}[e_{u}(t)+r_{u}(s, t)] \\
V(s, t)=e^{\lambda_{v} s}[e_{v}(t)+r_{v}(s, t)].
\end{gather*}
We say that \emph{the half cylinders $\tl u$ and $\tl v$ approach
$\gamma$ in the same direction} if the eigenvectors $e_{u}$ and $e_{v}$ are positive scalar
multiples of each other, and similarly we say that 
we say that \emph{the half cylinders $\tl u$ and $\tl v$ approach
$\gamma$ in the opposite direction} if the eigenvectors $e_{u}$ and $e_{v}$ are negative
scalar multiples of each other.  We note that in either case we must have
$\lambda_{u}=\lambda_{v}$.
For a pair of negative half-cylinders asymptotic to the same periodic orbit, the definitions
are exactly analogous: the cylinders are said to approach in the same direction if the
eigenvectors controlling the approach are positive multiples of each other, and are said
to approach in the opposite direction if they are negative scalar multiples of each other.

The notion of approaching an orbit in the same (or opposite) direction
is of particular use
at an even orbit due to the following lemma, which shows that pairs of pseudoholomorphic ends
approaching an even orbit with the same sign and extremal winding
(i.e.\ the bound in \eqref{e:wind-infinity-inequality} is achieved)
always either approach 
in the same or opposite direction.

\begin{lemma}\label{l:same-or-opposite}
Let $\gamma$ be an even periodic orbit.
Let $\tl u$, $\tl v:[R, \infty)\times S^{1}\to\R\times M$ be
either both positive or both negative
pseudoholomorphic half-cylinders
asymptotic to $\gamma$, and assume that
\[
\winfty^{\Phi}(\tl u)=\winfty^{\Phi}(\tl v)=\mu^{\Phi}(\gamma)/2.
\]
for any symplectic trivialization $\Phi$ of $\gamma^{*}\xi$.  Then $\tl u$ and $\tl v$ either
approach $\gamma$ in the same direction, or opposite direction.
\end{lemma}

\begin{proof}
Let
$\lambda_{-}<0$ denote the largest negative eigenvalue of $\A_{\gamma, J}$.
Then, according to Theorem \ref{t:spectral-conley-zehnder}
and the fact that the parity of an orbit is equal to the parity of its Conley--Zehnder index,
in any symplectic trivialization
$\Phi$ of $\gamma^{*}\xi$ the winding of an eigenvector of $\A_{\gamma, J}$ with eigenvalue
$\lambda_{-}$ is $\fl{\mu^{\Phi}(\gamma)/2}=\mu^{\Phi}(\gamma)/2$.
It further follows from the same theorem and Lemma \ref{l:operator-spectrum}
that eigenvectors of $\A_{\gamma, J}$ having smallest possible positive eigenvalue
also have winding equal to $\mu^{\Phi}(\gamma)/2$.
Since Lemma \ref{l:operator-spectrum} tells us that the span of the collection of eigenvectors
having winding equal to $\mu^{\Phi}(\gamma)/2$ is two dimensional,
we can conclude that
\[
\dim\ker(\A_{\gamma, J}-\lambda_{\pm})=1
\]
and that all eigenvectors $e$ of $\A_{\gamma, J}$ with negative (resp.\ positive)
eigenvalue
satisfying
\[
\wind\Phi^{-1}e=\mu^{\Phi}(\gamma)/2
\]
have eigenvalue $\lambda_{-}$ (resp.\ $\lambda_{+}$).

Now, if $\tl u$ and $\tl v$ are positive (resp.\ negative) half-cylinders
and have winding
\[
\winfty^{\Phi}(\tl u)=\winfty^{\Phi}(\tl v)=\mu^{\Phi}(\gamma)/2.
\]
then the eigenvector controlling the approach of each cylinder must have eigenvalue
$\lambda_{-}$ (resp.\ $\lambda_{+}$).  Since we've argued above that the eigenspaces
$\ker(\A_{\gamma, J}-\lambda_{\pm})$ are $1$-dimensional,
we conclude that the eigenvectors associated to each cylinder are scalar multiples of each other
which is
equivalent to saying
that $\tl u$ and $\tl v$ approach $\gamma$ in either the same or opposite direction.
\end{proof}

We now state the main theorem concerning the intersection properties of pseudoholomorphic
half-cylinders that approach an even orbit with extremal winding in the same
direction.

\begin{theorem}(c.f.\ \cite[Theorem 5.15]{siefring2011})\label{t:dir-of-approach}
Let $\tl u=(a, u)$, $\tl v=(b, v):[R, \infty)\times S^{1}\to\R\times M$ be either
both positive or both negative pseudoholomorphic half-cylinders asymptotic
to an even periodic orbit $\gamma$.
Assume that $\tl u$ and $\tl v$ have extremal winding, i.e.\
\begin{equation}\label{e:winding-assumption}
\winfty^{\Phi}(\tl u)=\winfty^{\Phi}(\tl v)=\mu^{\Phi}(\gamma)/2,
\end{equation}
and that $\tl u$ and $\tl v$ approach $\gamma$ in the same direction.
Then the projections $u$, $v$ of the maps $\tl u$, $\tl v$ to the $3$-manifold $M$
intersect.
\end{theorem}

\begin{proof}[Proof 1]
The proof is a combination of the proofs of Theorem 5.14, Theorem 5.15,
Lemma 5.10 in \cite{siefring2011},
and a local version of Theorem 2.2 in \cite{siefring2011}/Theorem \ref{t:intpositivity} above.
If the images of
$\tl u$ and $\tl v$ differ by the $\R$-action on some neighborhood
of infinity, then the projections to $M$ will be identical on the same neighborhood of
infinity so there is nothing more to prove.
We thus, without loss of generality, assume that the images of $\tl u$ and $\tl v$ do not differ
by the $\R$-action on any neighborhood of infinity.
Given this assumption
Theorems 5.14-5.15 in \cite{siefring2011} shows that the there is a constant
$c_{0}\in\R$ so that the asymptotic intersection number $\delta_{\infty}(\tl u, \tl v_{c})$ is
positive, where $\tl v_{c}=(b+c, v)$ is half-cylinder.
On the other hand, the argument in Lemma 5.10 in \cite{siefring2011} shows that
$\delta_{\infty}(\tl u, \tl v_{c})=0$ for all $c\ne c_{0}$ nearby to $c_{0}$
(in fact, the proof there reveals, for all $c\ne c_{0}$).
Since intersections are isolated, we can, after perhaps restricting the domains, define 
an algebraic intersection number $\inum(\tl u, \tl v_{c_{0}})$, relative intersection number
$i^{\Phi}(\tl u, \tl v_{c_{0}})$, and holomorphic intersection number
$[\tl u]*[\tl v_{c_{0}}]=i^{\Phi}(\tl u, \tl v_{c_{0}})+\mu^{\Phi}(\gamma)/2$.
Moreover, $i^{\Phi}(\tl u, \tl v_{c_{0}})$ and $[\tl u]*[\tl v_{c_{0}}]$ will be invariant under small
perturbations of $\tl v_{c_{0}}$, and the analogy of formula \eqref{e:intpositivity}
\[
[\tl u]*[\tl v_{c_{0}}]=\inum(\tl u, \tl v_{c_{0}})+\delta_{\infty}(\tl u, \tl v_{c_{0}})
\]
holds as well for the localized versions of the intersection product.
As $c_{0}$ changes to a sufficiently nearby $c\ne c_{0}$ in the   equation, the local
holomorphic intersection product $[\tl u]*[\tl v_{c}]$ remains unchanged,
while we've just argued that the
asymptotic intersection number $\delta_{\infty}(\tl u, \tl v_{c})$ changes from a positive number
to zero.  Thus the algebraic intersection number $\inum(\tl u, \tl v_{c})$ must increase,
and since $\inum(\tl u, \tl v_{c})\ge 0$ for all $c$, we conclude that $\inum(\tl u, \tl v_{c})>0$ for
$c$ very near to $c_{0}$.  Since $\tl u=(a, u)$ and $\tl v_{c}=(b+c, v)$ intersecting implies
$u$ and $v$ intersect, this completes the proof.
\end{proof}

For the convenience of the reader we provide a self-contained presentation of the above
argument below.

\begin{proof}[Proof 2]
 For simplicity we will carry out the proof assuming that both cylinders are positive.
The proof in the case that both are negative is completely analogous.
As in the previous proof, we continue to assume that the images of
$\tl u$ and $\tl v$
not differ
by the $\R$-action on any neighborhood of infinity.

Proceeding now with the above assumptions, we let
$U$, $V:[R', \infty)\times S^{1}\to \gamma^{*}\xi$
be asymptotic representatives of $\tl u$ and $\tl v$ respectively, that is $U$ and $V$
satisfy
\begin{gather*}
\tl u\circ \phi(s, t)=\bp{Ts, \exp_{\gamma(t)}U(s, t)} \\
\tl v\circ \psi(s, t)=\bp{Ts, \exp_{\gamma(t)}V(s, t)} 
\end{gather*}
for some proper embeddings $\phi$, $\psi:[R', \infty)\times S^{1}\to [R, \infty)\times S^{1}$.
We further observe that if $\tl v_{c}$ is the map $\tl v_{c}(z)=(b(z)+c, v(z))$ obtained by shifting
the $\R$-component of $\tl v$ by $c$, then
\[
\tl v_{c}\circ \psi_{c}(s, t)=\bp{Ts, \exp_{\gamma(t)} V_{c}(s, t)},
\]
where $\psi_{c}(s, t):=\psi(s-c/T, t)$ and $V_{c}(s, t):=V(s-c/T, t)$.

According to our winding assumption \eqref{e:winding-assumption} and the fact
that $\gamma$ is an even orbit,
the asymptotic formulas for $U$ and $V$ must
be of the form
\begin{gather}
U(s, t)=e^{\lambda s}[e_{u}(t)+r_{1}(s, t)] \label{e:U-asymp-formula} \\ 
V(s, t)=e^{\lambda s}[e_{v}(t)+r_{2}(s, t)] 
\end{gather}
with $\lambda$ the largest negative eigenvalue of $\A_{\gamma, J}$, and
$r_{i}(s, t)\to 0$ exponentially in $s$.
Hence the asymptotic formula for $V_{c}=V(\cdot-c/T, \cdot)$ is of the form
\begin{equation}\label{e:Vc-asymp-formula}
\begin{aligned}
V_{c}(s, t)&=e^{\lambda s}e^{-\lambda c/T}[e_{v}(t)+r_{2}(s-c/T, t)] \\
&=e^{\lambda s}[K_{c}e_{v}(t)+r_{c}(s, t)] 
\end{aligned}
\end{equation}
where $K_{c}=e^{-\lambda c/T}>0$ and $r_{c}=K_{c}r_{2}(\cdot -c/T, \cdot)$
decays exponentially in $s$.

We seek to understand how the intersection behavior of two cylinders
$\tl u$ and $\tl v_{c}$ changes as $c$ changes.
We first observe that by the assumption that $\tl u$ and $\tl v$ approach
$\gamma$ in the same direction, there is a $c_{0}$ so that
$e_{u}=K_{c_{0}}e_{v}$.
If $\tl u$ and $\tl v_{c_{0}}$ intersect, the projections
$u$ and $v$ intersect, so there is nothing more to prove.  We assume then
that $\tl u$ and $\tl v_{c_{0}}$ don't intersect and consider the difference
$U(s, t)-V_{c_{0}}(s, t)$ for $(s, t)\in [R'', +\infty)$ for some large $R''$.
According to Theorem \ref{t:relative-asymptotics} we have that
\begin{equation}\label{e:U-Vc0-diff-form}
U(s, t)-V_{c_{0}}(s, t)=e^{\lambda_{1}s}[e_{1}(t)+r(s, t)]
\end{equation}
with $r(s, t)\to 0$ exponentially.
Meanwhile, direct computation using formulas \eqref{e:U-asymp-formula}
and \eqref{e:Vc-asymp-formula}
shows that
\[
U(s, t)-V_{c_{0}}(s, t)=e^{\lambda s}[r_{1}(s, t)-r_{c_{0}}(s, t)].
\]
Since $r_{1}-r_{c_{0}}$ converges exponentially to $0$,
comparing above two equations shows that $\lambda_{1}<\lambda$.
Since the orbit is even,
Lemma \ref{l:operator-spectrum} with Theorem \ref{t:spectral-conley-zehnder}
tells us that we
must have that $\wind \Phi^{-1}e_{1}<\wind\Phi^{-1}e_{u}=\mu^{\Phi}(\gamma)/2$
in any trivialization $\Phi$ of $\gamma^{*}\xi$.
We thus conclude from this observation and \eqref{e:U-Vc0-diff-form} that
\begin{equation}\label{e:U-Vc0-winding}
\wind\Phi^{-1}[U(s, \cdot)-V_{c_{0}}(s, \cdot)]=\wind\Phi^{-1}e_{1}<\mu^{\Phi}(\gamma)/2
\end{equation}
for all sufficiently large $s$.  Moreover,
since we assume that $\tl u$ and $\tl v_{c_{0}}$ don't intersect,
that \eqref{e:U-Vc0-winding} holds for all $s\in [R'', \infty)$.

Meanwhile, we can choose $c\ne c_{0}$ sufficiently close to
$c_{0}$ so that $U(s, t)-V_{c}(s, t)$ is defined for $(s, t)\in [R''+1, \infty)\times S^{1}$
and so that for all $s\in [R''+1, R''+2]$
\[
\wind\Phi^{-1}[U(s, \cdot)-V_{c}(s, \cdot)]
=\wind\Phi^{-1}[U(s, \cdot)-V_{c_{0}}(s, \cdot)]
<\mu^{\Phi}(\gamma)/2.
\]
On the other hand,
computation using \eqref{e:U-asymp-formula}, \eqref{e:Vc-asymp-formula}
$K_{c}=e^{-\lambda c/T}$, and $e_{u}=K_{c_{0}}e_{v}$ gives us
\begin{align*}
e^{-\lambda s}[U(s, t)-V_{c}(s, t)]
&=e_{u}(t)-K_{c}e_{v}(t)+r_{4}(s, t) \\
&=e_{u}(t)-e^{-\lambda (c-c_{0})/T}K_{c_{0}}e_{v}(t)+r_{4}(s, t) \\
&=e_{u}(t)-e^{-\lambda (c-c_{0})/T}e_{u}(t)+r_{4}(s, t)  \\
&=\sigma_{c_{0}}e_{u}(t)+r_{4}(s, t)
\end{align*}
where $r_{4}:=r_{1}-r_{c}$ decays exponentially in $s$
and $\sigma_{c_{0}}:=1-e^{-\lambda (c-c_{0})/T}\ne 0$.
We conclude that for all $s$ sufficiently large,
\[
\wind\Phi^{-1}[U(s, t)-V_{c}(s, t)]=\wind\Phi^{-1}e_{u}=\mu^{\Phi}(\gamma)/2.
\]
Since the winding of $\Phi^{-1}[U(s, t)-V_{c}(s, t)]$ changes as $s$ changes,
$U(s, t)-V_{c}(s, t)$ must have at least one zero.
This implies that
$\tl u$ and $\tl v_{c}$ must intersect at least once,
which in turn implies that the projections $u$ and $v$ intersect.
This completes the proof.
\end{proof}

Finally, we close this section with a result concerning intersections between
components of a holomorphic building resulting as the limit of a sequences of holomorphic curves
having intersection number equal to zero.

\begin{lemma}\label{l:building-isect}
Let $C_{k}$, $D_{k}\in\M(\lambda, J)$ be sequences of holomorphic curves
satisfying $C_{k}\ne D_{k}$ and
$C_{k}*D_{k}=0$ for all $k$,
and assume that
$C_{k}$ and $D_{k}$ converge (in the sense of \cite{behwz})
respectively to a holomorphic buildings $C_{\infty}$, $D_{\infty}$.
Then for every component $C'$ of $C_{\infty}$ 
and $D'$ of $D_{\infty}$,
the projections of $C'$ and $D'$ to $M$ are either disjoint or identical.
\end{lemma}

\begin{proof}
Assume that the projections of
some components
$C'$ and $D'$  of the limit buildings to $M$ are neither disjoint
nor identical.
Then for some value of $d'\in\R$,
$C'$ and the $\R$-translate $d'\cdot D'$ have at least one isolated intersection.
But according to the definition of SFT-convergence from \cite{behwz}
there exist sequences of constants $c_{k}$, $d_{k}\in\R$
so that $c_{k}\cdot C_{k}$ and $d_{k}\cdot D_{k}$ converge
respectively in $C^{\infty}_{loc}$ to $C'$ and $D'$
(see Proposition \ref{p:compactness-local-C-infinity}),
and thus
$(d_{k}d')\cdot D_{k}$ converges in $C^{\infty}_{loc}$ to $d'\cdot D$.
But, since $C'$ and $D'$ have at least one isolated intersection,
we can conclude from the $C^{\infty}_{loc}$ convergence that
$c_{k}\cdot C_{k}$ and $(d_{k}d')\cdot D_{k}$ have at least one isolated intersection 
for sufficiently large values of $k$.
Theorem \ref{t:intpositivity} then allows us to conclude that
$(c_{k}\cdot C_{k})*(d_{k}d'\cdot D_{k})>0$ for sufficiently large $k$.  But
by homotopy invariance of the $*$-product, we have that
\[
(c_{k}\cdot C_{k})*(d_{k}d'\cdot D_{k})=C_{k}*D_{k}=0.
\]
This contradiction completes the proof.
\end{proof}

\subsection{Fredholm theory and transversality}\label{s:fredholm}
We will briefly review the Fredholm theory for embedded (or immersed)
pseudoholomorphic curves from
\cite{hwz:prop3}.

First, given $(M, \lambda)$ and compatible $J\in \J(M, \lambda)$,
we define a metric $g_{J}$ on $M$ by
\[
g_{J}(v, w)=\lambda(v)\lambda(w)+d\lambda\bp{\pi_{\xi}(v), J\pi_{\xi}(w)}
\]
where $\pi_{\xi}:TM\approx\R\X\oplus\xi\to\xi$ is the projection to $\xi$ along
$\X$.  That $g_{J}$ defined this way is a metric on $TM$ follows from the
definition of compatibility of $J$.
We extend this to a metric $\tl g_{J}$ on $\R\times M$ by defining
\[
\tl g_{J}\bp{(h, v), (k, w)}=h\cdot k+g_{J}(v, w)
\]
where $(h, v)$, $(k, w)\in\R\oplus TM\approx T(\R\times M)$.
Compatibility of $J$ with $(\xi, d\lambda)$ and the definition of the extension
of $J$ to an almost complex structure $\tl J$ on $\R\times M$ implies that
$\tl g_{J}$ is a hermitian metric on the $(\R\times M, \tl J)$, that is
$\tl J$ is a $\tl g_{J}$-orthogonal endomorphism of $T(\R\times M)$.

Now, let $C=[\Sigma, j, \Gamma, \tl u=(a, u)]\in\M(\lambda, J)$ be
an embedded $\tl J$-holomorphic curve, and choose
a model parametrization $(\Sigma, j, \Gamma, \tl u=(a, u))$.
Then $C$ has a well-defined normal bundle $N_{C}$ which
can be realized as a subbundle of $T(\R\times M)|_{C}=\tl u^{*}T(\R\times M)$
by letting
\[
(N_{C})_{\tl u(z)}=d\tl u(z)(T_{z}\Sigma)^{\perp}
\]
with $\perp$ denoting the $\tl g_{J}$ orthogonal complement within $T_{\tl u(z)}(\R\times M)$.
We consider curves
which are parametrized by mapping
sections of the normal bundle $N_{C}$ of the curve $C$
to $\R\times M$ via the exponential map $\wt\exp$ of the metric $\tl g_{J}$, that is,
those curves $C'=[\Sigma', j', \Gamma', \tl v=(b, v)]\in\M(\lambda, J)$
for which there exists a smooth map $\psi:\Sigma\setminus\Gamma\to\Sigma'\setminus\Gamma'$
and a smooth section $V$ of $N_{C}$ so that
\[
\tl v(\psi(z))=\wt\exp_{\tl u(z)}V(z).
\]

In order to to do this, we first recall that the asymptotic behavior of the curve $C$ implies that
$\wt\exp$ is an immersion on some $\ep$-neighborhood
$N_{C}^{\ep}$
of the zero section of $N_{C}$ with respect to the metric on $N_{C}$ induced from $\tl g_{J}$
(see e.g.\ Corollary 2.7 in \cite{siefring2005}).
We can thus define an almost complex structure $\bar J$ on this $\ep$-neighborhood of the
zero section of $N_{C}$ by pulling back $\tl J$ via the exponential map.
Give a connection $\nabla$ on $N_{C}$, we get a splitting
\[
T_{(z, V)}N_{C}\approx T_{z}(\Sigma\setminus\Gamma)\oplus (N_{C})_{z}
\]
of the tangent space of $N_{C}$ into horizontal and vertical distributions.  This splitting
is canonical along the zero section.  With respect to the splitting induced by a given connection
we can write
\begin{equation}\label{e:bar-J}
\bar J(z, V)
=
\begin{bmatrix}
i(z, V) & \tl\Delta(z, V) \\ \Delta(z, V) & J(z, V)
\end{bmatrix}
\end{equation}
with
$i\in\End(T(\Sigma\setminus\Gamma))$,
$\tl\Delta\in\Hom(N_{C},T(\Sigma\setminus\Gamma))$,
$\Delta\in\Hom(T(\Sigma\setminus\Gamma), N_{C}) $, and
$J\in\End(N_{C})$.
Moreover, along the zero section of $N_{C}$ we have
\[
\bar J(z, 0)=
\begin{bmatrix}
i(z, 0) & \tl\Delta(z, 0) \\ \Delta(z, 0) & J(z, 0)
\end{bmatrix}
=
\begin{bmatrix}
j(z) & 0 \\ 0 & J_{N}(z)
\end{bmatrix}
\]
with $j$ the complex structure on $T\Sigma$ and $J_{N}$ the complex structure on
$N_{C}$ induced from $\tl J$.
Note that squaring \eqref{e:bar-J} and using $\bar J^{2}=-I$ we get that
\[
\Delta\circ i=-J\circ \Delta.
\]
Letting $\Delta':N_{C}^{\ep}\to \Hom(N_{C}, \Hom(T(\Sigma\setminus\Gamma), N_{C}))$ denote the
map obtained from differentiating $\Delta$ in the fiber direction, we can differentiate the
above equation in the fiber direction and use that $\Delta$ vanishes along the zero section
to conclude that
\[
[\Delta'(0)V]\circ i(0)=-J(0)\circ[\Delta'(0)V]
\]
or equivalently
\[
[\Delta'(0)V]\circ j=-J_{N}\circ[\Delta'(0)V]
\]
for any section $V$ of $N_{C}$.
Thus, for any section $V$ of $N_{C}$,
$[\Delta'(0)V]\circ j$
is a $j$-$J_{N}$ anti-linear map from $T(\Sigma\setminus\Gamma)$ to $N_{C}$.

\begin{definition}
The linearized normal $\dbar$-operator $\ndbar{C}$ at an
embedded curve $C\in\M(\lambda, J)$ relative to the connection $\nabla$ on
$N_{C}$ is the operator
$\ndbar{C}:C^{\infty}(N_{C})\to
C^{\infty}(\Hom^{0, 1}(T(\Sigma\setminus\Gamma), N_{C}))$
defined by
\[
\ndbar{C}V=\nabla V+J_{N}\nabla_{j\cdot }V+[\Delta'(0)V]\circ j.
\]
\end{definition}

The following theorem summarizes results
about the linearized normal $\dbar$-operator proved in \cite{hwz:prop3} using
results from \cite{schwarz1995}.

\begin{theorem}
There exists a measure and metric on $\Sigma\setminus\Gamma$ and connection
on $N_{C}$ so that the extensions of $\ndbar{C}$ to maps
\[
\ndbar{C}:W^{k, p}(N_{C})\to W^{k-1, p}(\Hom^{0, 1}(T(\Sigma\setminus\Gamma), N_{C}))
\]
and
\[
\ndbar{C}:C_{0}^{k, \alpha}(N_{C})\to C_{0}^{k-1, \alpha}(\Hom^{0, 1}(T(\Sigma\setminus\Gamma), N_{C}))
\]
are Fredholm.  Moreover
each of the above operators has the same kernel, and
the Fredholm index $\ind(\ndbar{C})$
of each of the above operators is given by
\[
\ind(\ndbar{C})
=\ind(C)
\]
with $\ind(C)$ as defined in \eqref{e:index-formula}.
\end{theorem}

In \cite{hwz:prop3}, it is shown that the moduli space of pseudoholomorphic curves
near a given embedded $C\in\M(\lambda, J)$ can be given
as the zero set of a smooth, nonlinear section
$H:\mathcal{B}\to\mathcal{E}$ of a Banach space bundle
$\mathcal{E}$ defined over an open neighborhood $\mathcal{B}$
of $0$ in the Banach algebra
$C_{0}^{k, \alpha}(N_{C})$
of $C_{0}^{k, \alpha}$
sections of the normal bundle $N_{C}$ of $C$.
Moreover, if $\mathbf{0}\in\mathcal{B}$ denotes the zero
section of $N_{C}$, there is a natural isomorphism
\[
\bm{\alpha}:C_{0}^{k-1, \alpha}(\Hom^{0, 1}(T(\Sigma\setminus\Gamma), N_{C}))
\to\mathcal{E}_{\mathbf{0}}
\]
so that the linearization
$H'(\mathbf{0})$
of the section $H$ at the zero section $\mathbf{0}\in\mathcal{B}$ satisfies
\[
H'(\mathbf{0})V=\bm{\alpha}(\ndbar{C}V)
\]
for any $V\in C_{0}^{k, \alpha}(N_{C})$.
Thus, in the case that $\ndbar{C}$ is surjective, the implicit function
theorem can be applied to conclude that set of curves
near $C\in\M(\lambda, J)$ is a smooth manifold with dimension equal to the index
\eqref{e:index-formula}.
This leads to the following theorem, summarized from facts proved in  \cite{hwz:prop3}.

\begin{theorem}\label{t:fredholm-main}
Let $C\in\M(\lambda, J)$ be an embedded pseudoholomorphic curve
with parametrization $(\Sigma, j, \Gamma, \tl u=(a, u))$
and assume that
\[
\ndbar{C}:C_{0}^{k, \alpha}(N_{C})\to C_{0}^{k-1, \alpha}(\Hom^{0, 1}(T(\Sigma\setminus\Gamma), N_{C}))
\]
is surjective.
Then there exists an open neighborhood $B\subset\ker\ndbar{C}$ of the zero section of $N_{C}$
and a smooth embedding
$E:B\to C_{0}^{k, \alpha}(N_{C})$
mapping $0$ to the zero section satisfying:
\begin{enumerate}
\item For every $\tau\in B$, $E_{\tau}:\Sigma\setminus\Gamma\to N_{C}$ is smooth section
of the normal bundle to $C$.
\item The derivative $DE_{0}:\ker\ndbar{C}\to C_{0}^{k, \alpha}(N_{C})$
of $E:B\to C_{0}^{k, \alpha}(N_{C})$ at $0\in B$ is
the inclusion $\ker\ndbar{C}\hookrightarrow C_{0}^{k, \alpha}(N_{C})$,
i.e.\ $DE_{0}(v)=v$ for any $v\in\ker\ndbar{C}\subset C_{0}^{k, \alpha}(N_{C})$.

\item For each $\tau\in B$, there exists a distinct pseudoholomorphic curve
$C_{\tau}\in\M(\lambda, J)$
with parametrization
$(\Sigma_{\tau}, j_{\tau}, \Gamma_{\tau}, \tl v_{\tau}=(b_{\tau}, v_{\tau}))$
and a diffeomorphism\footnote{
We caution the reader that, in general, the continuous
extension of this diffeomorphism over the punctures
is not smooth.  We refer the reader to \cite{hwz:prop3} for more details.
}
$\psi_{\tau}:\Sigma\setminus\Gamma\to\Sigma_{\tau}\setminus\Gamma_{\tau}$
so that\footnote{
We note that in \cite{hwz:prop3},
rather than using the exponential map of the metric, a map
from the normal bundle of $C$ to $\R\times M$ is constructed by using a special trivialization
in a special coordinate system.
However, the essential point
for the results of \cite{hwz:prop3} to hold is that
one has a map from
a neighborhood of $0$ in $N_{C}$ to $\R\times M$ satisfying certain asymptotic conditions.
That the exponential map $\wt\exp$ of the metric $\tl g_{J}$ has the right properties
is easily seen from the asymptotic analysis in \cite{siefring2005}.
}
\[
\tl v_{\tau}\circ\psi_{\tau}=\wt\exp (E_{\tau}).
\]
\item The map
$F:B\times(\Sigma\setminus\Gamma)\to\R\times M$ defined by
\[
F(\tau, z)=\wt\exp_{\tl u(z)} E_{\tau}(z)
\]
is smooth.
\end{enumerate}
\end{theorem}

We remark that the last claim above is only proved in \cite{hwz:prop3} as part of a theorem
(Theorem 5.7) where
it's assumed that the original curve $C$ is a pseudoholomorphic plane
satisfying some additional properties.  The proof of that portion of theorem
however applies to any immersed curve $C$.  The key idea is that
the sections $E_{\tau}$ are smooth by elliptic regularity and that
the map
$\tau\mapsto E_{\tau}$ determines a smooth map
from $B\to C_{0}^{k, \alpha}(N_{C})$ for every positive integer $k$.

It's proven in \cite{hwz:prop3} that for a generic choice of $J\in\J(M, \lambda)$ the
linearized normal $\dbar$-operator $\ndbar{C}$ at any immersed
curve $C\in\M(\lambda, J)$
is surjective.
We will not state the precise result since it is not needed in our proof.
What is of interest here is the fact that under certain circumstances,
the surjectivity of the linearized normal Cauchy--Riemann operator $\ndbar{C}$
can be guaranteed provided that certain conditions on the topological invariants of the curve
$C$
are met.  Such so-called automatic transversality conditions were first described
Gromov in \cite{gromov1985}, with proofs in \cite{hofer-lizan-sikorav}
for compact curves (either without boundary or with totally real boundary conditions),
and very general results proven in \cite{wendl2010-automatic}
from which the following theorem can be deduced.

\begin{theorem}\label{t:automatic-transversality}
Let $C=[\Sigma, j, \Gamma, a, u]\in\M(\lambda, J)$ be immersed.   Then
the linearized normal $\dbar$-operator $\ndbar{C}$ at $C$ is surjective if
\begin{equation}\label{e:automatic-transversality}
\ind(C)\ge- \chi(\Sigma)+\#\Gamma_{even}+2=2g(\Sigma)+\#\Gamma_{even}
\end{equation}
where $\chi(\Sigma)$ is the Euler characteristic of the surface $\Sigma$, and
$\#\Gamma_{even}$ is the number of punctures of the curve which
limit to periodic orbits with even Conley--Zehnder index.
\end{theorem}

In the event that $\ind(C)$
is even and positive
there is a short proof of a special case of this result
(the essential case for our proof is $\ind(C)=2$, but we include the above result since
it's also of interest to know that index-$1$ curves in a stable foliation are regular).
We will recall the proof of this special case below since the proof is easy and uses a fact about
the zeros of elements of the kernel of the linearized normal $\dbar$-operator
that we will need later.
We state this fact in the following lemma.

\begin{lemma}\label{l:normal-zero-count}
Let $C=[\Sigma, j, \Gamma, a, u]\in\M(\lambda, J)$ be an embedded pseudoholomorphic curve
and let $V\in\ker\ndbar{C}$ be a nontrivial element of the kernel of the linearized normal Cauchy--Riemann
operator at $C$.  Then all zeroes of $V$ are isolated and have positive local index.
Moreover, if $i(V)$ denotes the total algebraic count of zeroes of $V$, then
\[
0\le i(V)\le \frac{1}{2}\bp{\ind(C)-\chi(C)+\#\Gamma_{even}}
\] 
with $\chi(C)$ the Euler characteristic of the curve, and $\#\Gamma_{even}$ the number of
asymptotic limits of the curve with even Conley--Zehnder index.
\end{lemma}

\begin{proof}
The proof is a straightforward generalization of arguments in
\cite[Proposition 5.6, Theorem 5.8]{hwz:prop2}, 
\cite[Theorem 2.11]{hwz:prop3}, and
\cite[Theorem 2.7]{hwz:foliations}.  We will highlight the main points.
As observed in \cite[Theorem 2.11]{hwz:prop3},
the fact that the zeroes of a nontrivial element of $\ker\ndbar{C}$ are isolated and have positive
local index follows from the similarity principle (see e.g.\ Appendix A.6 in \cite{hoferzehnder}).
As in \cite[Theorem 2.11]{hwz:prop3}, it can be argued that a nontrivial element
$V$ of $\ker\ndbar{C}$ satisfies an asymptotic formula of the same form as that given in
\cite{hwz:prop1, mora} or
Theorem \ref{t:relative-asymptotics} above.
Thus choosing a trivialization $\Phi$ of the contact structure
along the asymptotic limit $\gamma_{z}$
of a given puncture $z\in\Gamma$ and extending to a trivialization of the normal bundle
$N_{C}$ near the puncture, the section $V$ has a well-defined asymptotic winding number and
the argument of Theorem \ref{t:wind-infinity-bound} applies to show that
\[
\pm_{z}\wind^{\Phi}(V; z)\le \fl{\pm_{z}\mu^{\Phi}(\gamma_{z})/2}=\tfrac{1}{2}\bp{\pm_{z}\mu^{\Phi}(\gamma_{z})-p(\gamma_{z})}
\]
with $\pm_{z}$ the sign of the puncture $z$ and $p(z)$ is the parity of the orbit $\gamma_{z}$.
A straightforward zero-counting argument gives that
\[
i(V)=c_{1}^{\Phi}(N_{C})+\sum_{z\in\Gamma}\pm_{z}\wind^{\Phi}(V; z).
\]
Meanwhile, properties of the relative first Chern number imply that
\[
c_{1}^{\Phi}(N_{C})-c_{1}^{\Phi}(\xi|_{C})=-\chi(\Sigma\setminus\Gamma)=-\chi(\Sigma)+\#\Gamma
\]
(see \cite[Proposition 3.1]{hutchings2002}).
Combining the above with
formulas \eqref{e:total-cz-index} and \eqref{e:index-formula} leads to
\[
i(V)\le \frac{1}{2}\bp{\ind(C)-\chi(C)+\#\Gamma_{even}}
\]
as claimed.
\end{proof}

We now recall the proof of
the special case of 
Theorem \ref{t:automatic-transversality}.
The idea is that if the kernel of the linearized operator is too big, then
one can construct a section of the kernel with too many zeroes.
This same argument is applied in the proofs of 
\cite[Theorem 2.11]{hwz:prop3} and
\cite[Theorem 2.7]{abb-cie-hof}.

\begin{theorem}\label{t:automatic-transversality-special}
Let $C=[S^{2}, i, \Gamma, a, u]\in\M(\lambda, J)$ be a
an embedded, pseudoholomorphic (punctured)
sphere and assume that all punctures of are odd and that
$\ind(C)=2$.  Then the linearized normal $\dbar$-operator
$\ndbar{C}$ is surjective.
\end{theorem}

\begin{proof}
To show that
$\ndbar{C}$ is surjective, it suffices to show that
\[
\dim\ker\ndbar{C}=\ind(C)=2.
\]
Suppose to the contrary that $\dim\ker\ndbar{C}>2$.  Then we can find
three linearly independent vectors
$V_{1}$, $V_{2}$, $V_{3}\in\ker\ndbar{C}\subset C^{\infty}(N_{C})$.
Choosing a point $z_{0}\in S^{2}\setminus\Gamma$ and using that the normal
bundle $N_{C}$ has (real) dimension $2$,
we can find constants $c_{1}$, $c_{2}$, and $c_{3}$ so that
$\sum_{i=1}^{3}c_{i}V_{i}(z_{0})=0$.
Thus, the section
$V_{\bf c}$ of $N_{C}$ defined by
$V_{\bf c}=\sum_{i=1}^{3}c_{i}V_{i}$
is a nonzero element of $\ker\ndbar{C}$ which vanishes at $z_{0}$ and therefore,
according to Lemma \ref{l:normal-zero-count} satisfies
$i(V_{\bf c})\ge 1$ since all zeroes have positive local index.
But Lemma \ref{l:normal-zero-count} also tells us that
\[
i(V_{\bf c})\le \frac{1}{2}\bp{\ind(C)-\chi(S^{2})+\#\Gamma_{even}}=\frac{1}{2}\bp{2-2+0}=0.
\]
We thus have the contradiction $1\le i(V_{\bf c})\le 0$ which completes the proof.
\end{proof}

\section{Stable finite energy foliations and moduli spaces of foliating curves}\label{s:foliating-curves}

In this section we will develop some general theory for finite energy foliations and
collect facts about 
the moduli spaces of curves which make up finite energy foliations.
We start with a definition.
\begin{definition}\label{d:fef}
Let $(M, \lambda, J)$ be a three manifold equipped with a nondegenerate contact form
and compatible $J\in\J(M, \lambda)$.  A \emph{stable finite energy foliation }
$\F$ of total energy $E_{0}$ for the data
$(M, \lambda, J)$ is a collection of
simple curves $C\in\M(\lambda, J)$ satisfying:
\begin{itemize}
\item For every point $p\in\R\times M$ there is a unique curve $C\in\F$ passing through $p$.
\item Every $C\in\F$ is either a trivial cylinder or satisfies $\ind(C)\in\br{1, 2}$
\item For any $C_{1}$, $C_{2}\in\F$ with $\ind(C_{i})\in\br{1, 2}$,
$C_{1}*C_{2}=0$.
\item $E_{0}=\sup_{C\in\F} E(C)$.
\end{itemize}
\end{definition}
We note that this definition is a slightly weaker one than that given in the introduction in that
we don't explicitly require here that the curves of $\F$ form a smooth foliation of $\R\times M$.
We will see below however, that this condition follows from the above assumptions and, thus, the
two definitions are in fact equivalent.
We observe that the last condition
in our definition of stable finite energy foliation above
applies when $C_{1}=C_{2}$.  
The following theorem collects some facts about the moduli spaces of curves satisfying
$C*C=0$ and $\ind(C)\in\br{1, 2}$
that follow from results reviewed in the preceding sections.

\begin{theorem}\label{t:foliating-curves}
Let $C=[\Sigma, j, \Gamma, a, u]\in\M(\lambda, J)$ be a simple
pseudoholomorphic curve, and assume that $C*C=0$ and
$\ind(C)\in\br{1, 2}$.  Then:
\begin{enumerate}
\item $C$ is embedded.
\item $C$ is nicely embedded;
that is, the projection of $C$ to $M$ is an embedding transverse to the Reeb flow
and doesn't intersect any of its asymptotic limits.
\item For each $z\in\Gamma$, the bound from \eqref{e:wind-infinity-inequality} is achieved, i.e.\
\[
\pm_{z}\winfty^{\Phi}(\tl u; z)= \fl{\pm_{z}\mu^{\Phi}(\tl u; z)/2}
\]
where $\pm_{z}$ denotes the sign of the puncture $z$.
\item The genus $g(\Sigma)$ of the domain is zero, i.e.\ $(\Sigma, j)$ is biholomorphic to the
Riemann sphere $(S^{2}, i)$. 
\item The number $\Gamma_{even}$ of punctures of $C$ asymptotic to even orbits is given
by
\[
\#\Gamma_{even}=2-\ind(C).
\]
\item The linearized normal Cauchy--Riemann operator $\ndbar{C}$ is surjective.
\item With $n=\ind(C)$, there exists an $\ep>0$ and an injective immersion
\[
\tl F_{C}:B_{\ep}^{n}(0)\times \Sigma\setminus\Gamma\to\R\times M
\]
so that the map $z\mapsto \tl F_{C}(0, z)$ is a parametrization of the curve $C$,
and so that for every $\tau\in B_{\ep}^{n}(0)$,
there is a pseudoholomorphic curve
\[
C_{\tau}=[\Sigma_{\tau}, j_{\tau}, \Gamma_{\tau}, \tl u_{\tau}=(a_{\tau}, u_{\tau})]\in\M(\lambda, J)
\]
and a diffeomorphism\footnote{
As with Theorem \ref{t:fredholm-main}, we again
caution the reader here
that the continuous extension of this diffeomorphism over the punctures
is not, in general, smooth.
}
\[
\psi_{\tau}:\Sigma\setminus\Gamma\to\Sigma_{\tau}\setminus\Gamma_{\tau}
\]
so that
\[
\tl F_{C}(\tau, \cdot)=\tl u_{\tau}\circ \psi_{\tau}.
\]
\end{enumerate}
\end{theorem}

\begin{proof}
The first three claims follow immediately from \cite[Corollary 5.17]{siefring2011}
(relevant portions are reviewed above in Theorem \ref{t:self-gin-zero}).
The fourth and fifth claims also follow from
\cite[Corollary 5.17]{siefring2011}/Theorem \ref{t:self-gin-zero}.
Indeed, we get from that result that
$C*C=0$ implies that
\[
\ind(C)-\chi(\Sigma)+\#\Gamma_{even}=0,
\]
which, if $\ind(C)\ge 1$, implies that
\[
\chi(\Sigma)\ge 1+\#\Gamma_{even}
\]
and thus we must have $\chi(\Sigma)=2-2g(\Sigma)=2$ or, equivalently, $g(\Sigma)=0$ establishing
the third claim.
Substituting $\chi(\Sigma)=2$ in the above then immediately yields the fifth claim.

Next, given $g(\Sigma)=0$ and $\#\Gamma_{even}=2-\ind(C)$, we have that
\begin{align*}
\ind(C)-2g(\Sigma)-\#\Gamma_{even}
&=\ind(C)-\#\Gamma_{even} \\
&=2(\ind(C)-1)
\end{align*}
which is greater than equal to zero provided $\ind(C)\ge 1$.
Thus
\[
\ind(C)\ge 2g(\Sigma)+\#\Gamma_{even}
\]
provided $\ind(C)\ge 1$ and Theorem \ref{t:automatic-transversality}
then allows us to conclude that
the linearized normal $\dbar$-operator is surjective.
(We note that in the $\ind(C)=2$ case
we have $\#\Gamma_{even}=2-\ind(C)=0$ so the special case,
Theorem \ref{t:automatic-transversality-special}, of the automatic transversality result holds.)

The final claim is a generalization of Theorem 5.7 in \cite{hwz:prop3},
and follows from Theorem \ref{t:fredholm-main}, (a generalization of)
Lemma \ref{l:normal-zero-count},
and the fact that $C*C=0$.  Indeed, since $\ndbar{C}$ is surjective,
Theorem \ref{t:fredholm-main} holds,
and we obtain a neighborhood $B$ of $0\in\ker\ndbar{C}$
and a smooth map
$F:B\times\Sigma\setminus\Gamma\to\R\times M$
so that each of the maps $z\mapsto F(\tau, z)$ parametrizes a
distinct pseudoholomorphic curve $C_{\tau}$ homotopic to $C$.
The assumption $C*C=0$ with homotopy invariance of the holomorphic intersection number
implies that $C_{\tau_{1}}*C_{\tau_{2}}=0$ for any $\tau_{1}$, $\tau_{2}\in B$.
Hence item (1) above along with Theorem \ref{t:intpositivity} imply that the $C_{\tau}$ form
a family of pairwise disjoint, embedded/nicely-embedded pseudoholomorphic curves.
This in turn implies that the map $F$ is injective since double points of $F$
can be seen as either intersections between two distinct $C_{\tau}$'s or
a self-intersection of some given $C_{\tau}$.
We next claim that $F$ is an immersion.
This argument proceeds essentially the same as in \cite[Theorem 5.7]{hwz:prop3}
which proves a similar result in the special case that the curve $C$ is a plane.
We explain the main points here.
Since $F$ is given by
\[
F(\tau, z)=\wt\exp_{\tl u(z)} E_{\tau}(z)
\]
and we have previously remarked that $\wt\exp$ is an immersion on some uniform
neighborhood of the zero section of $N_{C}$, it suffices to show that the map
$(\tau, z)\mapsto E_{\tau}(z)$ is an immersion.
Since $E_{\tau}$ is a smooth section of a vector bundle, it suffices in turn to show that
the fiber derivative
$D_{\tau}E_{\tau}(z)$ at any point $z$ has full rank.
Letting $\br{v_{i}}_{i=1}^{\ind(C)}$ be a basis for $\ker\ndbar{C}$, it suffices to show that
the sections $D_{\tau}E_{\tau}(z)v_{i}$ are pointwise linearly independent.
If not, then we could construct a nontrivial section $v$ of $N_{C}$ in
the image of $D_{\tau}E_{\tau}$
having a zero at some point.
However, it can be shown that sections in the image of $D_{\tau}E_{\tau}$
are in the kernel of a linear Fredholm operator $L_{\tau}$
of the same type
as $\ndbar{C}$.  In particular, the proof of Lemma \ref{l:normal-zero-count}
applies to elements of the kernel of $L_{\tau}$ and shows that
any nontrivial section $v$ of $N_{C}$
in the kernel of $L_{\tau}$
is nonvanishing since
$\ind(C)-\chi(C)+\#\Gamma_{even}=0$.  This contradiction completes the proof that,
for some sufficiently small neighborhood $B$ of $0\in\ker\ndbar{C}$,
$F$ is an injective immersion on $B\times \Sigma\setminus\Gamma$.
With $n=\ind(C)$, we choose a basis $\br{v_{i}}_{i=1}^{n}$ for $\ker\ndbar{C}$
and 
get a map $\tl F:B^{n}_{\ep}(0)\times\Sigma\setminus\Gamma\to\R\times M$
by defining $\tl F(c_{i}, z)=F(\sum_{i}c_{i}v_{i}, z)$,
which will be an injective immersion provided $\ep$ is small enough.
\end{proof}

We next prove a general lemma which says that
up to $\R$-translation all but finitely many curves in a stable finite energy foliation
have index $2$.

\begin{lemma}\label{l:finite-index-0-and-1}
Let $\F$ be a stable finite energy foliation for the data $(M, \lambda, J)$
(according to Definition \ref{d:fef}).  Then:
\begin{itemize}
\item $\F$ contains a finite number of trivial cylinders.

\item Up to $\R$-translation, $\F$ contains a finite number of curves
$C$ with $\ind(C)=1$.
\end{itemize}

\end{lemma}

\begin{proof}
First, consider a curve $C=[\Sigma, j, \Gamma_{+}\cup\Gamma_{-}, a, u]\in\M(\lambda, J)$
and assume that
at $z^{+}_{i}\in\Gamma_{+}$,  $u$ is asymptotic to an orbit with period $T_{i}^{+}$ and, similarly,
that at $z^{-}_{j}\in\Gamma_{-}$, $u$ is asymptotic to an orbit with period $T_{j}^{-}$.
Then the asymptotic behavior, the compatibility of $J$ with $d\lambda$,
and Stokes' Theorem can be used to show that :
\begin{itemize}
\item the energy $E(C)$ of the curve (defined by \eqref{e:energy-def} above) is
given by
\[
E(C)=\sum_{z^{+}_{i}\in\Gamma_{+}}T_{i}^{+}
\]
and
\item the $d\lambda$-energy $E_{d\lambda}(C)$, defined by
\begin{equation}\label{e:dlambda-energy-def}
E_{d\lambda}(C):=\int_{\Sigma\setminus\Gamma}u^{*}d\lambda,
\end{equation}
is nonnegative and
\[
E_{d\lambda}(C)=\sum_{z^{+}_{i}\in\Gamma_{+}}T_{i}^{+}-\sum_{z^{-}_{j}\in\Gamma_{-}}T_{j}^{-}.
\]
\end{itemize}
Thus, in a finite energy foliation, the period of any orbit appearing as an asymptotic limit of a curve
of the foliation is bounded above by the energy of the foliation.
Since we assume that $\lambda$ is nondegenerate,
one can then use Arzel\`{a}--Ascoli and the fact that nondegenerate orbits are isolated
to argue that there are only a finite number of unparametrized
periodic orbits of $X_{\lambda}$
having period less than any given positive number.
Thus a stable finite energy foliation can contain only a finite number of trivial cylinders.

To prove the second claim
we will argue by contradiction.  Suppose  that there are an infinite number of
index-$1$ curves in $\F$, each distinct up to $\R$-translation.  Then
we can find a sequence of curves $C_{k}\in\F$ with $\ind(C_{k})=1$
and so that no two of the $C_{k}$ differ by the $\R$-action.
Then, applying the main theorem of \cite{wendl2010-compactness}
(reviewed as Theorem \ref{t:spec-cpct} above),
we can pass to a subsequence, still denoted $C_{k}$, which converges
to a connected, nicely-embedded, non-nodal pseudoholomorphic building $C_{\infty}$
whose components have indices summing to $1$.  We will argue below that the limit building
$C_{\infty}$ is simply an embedded curve with $\ind(C_{\infty})=1$.  Once we know this,
the completeness property 
\cite[Theorem 7.1]{hwz:prop3}
implies that for sufficiently large $k$, the $C_{k}$ belong to the same
connected component of the moduli space as $C_{\infty}$ and thus differ by
an $\R$-shift.  This contradiction will complete the proof.

To argue that the building $C_{\infty}$
consists of just a single embedded curve, we first note that
Lemma \ref{l:building-isect} allows us to conclude that
all components of the building $C_{\infty}$ have image identical
to curves in $\F$ and further, since $C_{\infty}$ is a nicely-embedded building, that
all nontrivial components of $C_{\infty}$ are curves in $\F$ (as opposed to possibly being
multiple covers of such curves).
Since
$C_{\infty}$ consists of only trivial cylinders (which have index $0$)
and nontrivial curve of $\F$ (which have index at least $1$),
and since the indices of the components of $C_{\infty}$ must sum to $1$, 
we can conclude there is precisely one nontrivial component.
Moreover, since the building $C_{\infty}$ is connected, stable, and has 
no nodes, we can conclude that $C_{\infty}$
contains no trivial cylinders, and thus consists of just a single, embedded curve belonging to
$\F$.  This completes the proof.
\end{proof}

We now have the following corollary which shows that stable
finite energy foliations are indeed smooth foliations of $\R\times M$ which 
are invariant under the $\R$-action and project to $M$ to give smooth foliations
of the complement of a finite collection of periodic orbits in $M$.
Moreover, the projected leaves of the foliation are transverse to the Reeb flow.

\begin{corollary}\label{c:smooth-R-inv-foliation}
Let $\F$ be a stable finite energy foliation for the data $(M, \lambda, J)$
(according to Definition \ref{d:fef}).  Then:
\begin{enumerate}
\item If $C_{0}\in\F$ and $C_{1}\in\M(\lambda, J)$ is a simple\footnote{
The assumption that $C_{1}$ is also simple can be eliminated.  Indeed,
it can be shown that if $C_{0}$ is a simple curve with $C_{0}*C_{0}=0$ and
$\ind(C_{0})\in\br{1, 2}$ and $C_{1}$ is homotopic to $C_{0}$ then $C_{1}$ must
also be simple, but we will not need this here.
}
curve which is relatively homotopic
to $C_{0}$, then $C_{1}\in\F$.

\item The family of curves $\F$ is invariant under translation in the $\R$-coordinate,
i.e.\ if $C=[\Sigma, j, \Gamma, a, u]\in\F$
then $c\cdot C:=[\Sigma, j, \Gamma, a+c, u]\in\F$.

\item The curves in $\F$ form a smooth foliation of $\R\times M$.

\item There exists a finite collection $B$ of periodic orbits, so that the curves in $\F$
not fixed by the $\R$-action project to $M$ to form a smooth foliation of
$M\setminus B$ transverse to the flow.
\end{enumerate}
\end{corollary}

\begin{proof}
To prove the first statement, we will argue by contradiction.
Assume, to the contrary, that $C_{0}\in\F$, and that $C_{1}$ is a simple curve
relatively homotopic to $C_{0}$ with $C_{1}\notin\F$.
Then for any given point $p$ in the image of $C_{1}$ there is a simple curve
$C_{p}\in\F$ passing through $p$ and thus intersecting $C_{1}$.
Theorem \ref{t:intpositivity} then implies that $C_{1}*C_{p}\ge 1$.
However, since $C_{0}$ and $C_{p}$ are both curves in the family $\F$,
we have that $C_{0}*C_{p}=0$ by definition of stable finite energy foliation.
Thus the homotopy invariance of the intersection product
from Theorem \ref{t:hin-prop} gives us the contradiction
\[
1\le C_{1}*C_{p}=C_{0}*C_{p}=0.
\]
This completes the proof that if $C_{0}\in\F$, all simple curves relatively homotopic to $C_{0}$
are also in $\F$.
The second statement is then an immediate corollary of the first since any
curve in $\M(\lambda, J)$ is relatively homotopic to its $\R$-translates.

We next address the third claim above.
By Theorem \ref{t:foliating-curves} above, all curves in $\F$ are embeddings.
The fact that the curves of $\F$ form a smooth foliation of $\R\times M$
then follows from an argument similar to that in the
paragraphs following Lemma 6.10 in section 6.3 of \cite{hwz:foliations}.
We first observe that Lemma \ref{l:finite-index-0-and-1} tells us that the set of points
of $\R\times M$ with index-$2$ curves passing through them is open and dense.
For a point $p\in\R\times M$ with an index-$2$ curve $C\in\F$ passing through it,
it follows from the last item in Theorem \ref{t:foliating-curves}
that $C$ belongs to a smoothly varying $2$-dimensional family of 
pseudoholomorphic curves $C_{\tau}$ which foliate a neighborhood of $p$.
Moreover, it follows from the preceding
paragraph that each of the curves $C_{\tau}$ is in $\F$, and thus that curves
of $\F$ foliate some neighborhood of $p$.

Next, considering a point $p$ lying on an index-$1$ curve $C\in\F$, we've already observed
that all $\R$-translates of $C$ belong to $\F$.
If $p_{k}$ is a sequence of points converging to $p$ and not lying on an $\R$-translate of
$C$, we can conclude from Lemma \ref{l:finite-index-0-and-1} that for sufficiently large
$k$, $p_{k}$ lies on an index-$2$ curve $C_{k}\in\F$.
Moreover, by Theorem \ref{t:spec-cpct},
we can find a sequence of local parametrizations of some subsequence
$C_{k_{j}}$ which converge in $C^{\infty}_{loc}$ to a
parametrization of a curve $C_{\infty}$ passing through
$p$.
We claim that $C_{\infty}=C$.
Indeed, if $C_{\infty}$ doesn't have identical image with $C$, it must have an isolated intersection
with $C$.  This would then allow us to conclude that the $C_{k_{j}}$
intersect $C$ for sufficiently large $j$ and thus, by Theorem \ref{t:intpositivity}, that
$C_{k_{j}}*C\ge 1$.  This contradicts the fact that $C_{k_{j}}*C=0$ by the assumption that
$C$ and all $C_{k}$ are in the family $\F$.  We conclude that $C_{\infty}$ has the same
image as $C$ and further,
since  Theorem \ref{t:spec-cpct} tells us that $C_{\infty}$ must be either a trivial cylinder or
nicely embedded, that $C_{\infty}=C$.
This allows us to conclude that the curves of $\F$ smoothly foliate some neighborhood of $p$.
The argument for points lying one of the finitely-many
(according to Lemma \ref{l:finite-index-0-and-1}) trivial cylinders of $\F$ now proceeds along
similar lines with the use compactness and positivity of intersections.

We finally address the last claim.  We first define $B$ to be the collection
of periodic orbits which appear as asymptotic limits of curves in $\F$.
Then $B$ must be a finite set by Lemma \ref{l:finite-index-0-and-1} above.
By Theorem \ref{t:foliating-curves} every curve $C\in\F$ that is not a trivial cylinder
projects to an embedding transverse to the flow and disjoint from $B$.
Moreover, it follows from the assumption that $C_{1}*C_{2}=0$ for any two nontrivial 
curves $C_{1}$, $C_{2}\in\F$ that the projections of $C_{1}$ and $C_{2}$
to $M$ have either disjoint or identical images
(see e.g.\ the discussion following Corollary 5.9 in \cite{siefring2011}).
Therefore, we have a unique embedded curve through every point of
$M\setminus B$.  Moreover, since the curves of $\F$ form a smooth foliation of $\R\times M$
invariant under $\R$-shifting, the projections of these curves to $\R\times M$ will form a smooth
foliation of $M\setminus B$ provided the pullback of the coordinate field $\partial_{a}$ on $\R$
to $\R\times M$ is not tangent to any of the curves.  Since, as a result of the
definition of $\tl J$, such a tangency can be identified with
tangency of the projected curve to the Reeb vector field, there can be no such tangencies.
This completes the proof.
\end{proof}

We can also prove a converse to last part of the above result; specifically, 
the next result shows
that as an alternate definition of stable finite energy foliation, one can consider
the projections of curves to $M$ which foliate the complement of a finite collection of periodic
orbits.

\begin{corollary}\label{c:fol-alt-def}
Let $B\subset M$ be a finite collection of simple periodic orbits, and let
$\F\subset\M(\lambda, J)/\R$ be a collection of simple curves $C\in\M(\lambda, J)/\R$
satisfying:
\begin{itemize}
\item Each $C\in\F$ is disjoint from $B$.
\item For each $p\in M\setminus B$ there is a
(not necessarily unique) curve $C\in\F$ passing through $p$.
\item $\ind(C)\in\br{1, 2}$ for all $C\in\F$.
\item $C_{1}*C_{2}=0$ for all $C_{1}$, $C_{2}\in\F$.
\item The energies of the curves in $\F$ are uniformly bounded; that is,
$E(\F):=\sup_{C\in\F}E(C)$ is finite.
\end{itemize}
Then the collection of curves $\tl\F$ in $\M(\lambda, J)$ consisting of all possible lifts 
of curves $C\in\F$ to curves in $\R\times M$ together with cylinders over the periodic orbits
in $B$ form a finite energy foliation.
\end{corollary}

\begin{proof}
Given a point $p\in M\setminus B$ there is, by assumption, a curve $C\in\F$
passing through it.
Considering the set of all possible lifts gives a curve through each point of
$\R\times (M\setminus B)$.  Moreover, by the assumption that the
holomorphic intersection
numbers between all such curves is zero, we indeed get a unique curve through each point
of $\R\times (M\setminus B)$
by Theorem \ref{t:intpositivity}.
Moreover, by the assumption that the curves of $\F$
are disjoint from $B$, we obtain a unique curve through
each point of $\R\times M$ by including the trivial cylinders over the orbits in $B$ in the collection
we consider.  The remaining properties of a finite energy foliation from
Definition \ref{d:fef} are then easily verified from our remaining assumptions.
\end{proof}

In the rest of this section we will focus on
the structure of moduli spaces of simple curves satisfying
$C*C=0$ and $\ind(C)=2$.  Because of the important role that such curves play in what
follows, it will be convenient to have a term for such curves.
\begin{definition}
A curve $C\in\M(\lambda, J)$ is said to be a
\emph{foliating curve} if
$C*C=0$ and $\ind(C)=2$.
\end{definition}
Given a curve $C\in\M(\lambda, J)$ we will use the notation
$\M(C)$ to indicate the moduli space of simple curves
in the same relative homotopy class as $C$ and
$\M_{1}(C)$ to indicate the moduli space of simple curves with one marked point in the same
relative homotopy class as $C$.
We note the results of \cite{hwz:prop3}, reviewed in
Section \ref{s:fredholm} above, give a local manifold structure on these spaces
in the event that linearized normal $\dbar$-operator is surjective.
However, the fact these local manifold structures glue together to give a 
global manifold structure on the moduli space is only addressed
in \cite{hwz:prop3} as a special case of the fact that the local models for the universal
moduli space glue together to give a global Banach manifold structure on the universal
moduli space.
In the event that the curves in question are foliating curves, a simpler argument
is possible using
Theorem \ref{t:foliating-curves} above.  We state this result as a corollary.

\begin{corollary}\label{c:manifold-structure}
Let $C\in\M(\lambda, J)$ be a foliating curve, that is, assume
that $C$ is simple, $C*C=0$ and $\ind(C)=2$.
Then $\M(C)$ has the structure of a smooth, $2$-dimensional manifold, and
$\M_{1}(C)$ has the structure of a smooth $4$-dimensional manifold.  Moreover,
the
evaluation map $ev:\M_{1}(C)\to\R\times M$
is a smooth embedding,
the  forgetful map $\M_{1}(C)\to \M(C)$  is a smooth submersion,
and the
action of $\R$-shifting a curve defines smooth, free, proper $\R$-actions
on $\M_{1}(C)$ and $\M(C)$.
\end{corollary}

\begin{proof}
Given $C_{i}=[\Sigma_{i}, j_{i}, \Gamma_{i}, a_{i}, u_{i}]\in\M(C)$ for $i\in\br{1, 2}$,
Theorem \ref{t:foliating-curves} gives a local identification of the
moduli space of curves with one marked point with
$B_{\ep_{i}}^{2}(0)\times\Sigma_{i}\setminus\Gamma_{i}$ together with
an embedding
$\tl F_{C_{i}}:B_{\ep_{i}}^{2}(0)\times\Sigma_{i}\setminus\Gamma_{i}\to\R\times M$.
Because the maps $\tl F_{C_{i}}$ are local diffeomorphisms,
maps of the form
${\tl F_{C_{2}}}^{-1}\circ \tl F_{C_{1}}$
are smooth when defined, and thus the local identifications of
$\M_{1}(C)$ with sets of the form $B_{\ep_{i}}^{2}(0)\times\Sigma_{i}\setminus\Gamma_{i}$ 
piece together to give a global manifold structure on $\M_{1}(C)$ in which the evaluation map,
being locally given by the $\tl F_{C}$-maps,
are smooth immersions.
Moreover, since double points of $ev$ can be seen as intersections/self-intersections between
curves in $\M(C)$, the fact that $C*C=0$ implies that $ev$ is an injective map.

Next, we observe that
a local manifold structure on $\M(C)$ near the curve $C_{i}$ is given
by projecting
\[
\pi_{i}:B_{\ep_{i}}^{2}(0)\times\Sigma_{i}\setminus\Gamma_{i}\to B_{\ep_{i}}^{2}(0).
\]
Since the maps ${\tl F_{C_{2}}}^{-1}\circ \tl F_{C_{1}}$,
where defined, are smooth local diffeomorphisms which
restrict to diffeomorphisms on the fibers of the projections $\pi_{i}$,
a smooth local section $s_{1}$ for the projection $\pi_{1}$ is mapped
via ${\tl F_{C_{2}}}^{-1}\circ \tl F_{C_{1}}$ to a smooth local section for $\pi_{2}$,
and the composition
$\pi_{2}\circ {\tl F_{C_{2}}}^{-1}\circ \tl F_{C_{1}}\circ s_{1}$ is independent 
of the choice of smooth section $s_{1}$.
Such maps can then be used to construct smooth change-of-coordinate maps
giving a global manifold structure on $\M(C)$.
Moreover, since the forgetful map
$\M_{1}(C)\to\M(C)$ is given locally by one of the projections $\pi_{i}$ defined above,
the forgetful map is a smooth submersion in the manifold structure we've constructed.

To see that the $\R$-action
is a smooth, free, proper action on $\M_{1}(C)$ we first observe that the evaluation map
$ev:\M_{1}(C)\to\R\times M$ is $\R$-equivariant.  Since the $\R$-action on $\R\times M$ is
smooth, free, and proper and $ev$ is an embedding, it follows immediately that
the $\R$-action on $\M_{1}(C)$ is smooth, free, and proper.
Moreover, since the forgetful map
$\M_{1}(C)\to\M(C)$ is a smooth $\R$-equivariant submersion,
we can conclude that the $\R$ acts smoothly
on $\M(C)$ as well by considering smooth local sections
$\M(C)\to\M_{1}(C)$.
Freeness of the $\R$-action on $\M(C)$
follows from the well-known fact that
only trivial cylinders can be fixed points of the $\R$-action (or, in this case, from the
fact that $C*C=0$ implies that $C$ is disjoint from all of its nontrivial $\R$-translates).
Finally, properness follows from $\R$-equivariance of the forgetful map and
properness of the action on $\M_{1}(C)$.
\end{proof}

\begin{corollary}\label{c:mod-spaces-mod-r}
Let $C$ be a foliating curve.
Then the moduli space $\M_{1}(C)/\R$ is a smooth $3$-manifold and
the moduli space $\M(C)/\R$ is a smooth $1$-manifold.
Moreover,
$ev:\M_{1}(C)/\R\to M$ is an embedding,
and the forgetful map
$\M_{1}(C)/\R\to\M(C)/\R$ is a smooth submersion.
\end{corollary}

\begin{proof}
The facts that $\M_{1}(C)/\R$ is a smooth $3$-manifold and that
$\M(C)/\R$ is a smooth $1$-manifold follow directly from
Corollary \ref{c:manifold-structure}
since the $\R$-action on $\M_{1}(C)$ and $\M(C)$ is free and proper,
while the fact that the forgetful map
$\M_{1}(C)/\R\to\M(C)/\R$ is a smooth submersion
follows from the fact that the forgetful map
$\M_{1}(C)\to\M(C)$ is an $\R$-equivariant smooth submersion.
Finally, to see that $ev:\M_{1}(C)/\R\to M$ is an embedding,
we first observe that it follows from the fact that the evaluation map
$ev:\M_{1}(C)\to\R\times M$
is an $\R$-equivariant immersion
that $ev:\M_{1}(C)/\R\to M$ is also immersion.
Since $\M_{1}(C)/\R$ and $M$ are the same dimension, it remains to 
show that $ev:\M_{1}(C)/\R\to M$ is injective.
But since $C*C=0$, it follows
that the projections of distinct curves
in $\M(C)$ to $M$ 
are embedded and have disjoint
image unless they differ by the $\R$-action.
Since double points of $ev:\M_{1}(C)/\R\to M$ can be seen as
intersections/self-intersections of curves in $\M(C)/\R$,
we conclude that $ev:\M_{1}(C)/\R\to M$ in injective, and thus an embedding.
\end{proof}

For the following let $\psi_{t}$ denote the flow of the Reeb vector field.

\begin{corollary}\label{c:flow-loc-diff}
Let $C=[\Sigma, j, \Gamma, da, u]\in\M(\lambda, J)/\R$
be a foliating curve.
Then given any $p\in u(\Sigma\setminus\Gamma)$, there exists an
$\ep>0$ so that:
\begin{enumerate}
\item For every $t\in (-\ep, \ep)$ there exists a unique point of $\M(C)/\R$
passing through $\psi_{t}(p)$.

\item The map taking a point $t\in(-\ep, \ep)$ to the unique curve in $\M(C)/\R$
passing through $\psi_{t}(p)$
is a local diffeomorphism.
\end{enumerate}
\end{corollary}

\begin{proof}
The first claim follows from
Corollary \ref{c:mod-spaces-mod-r}.
Indeed, since the evaluation map $ev:\M_{1}(C)/\R\to M$ is an embedding, the image
of an open set around $(C, z)$
contains an open neighborhood $U$ of the point $p:=u(z)$. 
Thus, there exists some $\ep>0$ so that 
$\psi_{t}(p)\in U$ for all $t\in (-\ep, \ep)$,
which tells there us there is a point of $\M_{1}(C)/\R$ mapping via $ev$ to $p$, which is equivalent
to there being a curve in $\M(C)/\R$ passing through $p$.
Moreover, the fact that 
the evaluation map is injective implies that there is at most one curve in $\M(C)$
passing through any given point in $M$.

Next we show that the map
taking $t\in(-\ep, \ep)$ to the unique curve in $\M(C)/\R$ passing through $p$
is  a local diffeomorphism.
By construction, the map taking an interval $(-\ep, \ep)$ to $\M(C)/\R$ is given by the composition
\[
\begin{CD}
(-\ep, \ep) @>{\psi_{\cdot}(p)}>> M @>{ev^{-1}}>> \M_{1}(C)/\R @>>> \M(C)/\R
\end{CD}
\]
with the last map the forgetful map.  Since the composition
of the first two maps gives an embedding of $(-\ep, \ep)$ in $\M_{1}(C)/\R$ it suffices to show
that this embedding is transverse to the fibers of the forgetful map.
However, since the embedding $ev:\M_{1}(C)/\R\to M$ maps the fibers of the forgetful map
to nicely-embedded pseudoholomorphic curves,
a tangency of the map $t\mapsto ev^{-1}(\psi_{t}(p))$ to a fiber of the forgetful map
corresponds
via the embedding $ev$
with a tangency of the map $t\mapsto\psi_{t}(p)$ to
a curve $C'\in\M(C)/\R$, that is, a tangency of the Reeb vector field to a curve $C'\in\M(C)/\R$.
Since we know from
Theorem \ref{t:foliating-curves} that the Reeb vector field is everywhere transverse to
every curve in $\M(C)/\R$, no such tangency can exist.
We've thus shown the map taking
a point $t\in (-\ep, \ep)$ to the unique curve passing through $p$ is a local diffeomorphism.
\end{proof}

\section{The connected sum construction}\label{s:connect-sum}

This section is devoted to the proof of Theorem \ref{t:contact-connect-sum} below,
which shows
that we can perform a connected sum on a manifold
$M$ with contact form $\lambda$ and obtain a
contact form on the surgered manifold which has certain additional properties
which will allow us to prove Theorem \ref{t:0-surg-main}.
Previous descriptions/constructions of connected sums in contact manifolds
can be found in
\cite{meckert1982, weinstein1991}.
For our main theorem, we will need the Reeb vector field of the new contact form
to have some specific properties not addressed in these previous constructions.

For the statement of the theorem, we will need the following definition.
We will say that an open set $U$ in a contact manifold $(M, \lambda)$
is a \emph{flow-tube neighborhood} of a point $p\in M$ if the
closure $\bar U$ of $U$ is contained in a coordinate neighborhood in which
$\bar U$ takes the form
\[
\bar U=\overline{B_{\ep}(p)}\times [-\ep, \ep]\subset\R^{2}\times \R=\br{(x, y)}\times\br{z}
\]
for some $\ep>0$ and
the Reeb vector field takes the form $\X=\pm\dz$.

\begin{theorem}\label{t:contact-connect-sum}
Let $M$ a $3$-manifold
equipped with a nondegenerate contact form $\lambda$
and let $p$ and $q$ be distinct points in $M$, and let $\mathcal{O}$ be an open neighborhood
of $\br{p, q}$.
Then there exist disjoint flow-tube neighborhoods $U\subset\mathcal{O}$
and $V\subset\mathcal{O}$ of $p$ and $q$ respectively,
a manifold $M'$ equipped with a contact form $\lambda'$,
and an embedding $i:M\setminus\br{p, q}\to M'$ so that:
\begin{enumerate}
\item The contact form $\lambda'$ on $M'$ is nondegenerate.

\item The pullback $i^{*}\lambda'$ agrees with $\lambda$ on $M\setminus\br{U\cup V}$, that is,
if $\iota$ denotes the composition
\[
M\setminus\br{U\cup V}\hookrightarrow M\setminus\br{p, q}\stackrel{i}{\longrightarrow} M'
\]
with $M\setminus\br{U\cup V}\hookrightarrow M\setminus\br{p, q}$ the obvious inclusion,
then
\[
\iota^{*}\lambda'=\lambda.
\]

\item The set
\[
M'\setminus i(M\setminus\br{p, q})
\]
is diffeomorphic to an embedded $2$-sphere in $M'$, and
the set
\[
N:=M'\setminus \overline{i(M\setminus\br{U\cup V})},
\]
called the neck,
is diffeomorphic to $\R\times S^{2}$.
\item Letting $X_{\lambda'}$ denote the Reeb vector field of the contact form $\lambda'$, there
exists a simple, even periodic orbit $\gamma_{0}\subset N$ of $X_{\lambda'}$
contained entirely within $N$.  All other simple periodic orbits of $X_{\lambda'}$ pass
through points of $M'\setminus N$.

\item Given any compatible
$J\in\J(M', \lambda')$ we can find a $J'\in\J(M', \lambda')$ agreeing with
$J$ outside of the neck $N$ for which there
exists a pair of (nicely) embedded, disjoint
pseudoholomorphic planes $P^{\pm}=[S^{2}, i, \br{\infty}, da^{\pm}, u^{\pm}]\in\M(\lambda', J')/\R$
asymptotic to $\gamma_{0}$ in opposite directions with extremal winding.
Moreover,
$P^{\pm}*P^{\pm}=0=P^{+}*P^{-}$ and
the union
$P^{+}\cup\gamma_{0}\cup P^{-}$
of the planes and the periodic orbit form a ($C^{1}$-)smooth\footnote{
Our proof will actually provide a $C^{\infty}$-smooth sphere, but
for our main result we need only assume that the two planes approach $\gamma_{0}$
in opposite directions, in which case Theorem \ref{t:relative-asymptotics}
can be used to show that the resulting sphere is $C^{1}$.
This is addressed in \cite{fs2}.
}
sphere in $N\approx \R\times S^{2}$ which
generates $\pi_{2}(N)$.

\item\label{i:flow-lines} Let $\psi_{t}$ denote the flow of $\X$ and $\tl\psi_{t}$ denote the flow of
$X_{\lambda'}$.  Then:
\begin{enumerate}
\item If $p_{+}$ and $p_{-}$ are points in $\partial U$ and
$\gamma_{p}:[a, b]\subset\R\to \bar U\subset M$ is a smooth integral curve-segment of $\X$
connecting $p_{-}$ to $p$ to $p_{+}$ within $\bar U$, then there exist smooth
integral curve-segments
$\tl\gamma_{p, \pm}$
of $X_{\lambda'}$ lying in $\bar N$ so that
\begin{itemize}
\item $\tl\gamma_{p, -}$ connects $i(p_{-})$ to the plane $P^{-}$ and the interior of
$\tl\gamma_{p, -}$ lies in $N\setminus\br{P^{+}\cup\gamma_{0}\cup P^{-}}$.
\item $\tl\gamma_{p, +}$ connects the plane $P^{+}$ to $i(p_{+})$ and the interior of
$\tl\gamma_{p, +}$ lies in $N\setminus\br{P^{+}\cup\gamma_{0}\cup P^{-}}$.
\end{itemize}

\item Similarly, if $q_{\pm}$ are points in $\partial V$ and
$\gamma_{q}:[a', b']\subset\R\to \bar V\subset M$
is a smooth integral curve-segment of $\X$
connecting $q_{-}$ to $q$ to $q_{+}$ within $\bar V$, then there exist smooth
integral curve-segments
$\tl\gamma_{q, \pm}$
of $X_{\lambda'}$ lying in $\bar N$ so that
\begin{itemize}
\item $\tl\gamma_{q, -}$ connects $i(q_{-})$ to the plane $P^{+}$ and the interior of
$\tl\gamma_{q, -}$ lies in $N\setminus\br{P^{+}\cup\gamma_{0}\cup P^{-}}$.
\item $\tl\gamma_{p, +}$ connects the plane $P^{-}$ to $i(q_{+})$ and the interior of
$\tl\gamma_{q, +}$ lies in $N\setminus\br{P^{+}\cup\gamma_{0}\cup P^{-}}$.
\end{itemize}
\end{enumerate}
\end{enumerate}
\end{theorem}

Before proving this theorem, we will describe a contact connected sum on two copies
of $\R^{3}$ equipped with specific contact forms.
Since the connected sum operation we describe can be localized into arbitrarily small regions,
Darboux's theorem for contact manifolds will then allow us to
transfer the construction to any contact $3$-manifold.
We describe this construction in a series of lemmas.
In order to focus on the main points of the construction
we delay some details involving longer but more straightforward computations to
Appendix \ref{a:details}.

We consider $\R^{3}=\br{(x, y, z)}$
equipped with the contact forms $\lambda_{+}$ and $\lambda_{-}$ defined by
\[
\lambda_{\pm}=\pm dz+\frac{1}{2}(x\,dy-y\,dx)
\]
We equip $S^{2}$ with polar coordinate
$\phi\in\R/2\pi\Z$ and azimuthal coordinate $\theta\in[0, \pi]$ and consider the
$1$-form $\lambda_{1}$ on $\R\times S^{2}$ defined by
\[
\lambda_{1}=3\cos\theta\,d\rho-\rho\sin\theta\,d\theta+\frac{1}{2}\sin^{2}\theta\,d\phi
\]
where $\rho$ is the $\R$-coordinate.
It follows from Lemma \ref{l:smoothstuff} that $\lambda_{1}$
does in fact extend over the $\theta\in\br{0, \pi}$ locus to 
gives
a smooth $1$-form on $\R\times S^{2}$, and further that $\lambda_{1}$ is a contact
form on $\R\times S^{2}$.

\begin{lemma}\label{l:connect-sum-pullback}
Consider the maps $\Phi_{\pm}:\R^{\pm}\times S^{2}\to \R^{3}\setminus\br{0}$
defined by
\begin{equation}\label{e:Phi-pm-definition}
\Phi_{\pm}(\rho, \phi, \theta)
=\pm(\rho\sin\theta\cos\phi, \rho\sin\theta\sin\phi, \rho^{3}\cos\theta).
\end{equation}
Then $\Phi_{+}$ and $\Phi_{-}$ are smooth diffeomorphisms satisfying
\begin{equation}
\Phi_{\pm}^{*}\lambda_{\pm}=\rho^{2}\lambda_{1}
\end{equation}
with $\lambda_{+}$, $\lambda_{-}$, and $\lambda_{1}$ as defined above.
\end{lemma}

The proof of this lemma involves straightforward computation
and we give the details in Lemma \ref{l:connect-sum-pullback-app} in Appendix \ref{a:details}.
This lemma shows that we can take a connected sum between these two copies of
$\R^{3}$ in a way which preserves the Reeb flow outside of an arbitrarily small
neighborhood of the surgered region.  Indeed,
according to this lemma,
any smooth positive function
$f:\R\times S^{2}\to\R^{+}$
gives us a contact form $f\lambda_{1}$ on $\R\times S^{2}$
which is contactomorphic on $\R^{\pm}\times S^{2}$ to $(\R^{3}\setminus\br{0}, \lambda_{\pm})$
via the maps $\Phi_{\pm}$.
Furthermore, the Reeb flow of $f\lambda_{1}$ is conjugate via $\Phi_{\pm}$ 
to that of the Reeb vector field(s) for $(\R^{3}, \lambda_{\pm})$
on any region where $f(\rho, p)=\rho^{2}$.
Since we can easily construct smooth positive functions $f:\R\times S^{2}\to\R^{+}$
satisfying $f(\rho, p)=\rho^{2}$ on an arbitrarily small neighborhood of $\rho=0$,
this shows the Reeb vector fields of $f\lambda_{1}$ and those of $\lambda_{\pm}$
are identified via $\Phi_{\pm}$ outside of an arbitrarily neighborhood of the surgered region.

To establish that the connected sum operation can be carried out in such a way
as to ensure the
other properties we will need,
further properties on the function $f$ will be required.
Before discussing these properties
we first establish some properties of the contact form
\[
\lambda_{f}:=f\lambda_{1}
\]
and its associated contact structure
\[
\xi_{1}:=\ker\lambda_{f}=\ker\lambda_{1}.
\]
It will be convenient to define the function
\begin{equation}\label{e:g-def}
g(\theta):=2\cos^{2}\theta+1=3\cos^{2}\theta+\sin^{2}\theta
\end{equation}
and we note that $g$ defines a smooth function on $S^{2}$ as a result of
Lemma \ref{l:smoothstuff}.

\begin{lemma}\label{l:lambda-f-properties}
For $\theta\notin\br{0, \pi}$:
\begin{itemize}
\item The set
\begin{equation}\label{e:symplectic-basis}
\begin{aligned}
\B_{(\rho, \theta, \phi)}
&=\br{(fg)^{-1}(-3\cot\theta\,\dph+\frac{1}{2}\sin\theta\,\drh), 2\rho\csc\theta\,\dph+\dth} \\
&=:\br{v_{1}(\rho, \theta, \phi), v_{2}(\rho, \theta, \phi)}
\end{aligned}
\end{equation}
is a symplectic basis for
$(\xi_{1}, d\lambda_{f})$.

\item The Reeb vector field $\Xf$ of the contact form $\lambda_{f}$ is given by
\begin{equation}\label{e:xf}
\begin{aligned}
\Xf=
[gf^{2}]^{-1}\Big[&
(-\rho f_{\rho}-3f_{\theta}\cot\theta+2f)\,\dph  \\
&
+(3\cot\theta f_{\phi}-\frac{1}{2}\sin\theta f_{\rho})\,\dth  \\
&
+(\rho f_{\phi}+\frac{1}{2}\sin\theta f_{\theta}+f\cos\theta)\,\drh
\Big].
\end{aligned}
\end{equation}
\end{itemize}
\end{lemma}
The proof is straightforward computation.  Further details are given in 
Lemmas \ref{l:reeb-lemma-app}-\ref{l:lambda-f-properties-app} in the
Appendix \ref{a:details}.

We now have the following lemma
which identifies a condition which guarantees a periodic orbit
of
$\Xf$
on the sphere
$\rho=0$.

\begin{lemma}\label{l:existence-of-orbit}
The Reeb vector field $\Xf$ of $\lambda_{f}$ is a constant multiple of $\dph$ along the equator
$\theta=\pi/2$ of the sphere $\rho=0$ precisely when
$df=0$ there.
\end{lemma}

\begin{proof}
From \eqref{e:xf}, we have for $(\rho, \theta, \phi)=(0, \pi/2, \phi)$ that
\[
\Xf
=f^{-2}
\left[
2f\,\dph-\frac{1}{2}f_{\rho}\,\dth+\frac{1}{2}f_{\theta}\,\drh
\right].
\]
Thus, $\Xf(0, \pi/2, \phi)$ is a positive multiple of $\dph$ precisely when
$f_{\rho}(0, \pi/2, \phi)=f_{\theta}(0, \pi/2, \phi)=0$, in which case the formula for $\Xf$
along $(\rho, \theta, \phi)=(0, \pi/2, \phi)$ reduces to
$\Xf=(2/f)\,\dph$.
Thus $\Xf(0, \pi/2, \phi)$ is a constant multiple of $\dph$ precisely when
$f(0, \pi/2, \phi)$ is constant, which is equivalent to requiring
$f_{\phi}(0, \pi/2, \phi)=0$.
\end{proof}

By further restricting the function $f$ we can say that the periodic
orbit identified in the above lemma is the only (simple) periodic orbit of $\Xf$,
and we can arrange that the flow of $\Xf$ is tangent to
$\R\times\br{\theta=0, \pi}$.

\begin{lemma}\label{l:orbit-and-flow}
Let $f:\R\to\R^{+}$ be a smooth, positive function satisfying
\[
\rho f'(\rho)>0
\]
for all $\rho\ne 0$.
Then
\begin{itemize}
\item $\Xf$ has a unique (simple) periodic orbit occurring at
$\rho=0$, $\theta=\pi/2$.
\item Along $\theta=0$ (resp.\ $\theta=\pi$) locus,
$\Xf$ is a positive (resp.\ negative) multiple of $\drh$.
\end{itemize}
\end{lemma}

\begin{proof}
If $f$ depends only on the $\R$-coordinate $\rho$,
then the formula \eqref{e:xf} of the Reeb vector field
of $\lambda_{f}$ reduces to
\begin{equation}\label{e:reeb-rho-dependent}
\Xf=
[g(\theta)f(\rho)^{2}]^{-1}\left[
(-\rho f'+2f)\,\dph 
-\frac{1}{2}\sin\theta f'\,\dth 
+f\cos\theta\,\drh
\right]
\end{equation}
Define the function $Z:\R\times S^{2}\to\R$ by
\[
Z(\rho, \theta, \phi)=\rho\cos\theta.
\]
It follows from Lemma \ref{l:smoothstuff} that $Z$ defined
as such extends to a smooth function on all of $\R\times S^{2}$.
Then
\[
dZ=\cos\theta\,d\rho-\rho\sin\theta\,d\theta
\]
and so
\[
dZ(\Xf)
=(g(\theta)f(\rho)^{2})^{-1}
\bp{f\cos^{2}\theta+\frac{1}{2}f'\rho\sin^{2}\theta}
\]
which is nonnegative everywhere.
Therefore $Z$ is monotonic along any flow line of $\Xf$, and 
any periodic orbit of $\Xf$ must be contained in the zero locus of
$dZ(\Xf)$.
But $dZ(\Xf)=0$
precisely when both
$f\cos^{2}\theta$ and $f'\rho\sin^{2}\theta$ vanish,
which, in turn, happens precisely when
$\rho=0$ and $\theta=\pi/2$.

To see the second claim is true, we observe from Lemma \ref{l:smoothstuff} that 
$\dph$ and $\sin\theta\,\dth$ define smooth vector fields on $S^{2}$ which vanish at the
north and south poles $\theta\in\br{0, \pi}$.
Thus, the formula above for the Reeb vector field tell us that
\[
\Xf(\rho, \phi, 0)=[g(0)f(\rho)]^{-1}\cos(0)\,\drh=\frac{1}{3f(\rho)}\drh
\]
and
\[
\Xf(\rho, \phi, \pi)=[g(\pi)f(\rho)]^{-1}\cos(\pi)\,\drh=\frac{-1}{3f(\rho)}\drh
\]
which establishes the second claim of the lemma.
\end{proof}

We next compute the Conley--Zehnder index of the periodic orbit guaranteed by the
above lemma provided an additional assumption on the function $f$.

\begin{lemma}\label{l:cz-computation}
Assume that $f:\R\to\R^{+}$ is a smooth positive function satisfying
\[
\rho f'(\rho)>0
\]
for $\rho\ne 0$
and
\[
f''(0)>0.
\]
Then relative to the symplectic trivialization
\[
\B_{(0, \pi/2, \phi)}=\br{v_{1}(0, \pi/2, \phi), v_{2}(0, \pi/2, \phi)}=\br{\tfrac{1}{2f(0)}\drh, \dth}
\]
of $(\xi_{1}, d\lambda_{f})$
from \eqref{e:symplectic-basis},
the Conley--Zehnder index of the unique simple periodic orbit 
$\gamma_{0}$ of $\Xf$ is $0$.
\end{lemma}

\begin{proof}
We first observe that the proof of Lemma \ref{l:orbit-and-flow} above shows that
along the equator
\[
\Xf(0, \pi/2, \phi)=(2/f(0))\,\dph
\]
and thus the map $\gamma_{0}:\R/\Z\to\R\times S^{2}$ given by
$\gamma_{0}(t)=(0, \pi/2, 2\pi t)$ satisfies
\[
\dot\gamma_{0}(t)=2\pi\,\dph=(f(0)\pi)\, \Xf(\gamma_{0}(t))
\]
so $\gamma_{0}$ is a periodic orbit of period $\tau_{f}:=f(0)\pi$.

Let $\psi_{t}$ denote the flow of $\Xf$, that is $\psi_{t}$ satisfies
\[
\dot\psi_{t}(x)=\Xf(\psi_{t}(x))
\]
To compute the Conley--Zehnder
$\mu^{\Phi}(\gamma_{0})$
index of $\gamma_{0}$ in the trivialization
$\Phi$ arising from $\B_{(0, \pi/2, \phi)}$ we need to analyze the behavior of the linearized
flow $d\psi_{t}$ on $\xi_{1}$ in the trivialization $\Phi$.
Letting
\begin{align*}
\Psi(t)
&=\Phi(\psi_{t\tau_{f}}(0, \pi/2, 0))^{-1}d\psi_{t\tau_{f}}(0, \pi/2, 0)\Phi(0, \pi/2, 0) \\
&=\Phi(0, \pi/2, 2\pi t)^{-1}d\psi_{t\tau_{f}}(0, \pi/2, 0)\Phi(0, \pi/2, 0) \\
\end{align*}
we can write
\[
\Psi(t)=
\begin{bmatrix}
c_{11}(t) & c_{12}(t) \\
c_{21}(t) & c_{22}(t)
\end{bmatrix}
\]
where the $c_{ij}$ are defined by
\[
d\psi_{t\tau_{f}}(0, \pi/2, 0)v_{j}(0, \pi/2, 0)=\sum_{i}c_{ij}(t)v_{i}(0, \pi/2, 2\pi t)
\]
and satisfy $c_{ij}(0)=\delta_{ij}$.
Since the $d\psi_{t\tau_{f}}(0, \pi/2, 0)v_{j}(0, \pi/2, 0)$
defines a section of $\xi_{1}$ along $\gamma_{0}$
defined by pushing forward by the linearized flow
of $\tau_{f}\Xf$, the Lie derivative
$L_{\tau_{f}\Xf}=\tau_{f}L_{\Xf}$
(in the sense of \eqref{e:lie-derivative})
is well defined and vanishes.
Taking the Lie derivative $L_{\tau_{f}\Xf}$ 
then of the above equation gives us
\[
\sum_{i}c_{ij}'(t)v_{i}(0, \pi/2, 2\pi t)+c_{ij}(t) (L_{\tau_{f}\Xf}v_{i})(0, \pi/2, 2\pi t)=0.
\]
Letting $M(t)=[m_{ij}(t)]$ be the matrix
defined by
\[
(L_{\tau_{f}\Xf}v_{j})(0, \pi/2, 2\pi t)=-\sum_{i}m_{ij}(t)v_{i}(0, \pi/2, 2\pi t),
\]
we substitute in the above equation and use that the $v_{i}$ are a linearly independent to
conclude that
\[
c_{ij}'-\sum_{k}m_{ik}c_{kj}=0
\]
or, equivalently,
that $\Psi$ satisfies the linear ODE
\begin{equation}\label{e:linearized-flow-ode}
\begin{aligned}
\Psi'(t)&=M(t)\Psi(t)  \\
\Psi(0)&=I.
\end{aligned}
\end{equation}

To find $M(t)$ we extend
$v_{1}(0, \pi/2, 2\pi t)$ and $v_{2}(0, \pi/2, 2\pi t)$ to vector fields
\begin{align*}
\tl v_{1}(\rho, \theta, \phi)&=\frac{1}{2 f(0)}\,\drh  \\
\tl v_{2}(\rho, \theta, \phi)&=\dth
\end{align*}
which are locally constant in $(\rho, \theta, \phi)$ coordinates, and
use \eqref{e:reeb-rho-dependent} to compute
\begin{align*}
-(L_{\tau_{f}\Xf}v_{1})(0, \pi/2, 2\pi t)
&=\tau_{f}(v_{1}\Xf-\Xf\tl v_{1})(0, \pi/2, 2\pi t) \\
&=\tau_{f} (v_{1}\Xf)(0, \pi/2, 2\pi t) \\
&=(f(0)\pi)\frac{1}{2 f(0)}\drh \Xf(0, \pi/2, 2\pi t) \\
&=-\frac{\pi f''(0)}{4f(0)^{2}}\,\dth \\
&=-\frac{\pi f''(0)}{4f(0)^{2}}  v_{2}(0, \pi/2, 2\pi t)
\intertext{and}
-(L_{\tau_{f}\Xf}v_{2})(0, \pi/2, 2\pi t)
&=\tau_{f}(v_{2}\Xf-\Xf\tl v_{2})(0, \pi/2, 2\pi t) \\
&=\tau_{f} (v_{2}\Xf)(0, \pi/2, 2\pi t) \\
&=(f(0)\pi)\dth\Xf(0, \pi/2, 2\pi t) \\
&=-\pi\,\drh \\
&=-2\pi f(0) v_{1}(0, \pi/2, 2\pi t).
\end{align*}
We conclude
\[
M(t)=
\begin{bmatrix}
0 &-2\pi f(0) \\  -\frac{\pi f''(0)}{4f(0)^{2}}   & 0
\end{bmatrix}
=
\begin{bmatrix}
0 &-A^{2} \\ -B^{2} & 0
\end{bmatrix}.
\]
with $A=\sqrt{2\pi f(0)}$ and $B=\frac{\sqrt{\pi f''(0)}}{2f(0)}$.
Direct computation then shows that the solution
to \eqref{e:linearized-flow-ode} is given by
\[
\Psi(t)
=
\begin{bmatrix}
\cosh(AB t) & -(A/B)\sinh(ABt) \\
-(B/A)\sinh(ABt) & \cosh(ABt)
\end{bmatrix}
=
C
\begin{bmatrix}
e^{AB t} & 0 \\ 0 &  e^{-AB t}
\end{bmatrix}
C^{-1}
\]
where
$C$ is the symplectic matrix
\[
C=
\frac{1}{\sqrt{2}}
\begin{bmatrix}
A/B & 1 \\ -1 & B/A
\end{bmatrix}.
\]
A path of symplectic matrices of this form is well-known to have Conley--Zehnder index equal to
$0$ (see Lemma \ref{l:conley-zehnder-app} below) and thus
\[
\mu^{\Phi}(\gamma_{0})=\pathcz(\Psi)=0
\]
as claimed.
\end{proof}

We next show that we can choose a compatible $J$
on a neighborhood of $\rho=0$ so that
the northern/southern hemispheres of the
the sphere $\rho=0$ are projections of pseudoholomorphic planes
to $\R\times S^{2}$ asymptotic to the periodic orbit at the equator.

\begin{lemma}\label{l:planes}
Let $f:\R\to\R^{+}$ satisfy the hypotheses of Lemma \ref{l:cz-computation}
and let $J\in\J(\R\times S^{2}, \lambda_{f})$ be a compatible almost complex structure.
Then, for any open neighborhood $U$ of $\br{0}\times S^{2}$
there exists a compatible $J'\in \J(\R\times S^{2}, \lambda_{f})$
agreeing with $J$ outside of $U$
so that the planes
\begin{gather*}
P^{+}=\br{\rho=0, \theta\in[0, \pi/2)} \\
P^{-}=\br{\rho=0, \theta\in(\pi/2, \pi]}\\
\end{gather*}
given from the upper and lower hemispheres of the $\rho=0$ sphere
are projected pseudoholomorphic curves,
i.e.\ elements of $\M(\lambda_{f}, J')/\R$,
approaching their mutual asymptotic limit $\gamma_{0}$
in opposite directions with extremal winding.
Moreover,
\[
P^{+}*P^{+}=P^{-}*P^{-}=P^{+}*P^{-}=0.
\]
\end{lemma}

\begin{proof}
Considering $S^{2}$ as the unit sphere in $\R^{3}$ we can
define
diffeomorphisms between
$\C=\br{x+iy}$ and $P^{\pm}$ via radial projection from the origin to the planes
$(x, y, \pm 1)$.
In standard polar coordinates
$x=R\cos\Theta$, $y=R\sin\Theta$ on $\C$,
this radial projection map
from $P^{\pm}\to\C$ is given by
\begin{equation}\label{e:radial-proj}
\begin{aligned}
R(\theta, \phi)&=\tan\theta \\
\Theta(\theta, \phi)&=\phi.
\end{aligned} 
\end{equation}
Since $i$ in these coordinates is given by
\[
i(R\,\dR)=\dTh \qquad i(\dTh)=-R\,\dR
\]
and a straightforward computation shows that
\[
R\,\dR=\sin\theta\cos\theta\,\dth \qquad \dTh=\dph
\]
under the coordinate change \eqref{e:radial-proj},
we find that
the radial projection map induces a smooth complex structure on $T(S^{2}\setminus\br{\theta=\pi/2})$ given by
\begin{equation}\label{e:small-j-def}
\begin{aligned}
j\dth&=\sec\theta\csc\theta \,\dph \\
j\dph&=-\cos\theta\sin\theta\,\dth
\end{aligned} 
\end{equation}
in which the upper and lower hemispheres $P^{\pm}$ of $S^{2}$ are conformally
equivalent to $\C$.

We next claim that if $i_{0}:S^{2}\to\R\times S^{2}$ is the inclusion
\[
p\in S^{2}\mapsto (0, p)\in\R\times S^{2}
\]
then
$i_{0}^{*}\lambda_{f}\circ j$ is exact.
We have that $\lambda_{f}$ along the sphere $\rho=0$ is given by
\[
\lambda_{f}=f(0)\left[3\cos\theta\,d\rho+\frac{1}{2}\sin^{2}\theta\,d\phi\right]
\]
so the pullback of $\lambda_{f}$ to the sphere is given by
\[
i_{0}^{*}\lambda_{f}=\frac{1}{2}f(0)\sin^{2}\theta\,d\phi.
\]
Then
\[
d\phi\circ j=\sec\theta\csc\theta\,d\theta
\]
and hence
\begin{align*}
i_{0}^{*}\lambda_{f}\circ j
&=\frac{1}{2}f(0)\sin^{2}\theta\,d\phi\circ j \\
&=\frac{1}{2}f(0)\tan\theta\,d\theta \\
&=d\bp{-\frac{1}{2}f(0)\log\abs{\cos\theta}}.
\end{align*}
We note that the function
\[
(\theta, \phi) \mapsto -\frac{1}{2}f(0)\log\abs{\cos\theta}
\]
is smooth on the upper and lower hemispheres as a result of
Lemma \ref{l:smoothstuff}.

Next, let
\[
\pi_{\xi}:T(\R\times S^{2})=\Xf\oplus\xi_{1}\to\xi_{1}
\]
be the projection onto $\xi_{1}$ along $\Xf$,
given by the formula
\[
\pi_{\xi}(v)=v-\lambda_{f}(v)\Xf.
\]
We claim that along
$\rho=0$, $\pi_{\xi}|_{TS^{2}}:TS^{2}\to\xi_{1}$ is an isomorphism away from $\theta=\pi/2$.
Along $\rho=0$, the Reeb vector field is given by
\begin{align*}
\Xf
&=[g(\theta)f(0)^{2}]^{-1}\left[
2f(0)\,\dph 
+f(0)\cos\theta\,\drh
\right] \\
&=[(2\cos^{2}\theta+1)f(0)]^{-1}\left[
2\,\dph 
+\cos\theta\,\drh
\right]
\end{align*}
and $\lambda_{f}$ is given by
\[
\lambda_{f}=f(0)\left[3\cos\theta\,d\rho+\frac{1}{2}\sin^{2}\theta\,d\phi\right].
\]
Thus, with $v_{1}$ and $v_{2}$ as defined by \eqref{e:symplectic-basis},
straightforward computation shows that
\begin{align*}
\pi_{\xi}(\dth)
&=\dth-\lambda_{f}(\dth)\Xf \\
&=\dth \\
&=v_{2}(0, \theta, \phi)
\end{align*}
and
\begin{align*}
\pi_{\xi}(\dph)
&=\dph-\lambda_{f}(\dph)\Xf \\
&=-g(\theta)^{-1}\cos\theta\sin\theta
\left[
-3\cot\theta\,\dph +\frac{1}{2}\sin\theta\,\drh
\right] \\
&=-f(0)\cos\theta\sin\theta\,v_{1}(0, \theta, \phi).
\end{align*}
This computation
shows that $\pi_{\xi}|_{TS^{2}}$ is an isomorphism when $\theta\notin\br{0, \pi/2, \pi}$,
and since $TS^{2}=\xi_{1}$ at the north and south poles $\theta\in\br{0, \pi}$, it follows
that $\pi_{\xi}|_{TS^{2}}$ is an isomorphism away from $\theta=\pi/2$.

Given the results of the previous two paragraphs, we define
a compatible almost complex structure $J'$ on $\xi_{1}|_{\rho=0, \theta\ne\pi/2}$
by
\begin{equation}\label{e:J-definition}
J'=\pi_{\xi}|_{TS^{2}}\circ j\circ (\pi_{\xi}|_{TS^{2}})^{-1},
\end{equation}
and, as long as $J'$ extends smoothly over the equator $\theta=\pi/2$,
we have found a compatible $J'$ along $\rho=0$ for which $P^{\pm}$
are projected $\tl J'$-holomorphic curves.
Since the space of compatible complex multiplications
on a given symplectic vector space is nonempty and contractible,
there are no obstructions to extending a compatible $J'$ defined on $\rho=0$
smoothly to a $J'\in\J(\R\times S^{2}, \lambda_{f})$
which agrees outside of any given open neighborhood
of $\br{0}\times S^{2}$
with any previously chosen $J\in\J(\R\times S^{2}, \lambda_{f})$.
The computation of the previous paragraph together with the definition \eqref{e:small-j-def}
of $j$ shows, however, that
a $J'$ defined by \eqref{e:J-definition} will satisfy
\begin{align*}
J'(0, \theta, \phi)v_{1}(0, \theta, \phi)&=\tfrac{1}{f(0)}\,v_{2}(0, \theta, \phi) \\
J'(0, \theta, \phi)v_{2}(0, \theta, \phi)&=-f(0)\,v_{1}(0, \theta, \phi)
\end{align*}
away from the north and south poles $\theta\in\br{0, \pi}$.
Since $v_{1}$ and $v_{2}$ are a smooth basis for $\xi_{1}$ on $\theta\notin\br{0, \pi}$,
this implies that
the $J'$ defined by \eqref{e:J-definition} extends smoothly over the equator $\theta=\pi/2$.

It remains to check that
$P^{\pm}$ approach $\gamma_{0}$ in opposite directions with extremal winding
and that
\begin{equation}\label{e:planes-isect-zero}
P^{\pm}*P^{\pm}=P^{+}*P^{-}=0.
\end{equation}
We first claim that the planes approach their asymptotic limit with extremal winding, i.e.\ that
\begin{equation}\label{e:connect-sum-winding-bound}
\winfty^{\Phi}(P^{\pm})=\fl{\mu^{\Phi}(\gamma_{0})/2}.
\end{equation}
To see this, we observe that large $R=\text{constant}$ loops in $\C$
get mapped via the identification \eqref{e:radial-proj}
to $\theta=c$ loops with $c$ some constant close to but not equal to $\pi/2$.
It's straightforward to see that such loops lift via the exponential map
to sections of $\xi_{1}|_{\gamma_{0}}$ which have zero winding
relative to the trivialization $\Phi$ arising from the framing
$\B_{(0, \pi/2, \phi)}=\br{\tfrac{1}{2f(0)}\drh, \dth}$ from \eqref{e:symplectic-basis}.
Thus $\winfty^{\Phi}(P^{\pm})=0$.
Since we have already computed in Lemma \ref{l:cz-computation} that
$\mu^{\Phi}(\gamma_{0})=0$, we have confirmed \eqref{e:connect-sum-winding-bound}.
Since $P^{+}$ and $P^{-}$ are disjoint and $\gamma_{0}$ is even,
it then follows immediately from Lemma \ref{l:same-or-opposite}
and Theorem \ref{t:dir-of-approach} that $P^{+}$ and $P^{-}$ approach $\gamma_{0}$
in opposite directions.
Finally we prove that all intersection numbers are zero, i.e.\ that
\eqref{e:planes-isect-zero} holds.
We first observe that $P^{+}$ and $P^{-}$ are, by construction, disjoint embeddings.
Since we've already confirmed that $P^{+}$ and $P^{-}$ both converge to their unique
asymptotic limit with extremal winding,
\eqref{e:planes-isect-zero} is an immediately consequence of Corollaries
\ref{c:no-isect-gin-zero} and \ref{c:embedding-gin-zero}.
\end{proof}

We are now prepared to prove Theorem \ref{t:contact-connect-sum}.

\begin{proof}[Proof of Theorem \ref{t:contact-connect-sum}]
Recall $(M, \xi=\ker\lambda)$ denotes a contact $3$-manifold
equipped with nondegenerate contact form.
Given $p\ne q\in M$ and any open neighborhood $\mathcal{O}$ of $\br{p, q}$,
we can apply Darboux's theorem for contact manifolds (see e.g.\ \cite[Theorem 2.24]{geiges})
to find disjoint open neighborhoods $O_{p}$, $O_{q}\subset\mathcal{O}$
of $p$ and $q$ respectively,
and embeddings
$\phi_{p}:O_{p}\to\R^{3}$, and $\phi_{q}:O_{q}\to\R^{3}$ with
$\phi_{p}(p)=0$, $\phi_{q}(q)=0$ and
\begin{equation}\label{e:coord-pullback}
\begin{aligned}
\phi_{p}^{*}\lambda_{+}&=\lambda \\
\phi_{q}^{*}\lambda_{-}&=\lambda.
\end{aligned}
\end{equation}
Choosing an $\ep>0$ so that
$\overline{B_{\ep}(0)}\times [-\ep, \ep]\subset \phi_{p}(O_{p})\cap\phi_{q}(O_{q})$
gives flow tube neighborhoods
\begin{gather*}
U=\phi_{p}^{-1}\bp{B_{\ep}(0)\times (-\ep, \ep)} \\
V=\phi_{q}^{-1}\bp{B_{\ep}(0)\times (-\ep, \ep)} \\
\end{gather*}
of $p$ and $q$, respectively, identified via
$\phi_{p}$ and $\phi_{q}$ with neighborhoods of $0$ in $(\R^{3}, \lambda_{+})$
and $(\R^{3}, \lambda_{-})$, respectively.

Next, with the maps
$\Phi_{\pm}:\R^{\pm}\times S^{2}\to\R^{3}\setminus\br{0}$
as defined above in \eqref{e:Phi-pm-definition},
choose an $\ep'>0$ so that
\begin{align*}
\Phi_{+}\bp{(0,\ep')\times S^{2}}&\subset B_{\ep}(0)\times (-\ep, \ep) \\
\Phi_{-}\bp{(-\ep', 0)\times S^{2}}&\subset B_{\ep}(0)\times (-\ep, \ep).
\end{align*}
Given such an $\ep'>0$ we can find a smooth positive function
$f:\R\to\R^{+}$ satisfying\footnote{Such functions are easy to construct.  See Lemma \ref{l:function-construction} for an example.}
\begin{itemize}
\item $f(\rho)=\rho^{2}$ for $\abs{\rho}\ge \ep'/2$,
\item $\rho f'(\rho)$ for $\rho\ne 0$, and 
\item $f''(0)>0$.
\end{itemize}
As explained following Lemma \ref{l:connect-sum-pullback} above,
this $f$ gives us a 
contact form
$\lambda_{f}=f\lambda_{1}$
on $(-\ep', \ep')\times S^{2}$
so that the maps
\begin{gather*}
\Phi_{+}:\bp{(0,\ep')\times S^{2}, \lambda_{f}}\to (B_{\ep}(0)\times (-\ep, \ep)\setminus\br{0}, \lambda_{+}) \\
\Phi_{-}:\bp{(-\ep', 0)\times S^{2}, \lambda_{f}}\to (B_{\ep}(0)\times (-\ep, \ep)\setminus\br{0}, \lambda_{-})
\end{gather*}
are contact diffeomorphisms onto their images which,
on $(\ep'/2, \ep')\times S^{2}$ and $(-\ep', -\ep'/2)\times S^{2}$,
satisfy
\[
\Phi_{\pm}^{*}\lambda_{\pm}=\rho^{2}\lambda_{1}=f\lambda_{1}
\]
and hence by \eqref{e:coord-pullback}
\begin{equation}\label{e:connect-sum-form-pullback}
\begin{aligned}
(\phi_{p}^{-1}\circ\Phi_{+})^{*}\lambda&=f\lambda_{1}=\lambda_{f} \\
(\phi_{q}^{-1}\circ\Phi_{-})^{*}\lambda&=f\lambda_{1}=\lambda_{f}.
\end{aligned}
\end{equation}
We then define
\[
M'=\bp{M\setminus\br{p, q}\amalg (-\ep', \ep')\times S^{2}}/\sim
\]
where $\sim$ is the equivalence relation identifying points in
$(-\ep', 0)\times S^{2}$ and $(0, \ep')\times S^{2}$ with their
respective images in $M\setminus\br{p, q}$
under $\phi_{q}^{-1}\circ\Phi_{-}$ and $\phi_{p}^{-1}\circ\Phi_{+}$,
and we define
\[
i:M\setminus\br{p, q}\to M'
\]
to be the naturally induced inclusion.
With
\begin{align*}
B_{p}&:=(\phi_{p}^{-1}\circ\Phi_{+})\bp{(0, \ep'/2)\times S^{2}} \\
B_{q}&:=(\phi_{q}^{-1}\circ\Phi_{-})\bp{(-\ep'/2, 0)\times S^{2}},
\end{align*}
we define $\lambda'$ on $M'$ by
\[
\lambda'=
\begin{cases}
\lambda & \text{on $M\setminus\overline{B_{p}\cup B_{q}}$} \\
\lambda_{f} & \text{on $(-\ep', \ep)\times S^{2}$}.
\end{cases}
\]
It follows from \eqref{e:connect-sum-form-pullback} and the definition of $M'$ that
$\lambda'$ defines a smooth contact form on $M'$.
Moreover, it is easily verified that items (2) and (3)
in theorem are satisfied.

Next, since the neck
\[
N:=M'\setminus \overline{i(M\setminus\br{U\cup V})}
\]
equipped with the contact form $\lambda'$
can be identified via the maps
$\Phi_{+}^{-1}\circ \phi_{p}$ and $\Phi_{-}^{-1}\circ \phi_{q}$
with a subset of $\R\times S^{2}$ equipped with the contact form $\lambda_{f}$,
it follows immediately from Lemma \ref{l:orbit-and-flow}
that there is precisely one (simple) periodic orbit $\gamma_{0}$
of the Reeb vector field $X_{\lambda'}$ of $\lambda'$ contained in the neck.
Thus any other (simple) periodic orbit $X_{\lambda'}$ must pass through points of $M'\setminus N$.
Moreover, it follows from Lemma \ref{l:cz-computation} that $\gamma_{0}$ is
an even orbit.  Thus item (4) of the theorem is verified.
Item (5) meanwhile follows immediately from Lemma \ref{l:planes}.

To see that Condition (6) holds, we note that since the
Reeb vector field of $\lambda_{+}$ is $\dz$,
the points $p_{\pm}:=\phi_{p}^{-1}(0, 0, \pm\ep)$
are points in $\partial U$ which are connected by a flow line
$\gamma_{p}(t)=\phi_{p}^{-1}(0, 0, t)$
which is contained in $\bar U$ and passes through $p=\phi_{p}(0,0,0)$.
Moreover, since we can easily verify 
from \eqref{e:Phi-pm-definition}
that
\begin{align*}
\Phi_{+}^{-1}\bp{\br{(0, 0, z)\,|\, z>0}}&=\R^{+}\times\br{\theta=0} \\
\Phi_{+}^{-1}\bp{\br{(0, 0, z)\,|\, z<0}}&=\R^{+}\times\br{\theta=\pi}
\end{align*}
it follows that
\begin{align*}
\Phi_{+}^{-1}\circ\phi_{p}\circ\gamma_{p}\bp{(0,\ep)}&\subset \R^{+}\times\br{\theta=\pi} \\
\Phi_{+}^{-1}\circ\phi_{p}\circ\gamma_{p}\bp{(-\ep, 0)}&\subset \R^{+}\times\br{\theta=0}.
\end{align*}
We can then apply the second claim of Lemma \ref{l:orbit-and-flow}
to conclude that item (6a) of the theorem holds.
Item (6b) follows similarly.

At this point, all claims of Theorem \ref{t:contact-connect-sum} hold
with the possible exception of item (1): nondegeneracy of the contact form $\lambda'$.
By construction we have that
$i^{*}\lambda'=\lambda$ outside of the region identified with
$(-\ep'/2, \ep'/2)\times S^{2}$ in the construction.
Thus any periodic orbit created by this construction (i.e.\ not identified via $i$ with a
periodic orbit of $\X$) must pass through the region
$(-\ep'/2, \ep'/2)\times S^{2}$.  Moreover,
since $\gamma_{0}$ is the only periodic orbit contained within this region, any other
new (and thus potentially nondegenerate)
orbit created in the connected sum operation must pass through
the boundary $\br{\pm\ep'/2}\times S^{2}$ of the region.
Applying results from
\cite{robinson}, we can find a $C^{\infty}$-small function $h:M'\to\R$
supported in an arbitrarily small neighborhood of these
spheres, so that the contact form
$e^{h}\lambda'$
has only nondegenerate periodic orbits.
Moreover, since the support and ($C^{\infty}$-) size of the function can both be chosen arbitrarily
small,
and since all claims of the theorem 
remain true under sufficiently $C^{\infty}$-small
perturbations with sufficiently small support in a neighborhood of $\br{\pm\ep'/2}\times S^{2}$,
we can carry out this perturbation of the contact form while maintaining
all the claims of the theorem.
\end{proof}

\section{Proof of Theorem \ref{t:0-surg-main}}\label{s:main-proof}

In this section we prove Theorem \ref{t:0-surg-main}.
Here we will use the alternate definition of finite energy foliation
furnished by Corollary \ref{c:fol-alt-def}, and will thus work almost exclusively with
projections of pseudoholomorphic curves to the $3$-manifold.
All curves should thus be thought of as equivalence classes of maps
to the $3$-manifold unless otherwise stated.
Since we deal nearly exclusively with simple curves (i.e.\ those which do not
factor though a branched cover of degree $2$ or greater)
such an equivalence class of maps is entirely determined by the image in $M$ of a
representative map from the class.  We will thus generally make no distinction between a curve
and its image in $M$.

Our standing assumptions throughout the section will be:
\begin{itemize}
\item $(M, \lambda)$ is a $3$-manifold with a nondegenerate contact form $\lambda$,
\item $J\in\J(M,\lambda)$ is a compatible complex structure on $\xi=\ker\lambda$, and
\item $\F$ is a stable finite energy foliation for the data $(M, \lambda, J)$
with energy $E(\F)=E_{0}$.
\end{itemize}
Given the foliation $\F$, 
we consider a subset $\mathcal{U}$ of $M\times M\setminus\Delta(M)$
defined to be the set of pairs of distinct points $(p, q)\in M\times M\setminus\Delta(M)$
in $M$ with $p$ and $q$ lying on distinct index-$2$ leaves of the foliation.
It is straightforward to use Lemma \ref{l:finite-index-0-and-1}
and Corollary \ref{c:mod-spaces-mod-r} to argue that $\mathcal{U}$
is an open, dense subset of $M\times M\setminus\Delta(M)$.
We will show that the manifold
$M'$ formed by taking the connected sum at any given
pair of points
$(p, q)\in\mathcal{U}$
admits a contact form $\lambda'$ and compatible $J'\in\J(M', \lambda')$ so that
the data $(M', \lambda', J')$ admits a
stable finite energy foliation $\F'$  
with $E(\F')=E(\F)$.
Moreover, our construction will show that the change in the contact form
and almost complex structure can be localized to an arbitrarily small neighborhood
of the points $p$ and $q$; that is, if $i:M\setminus\br{p, q}\to M'$ is the natural
inclusion, we can arrange that $i^{*}\lambda'=\lambda$ and $i^{*}J'=J$ on
the complement of any given neighborhood of $\br{p, q}$.

To start the construction, we choose a pair of points
$(p, q)\in\mathcal{U}$; that is, we choose distinct points
$p$ and $q$ in $M$ so that
\begin{itemize}
\item $p$ lies on a curve $C_{p}\in\F$ with $\ind(C_{p})=2$,
\item $q$ lies on a curve $C_{q}\in\F$ with $\ind(C_{q})=2$, and
\item $C_{q}\ne C_{p}$ (in $\M(\lambda, J)/\R$).
\end{itemize}
We recall from Corollary \ref{c:smooth-R-inv-foliation} that all curves in the moduli spaces
$\M(C_{p})/\R$ and $\M(C_{q})/\R$ are also in the foliation $\F$
(where $\M(C)$ is the notation introduced in Section \ref{s:foliating-curves}
to indicate all simple curves in $\M(\lambda, J)$ which are relatively homotopic to $C$).
Letting $\psi_{t}$ denote the flow generated by the Reeb vector field $\X$ associated to
$\lambda$, we can apply Corollary \ref{c:flow-loc-diff}
to find an $\ep_{0}>0$
so that
\begin{itemize}
\item for each $t\in [-\ep_{0}, \ep_{0}]$ there is a unique curve of $\M(C_{p})/\R$ passing through the point $\psi_{t}(p)$,
and so that the map taking $t\in [-\ep_{0}, \ep_{0}]$ to the unique curve in $\M(C_{p})/\R$ passing through $\psi_{t}(p)$ is an embedding;
\item for each $t\in [-\ep_{0}, \ep_{0}]$ there is a unique curve of $\M(C_{q})/\R$ passing through the point $\psi_{t}(q)$,
and so that the map taking $t\in [-\ep_{0}, \ep_{0}]$ to the unique curve in $\M(C_{q})/\R$ passing through $\psi_{t}(q)$ is an embedding;
\item the collection of curves passing through the points $\psi_{t}(p)$ for $t\in [-\ep_{0}, \ep_{0}]$ and the collection of curves passing through the points $\psi_{t}(q)$ for $t\in [-\ep_{0}, \ep_{0}]$ are disjoint.
\end{itemize}
We define points $p_{\pm}=\psi_{\pm\ep_{0}}(p)$ and similarly $q_{\pm}=\psi_{\pm\ep_{0}}(q)$, and
let $C_{p,\pm}$ denote the unique curve in $\M(C_{p})/\R$ through $p_{\pm}$ and similarly
let $C_{q, \pm}$ denote the unique curve in $\M(C_{q})/\R$ through $q_{\pm}$.

\begin{figure}
\scalebox{.6}{\includegraphics{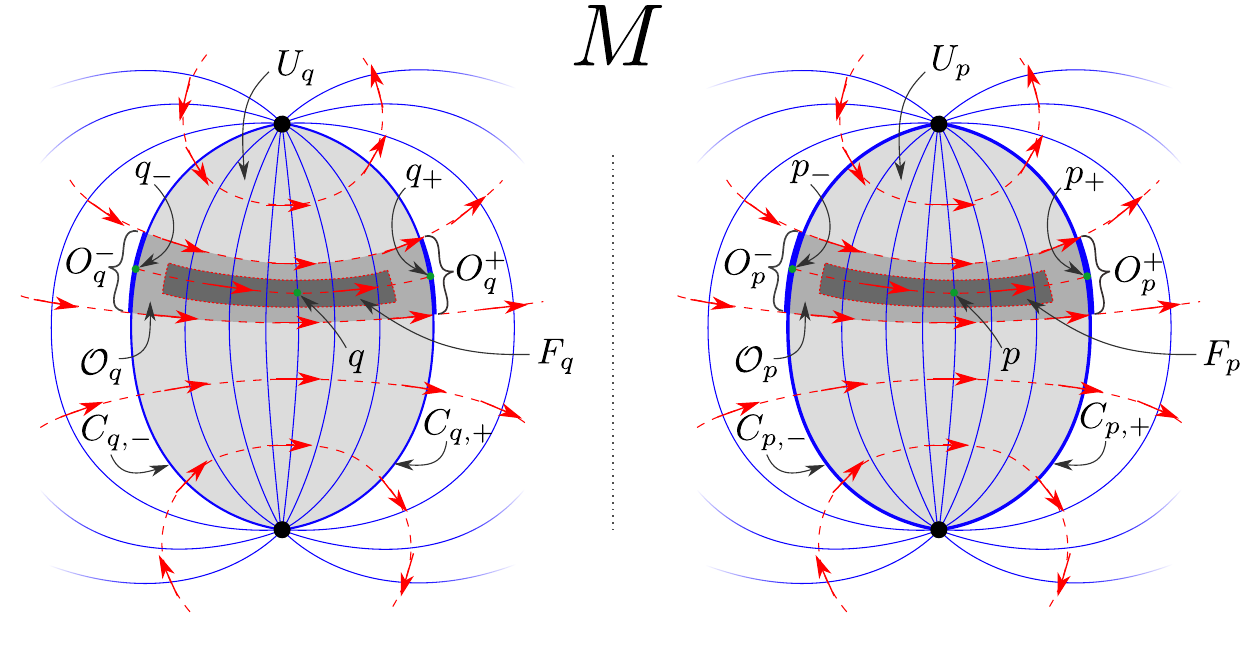}}
\end{figure}
Given the above, we define an
open set $U_{p}$ to be the union of the images of the curves in 
$\M(C_{p})/\R$ passing through the points
$\psi_{t}(p)$ for $t\in (-\ep_{0}, \ep_{0})$, and similarly $U_{q}$ is the image of the curves
in $\M(C_{q})/\R$ passing through the points $\psi_{t}(q)$ for $t\in (-\ep_{0}, \ep_{0})$.
There exists an open neighborhood $O^{-}_{p}$ of $p_{-}$ in $C_{p, -}$ and a positive function
$f_{p}:O_{p}^{-}\to\R^{+}$ so that
\[
\br{\psi_{f(z)}(z)\,|\,z\in O_{p}^{-}}
\]
is an open neighborhood of $p_{+}$ in $C_{p, +}$ and
$\psi_{t}(z)\in U_{p}$ for all $t\in(0, f(z))$.
Define an open set $\mathcal{O}_{p}$ by
\[
\mathcal{O}_{p}=\bigcup_{z\in O_{p}^{-}}\bigcup_{t\in(0, f(z))}\psi_{t}(z)
\]
and define an open set $\mathcal{O}_{q}\subset U_{q}$ analogously.
Choosing
an open flow tube neighborhood $F_{p}$ of $p$ with $\bar F_{p}\subset \mathcal{O}_{p}$ and
an open flow tube neighborhood $F_{q}$ of $q$ with $\bar F_{q}\subset \mathcal{O}_{q}$,
we can apply Theorem \ref{t:contact-connect-sum} to find
a nondegenerate contact manifold $(M', \xi'=\ker\lambda')$ with compatible $J'\in\J(M', \xi')$,
and an embedding $i:M\setminus\br{p, q}\to M'$
so that:
\begin{figure}
\scalebox{.55}{\includegraphics{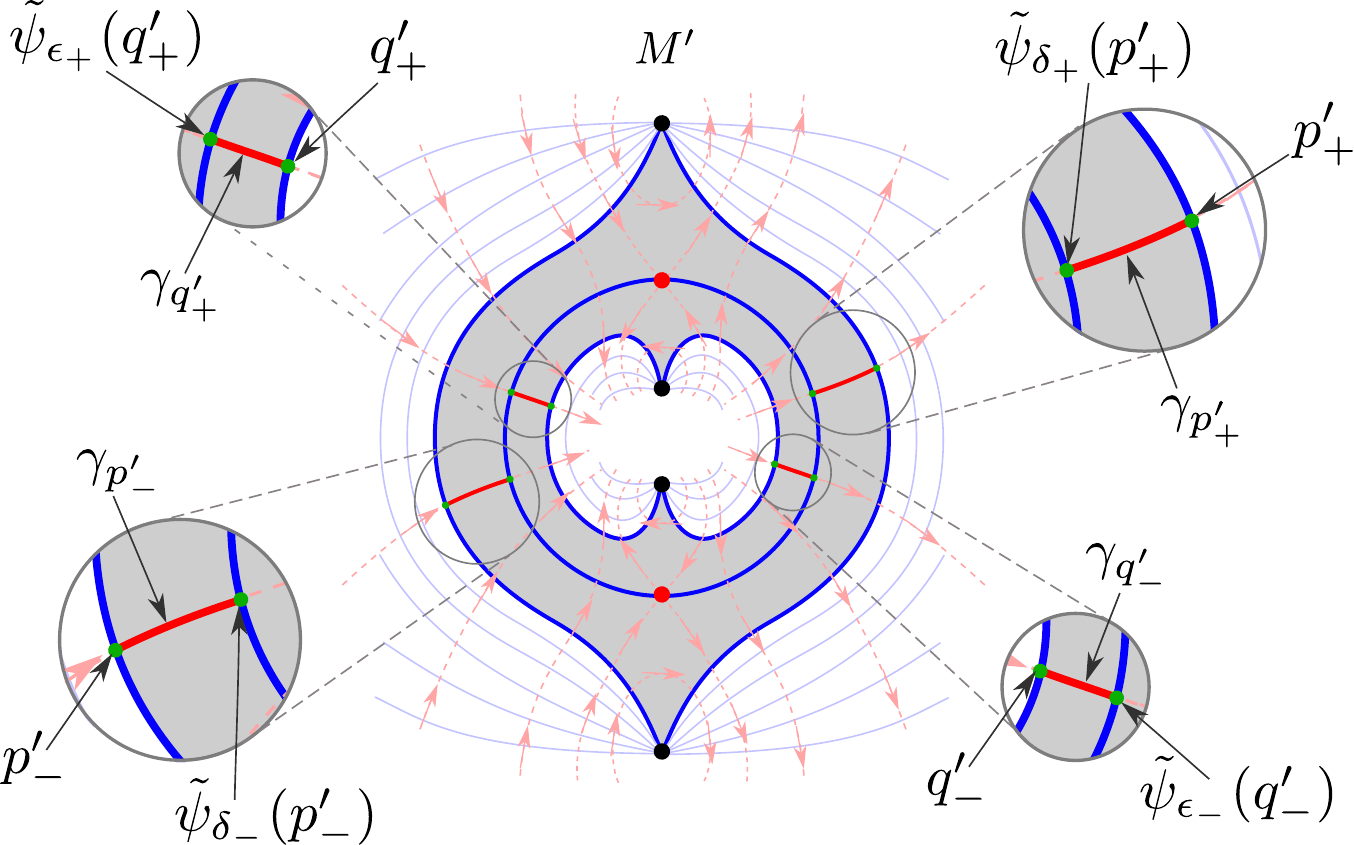}}
\end{figure}
\begin{enumerate}
\item $M'$ is diffeomorphic to the connected sum of $M$ taken at $p$ and $q$,
and the set $M'\setminus i(M\setminus(F_{p}\cup F_{q}))$ (called the neck) is diffeomorphic
to $\R\times S^{2}$.

\item On $M\setminus(F_{p}\cup F_{q})$, $i^{*}\lambda'=\lambda$ and $i^{*}J'=J$.

\item There is precisely one simple (unparametrized)
periodic orbit $\gamma_{0}$ contained in the neck.
Moreover $\gamma_{0}$ has even Conley--Zehnder index.

\item There exist two distinct, nonintersecting, nicely-embedded planes
$P^{\pm}\in\M(\lambda', J')/\R$ contained in the neck
which are asymptotic to $\gamma_{0}$
in opposite directions (see discussion preceding Lemma \ref{l:same-or-opposite})
with extremal winding
(see Theorem \ref{t:wind-infinity-bound}).
Moreover,
\[
P^{+}*P^{+}=P^{-}*P^{-}=P^{+}*P^{-}=0,
\]
and the union
of the images $P^{+}$, $P^{-}$ and $\gamma_{0}$
form a $C^{1}$-smooth sphere which divides the neck into two pieces,
each homeomorphic to $\R\times S^{2}$.

\item If $\tl\psi_{t}$ is the flow of the Reeb vector field $X_{\lambda'}$ then there exist
real numbers $\delta_{+}<0<\delta_{-}$ and $\ep_{+}<0<\ep_{-}$ so that,
\begin{itemize}
\item $\tl\psi_{\delta_{+}}(i(p_{+}))\in P^{+}$ and $\tl\psi_{t}(i(p_{+}))\notin P^{\pm}$ for all $t\in (\delta_{+}, 0]$.
\item $\tl\psi_{\delta_{-}}(i(p_{-}))\in P^{-}$ and $\tl\psi_{t}(i(p_{-}))\notin P^{\pm}$ for all $t\in [0, \delta_{-})$.
\item $\tl\psi_{\ep_{+}}(i(q_{+}))\in P^{-}$ and $\tl\psi_{t}(i(q_{+}))\notin P^{\pm}$ for all $t\in (\ep_{+}, 0]$.
\item $\tl\psi_{\ep_{-}}(i(q_{-}))\in P^{+}$ and $\tl\psi_{t}(i(q_{-}))\notin P^{\pm}$ for all $t\in [0, \ep_{-})$.
\end{itemize}
Letting $p'_{\pm}:=i(p_{\pm})$ and $q'_{\pm}:=i(q_{\pm})$,
we define embedded flow-line segments $\gamma_{p'_{\pm}}$, $\gamma_{q'_{\pm}}\subset M'$ by:
\begin{itemize}
\item $\gamma_{p'_{+}}=\br{\tl\psi_{t}(p'_{+})\,|\, t\in [\delta_{+}, 0]}$
\item $\gamma_{p'_{-}}=\br{\tl\psi_{t}(p'_{-})\,|\, t\in [0, \delta_{-}]}$
\item $\gamma_{q'_{+}}=\br{\tl\psi_{t}(q'_{+})\,|\, t\in [\ep_{+}, 0]}$
\item $\gamma_{q'_{-}}=\br{\tl\psi_{t}(q'_{-})\,|\, t\in [0, \ep_{-}]}$
\end{itemize}
\end{enumerate}

Since $i^{*}\lambda'=\lambda$ and $i^{*}J'=J$ outside of $F_{p}\cup F_{q}$,
and since the curves
$C_{p, \pm}$ do not meet $F_{p}\cup F_{q}$, it follows that
$C'_{p, \pm}:=i(C_{p, \pm})$ are nicely-embedded finite-energy curves in $\M(\lambda', J')/\R$.
We note that since $p\in F_{p}\subset U_{p}$ and $q\in F_{q}\subset U_{q}$
the connected sum operation yields an open set $U\subset M'$ arising
as the connected sum of the sets $U_{p}$ and $U_{q}$ taken at $p$ and $q$.
Moreover, the set $U$ is divided into two open subsets by sphere formed by
$P^{+}\cup \gamma_{0}\cup P^{-}$ in the neck.  We call
the subset coming from the $p$-side of the connected sum $U_{p}'$ and the subset coming from the $q$-side of
the connected sum $U_{q}'$.
We note that, by construction, the boundary of $U_{p}'$ consists of the curves $C'_{p, \pm}$ and
their asymptotic limits, along with $P^{\pm}$ and $\gamma_{0}$.

\begin{lemma}\label{l:no-orbits}
There are no periodic orbits of $X_{\lambda'}$
contained within $U_{p}'$.
\end{lemma}

\begin{proof}
We first argue that there are no periodic orbits $\gamma$ of
the Reeb vector field $\X$ on the unsurgered manifold
contained in the set
$U_{p}$.  Indeed, since $U_{p}$ is foliated by curves in $\M(C_{p})/\R$, then $\gamma$ would
intersect some curve $C'\in\M(C_{p})/\R$, which implies that
$(\R\times \gamma)*C'>0$.
By homotopy invariance of the $*$-product, we conclude
$(\R\times\gamma)*C>0$ for any $C\in\M(C_{p})/\R$
and,
in particular,
$(\R\times\gamma)*C_{p, \pm}>0$.
However, since $\gamma$ is contained in $U_{p}$ and thus not
an asymptotic limit of $C_{p, \pm}$,
this implies that $\gamma$ intersects $C_{p, \pm}$.  This,
however, contradicts the assumption that $\gamma\subset U_{p}$, and we conclude
that there are no periodic orbits $\gamma$ of $\X$ contained in $U_{p}$.

Now, assume $\gamma$ is a simple
periodic orbit of $X_{\lambda'}$ with $\gamma\subset U_{p}'$.
Then, we claim
the previous paragraph shows that $\gamma$ must enter the neck.
Indeed if not, then $\gamma$
is identified via
the map
$i:M\setminus\br{p, q}\to M'$
with a periodic orbit of $\X$ contained in $U_{p}$, of which, we have just
argued, there are none.  Moreover, since $\gamma_{0}$ is the only
periodic orbit contained entirely within the neck, $\gamma$ must
pass through points of $U_{p}'$ both inside and outside the neck.
But by construction ---
specifically
that the connected sum is carried out in flow tubes neighborhoods contained in
open sets consisting of flow lines connecting the curves $C_{p, +}$ and $C_{p, -}$
(or $C_{q, +}$ and $C_{q, -}$) ---
any flow line entering the neck
in $U_{p}'$ must hit $C'_{p, -}$ in backward time, while
any flow line exiting the neck in $U_{p}'$ must hit
$C'_{p, +}$ in forward time.
Thus $\gamma$ intersects either $C'_{p, +}$ or $C'_{p, -}$ which contradicts the assumption that
$\gamma\subset U_{p}'$.
\end{proof}

\begin{lemma}\label{l:curves-subset}
Let $C\in\M(C'_{p, \pm})/\R$.  Then either $C\subset U_{p}'$ or $C\cap U_{p}'=\emptyset$.
Moreover, the set of curves $C\in\M(C'_{p, \pm})/\R$ with $C\subset U_{p}'$ is nonempty and open.
\end{lemma}

\begin{proof}
We first observe that, by construction,
the boundary of $U'_{p}$
consists of the pseudoholomorphic curves
$C'_{p, \pm}$, $P^{\pm}$ and their asymptotic limits.
Moreover, we know that $C'_{p, \pm}*C'_{p, \pm}=C'_{p, +}*C'_{p, -}=0$ and,
since the curves
$P^{\pm}$ are disjoint from and share no common orbits with the curves
$C'_{p, \pm}$,
it follows immediately from the definition
of the $*$-product that $C'_{p, +}*P^{\pm}=C'_{p, -}*P^{\pm}=0$.
Given a curve $C\in\M(C'_{p, \pm})/\R$, homotopy invariance of the intersection number
then allows us to conclude that
the intersection number of $C$ with each 
of $C'_{p, \pm}$ and $P^{\pm}$ is zero.
Theorem \ref{t:gin-zero}/\ref{t:no-isect} then lets us further conclude that $C$
doesn't intersect any of the asymptotic limits of the curves in the boundary of $U_{p}'$.
Thus the curve $C$ can't intersect any of the curves in the boundary of $U_{p}'$
unless it coincides with that curve.  We conclude that either $C\subset U_{p}'$
or $C$ is disjoint from $U_{p}'$.

We next show that the set of curves $C\in\M(C'_{p, \pm})/\R$ with $C\subset U_{p}'$
is nonempty.  Given that we've shown in the previous paragraph that
a curve $C\in\M(C'_{p, \pm})/\R$
meeting $U_{p}'$ must be contained in $U_{p}'$,
it suffices to show there are curves $C\in\M(C'_{p, \pm})/\R$ meeting $U_{p}'$.
This follows from Corollary \ref{c:mod-spaces-mod-r}.
Indeed, since the evaluation map
$ev:\M_{1}(C'_{p, \pm})/\R\to M$ is an embedding, the image of the evaluation map
is open.  Therefore, given any point $x\in C'_{p, \pm}\subset \partial U_{p}'$, there is
an open set around $x$ in the image of evaluation map.  Since an open set around
a boundary point of $U_{p}'$ must meet $U_{p}'$, there are points in $U_{p}'$ in the image of
the evaluation map, which is equivalent to there being points in $U_{p}'$ with curves
in $\M(C'_{p, \pm})/\R$ passing through them.

Finally we show that the set of curves $C\in\M(C'_{p, \pm})/\R$ with $C\subset U_{p}'$
is open.  This follows from Corollary \ref{c:flow-loc-diff}.
Indeed, given a curve $C\in\M(C'_{p, \pm})/\R$ passing through a point
$x\in U_{p}'$ there is an $\ep>0$ so that
for every $t\in (-\ep, \ep)$
there is a unique
curve of $\M(C'_{p, \pm})/\R$ passing through $\tl\psi_{t}(x)$,
and so that the map taking a point
$t\in (-\ep, \ep)$ to the unique curve
in $\M(C'_{p, \pm})/\R$
passing through $\tl\psi_{t}(x)$ is an embedding.
Since $U_{p}'$ is open, we can, by shrinking $\ep$ if necessary, assume that
$\tl\psi_{t}(x)\in U_{p}'$ for all $t\in (-\ep, \ep)$.
We have thus found a subset $I$ of $\M(C'_{p, \pm})/\R$
diffeomorphic to an open interval
and containing $C$ so that each curve $C'\in I$ meets $U_{p}'$
and thus, according to the results of the first paragraph, 
is contained in $U_{p}'$.
\end{proof}

Since, as we've noted above, the images of $C'_{p, \pm}$ are contained in the boundary of
the set $U_{p}'$, it follows that
the curves $C'_{p, \pm}$ are in the boundary
of the set $\br{C\in\M(C'_{p, \pm})/\R\,|\, C\subset U_{p}'}$.
We define a submanifold with boundary $\M_{p, \pm}\subset\M(C'_{p, \pm})/\R$
to be
the connected
component of
\[
\br{C'_{p, \pm}}\cup \br{C\in\M(C'_{p, \pm})/\R\,|\, C\subset U_{p}'}
\]
containing $C'_{p, \pm}$.
As a result of the above discussion, it is clear that $\M_{p, \pm}$ is
diffeomorphic to
a half-open interval, and
in the following we seek to characterize
$\overline{\M_{p,\pm}}\setminus\M_{p,\pm}$,
where $\overline{\M_{p,\pm}}$ denotes the compactification
of $\M_{p, \pm}$ in the SFT topology
(relevant information is reviewed in Section \ref{s:compactness} above).
According to our construction ---
specifically that $\M(C'_{p, \pm})/\R$ is a smooth $1$-manifold
and  $\M_{p, \pm}$ 
can be identified with an embedded half-open subinterval ---
the set
$\overline{\M_{p,\pm}}\setminus\M_{p,\pm}$
is either an element of $\M(C'_{p, \pm})/\R$ or
is contained in the boundary of
$\M(C'_{p, \pm})/\R$ and thus,
according to the main theorem of \cite{wendl2010-compactness}
(reviewed above as Theorem \ref{t:spec-cpct}),
consists of stable,
nicely-embedded, non-nodal pseudoholomorphic buildings.
The next lemma shows that the latter alternative in fact always holds.

\begin{lemma}\label{l:non-compactness}
Every element of $\overline{\M_{p,\pm}}\setminus\M_{p,\pm}$
is a 
stable,
nicely-embedded, non-nodal pseudoholomorphic building
with at least two nontrivial components and at least two levels.
\end{lemma}

\begin{proof}
We prove this for $\M_{p, +}$.  The proof for $\M_{p, -}$ is identical.

We first argue that $\overline{\M_{p,\pm}}\setminus\M_{p,\pm}$ is either
a single element of $\M(C'_{p, +})/\R$ or is a subset of the boundary of $\M(C'_{p, +})/\R$.
As we discussed above, $\M(C'_{p, +})/\R$ has the structure of a smooth $1$-manifold
and, by construction, $\M_{p, +}$ is an embedded submanifold with boundary
which is diffeomorphic to a half-open interval.
Choosing an identification of $\M_{p, +}$ with an embedding
$i:[0, 1)\hookrightarrow \M(C'_{p, +})/\R$, we can identify
the set $\overline{\M_{p,\pm}}\setminus\M_{p,\pm}$
with limits of SFT-convergent sequences of the form
$i(a_{k})$ with $a_{k}\in [0, 1)$ an increasing sequence converging to $1$.
Assume that there exists some such sequence $a_{k}$ so that $i(a_{k})$ converges
to a curve $C_{\infty}\in\M(C'_{p, +})/\R$.
Then every open neighborhood of $C_{\infty}$ in $\M(C'_{p, +})/\R$ meets a set of the 
form $i\bp{(1-\ep, 1)}$ for some $\ep>0$.
Since $i$ is an embedding of a $1$-manifold in a $1$-manifold,
this allows us to conclude that for every sequence
$a_{k}\in [0, 1)$ with $a_{k} \to 1$, $i(a_{k})$ converges to $C_{\infty}$.  We conclude that
if $\overline{\M_{p,\pm}}\setminus\M_{p,\pm}$ contains an interior point $C_{\infty}$ of
$\M(C'_{p, +})$ then $\overline{\M_{p,\pm}}\setminus\M_{p,\pm}=\br{C_{\infty}}$.
Thus $\overline{\M_{p,\pm}}\setminus\M_{p,\pm}$ is either
a single element of $\M(C'_{p, +})/\R$ or is a subset of the boundary of $\M(C'_{p, +})/\R$
as claimed.

To show that $\overline{\M_{p,+}}\setminus\M_{p,+}$
can't be a curve in $\M(C'_{p, +})/\R$
we will argue by contradiction and
suppose to the contrary that
$\overline{\M_{p,+}}\setminus\M_{p,+}$ is an element of
$\M(C'_{p, +})/\R$.
In this case $\overline{\M_{p, +}}$ is diffeomorphic to a closed interval
embedded in $\M(C'_{p, +})/\R$ and thus
every sequence in $\M_{p,+}$
has a subsequence
converging to an element of $\M(C'_{p, +})/\R$.
We then claim that every point of $U_{p}'$
would have a curve from $\M_{p, +}$ passing through it and, since
$\M_{p, +}$ consists of foliating curves, that $U_{p}'$ is homeomorphic to $\R\times (\Sigma\setminus\Gamma)$.
Since, by construction, $U_{p}'$ is homeomorphic to $(\R\times (\Sigma\setminus\Gamma))$
with a point removed,
this contradiction would finish the proof.

Supposing then that every
sequence in $\M_{p,+}$ has a subsequence converging to an element of $\M(C'_{p, +})/\R$,
we claim that the set of points in $U_{p}'$ with a curve in
$\M_{p, +}\setminus\br{C'_{p, +}}$ passing through them
is nonempty and
both open and (relatively) closed in $U_{p}'$.
Both nonemptiness and openness follow immediately from Corollary \ref{c:mod-spaces-mod-r}; 
indeed, since the evaluation map $ev:\M_{1}(C'_{p, +})/\R\to M$ is an embedding
it has open image, and therefore the
set of points in $U_{p}'$ with curves from the open subset $\M_{p, +}$ of $\M(C'_{p, +})/\R$
passing through them is nonempty and open.
Closedness, in turn, follows from our assumption
that all sequences in $\M_{p, +}$ have subsequences converging to a point in
$\M(C'_{p, +})/\R$.
Indeed, let $p_{k}$ be a sequence in $U'_{p}$ converging to a point $p_{\infty}\in U_{p}'$,
and assume that for each point $p_{k}$ there is a curve $C_{k}\in\M_{p, +}$ passing through
$p_{k}$.  Then by assumption, some subsequence of the curves $C_{k}$ converges to a curve
$C_{\infty}\in\M(C'_{p, +})/\R$ passing through $p_{\infty}$
and, by the previous lemma, we have that $C_{\infty}\subset U_{p}'$
since it passes through the point $p_{\infty}\in U_{p}'$.
Since Corollary \ref{c:mod-spaces-mod-r} implies there is an open
set of points around $p_{\infty}$ having curves of $\M(C'_{p, +})/\R$ passing through
them, it follows that $C_{\infty}$ is in the same connected component of
the subset
\[
\br{C\in\M(C'_{p, +})/\R\,|\,C\subset U_{p}'}
\]
of $\M(C_{p,+})$
as the $C_{k}$
i.e.\ that $C_{\infty}\in\M_{p, +}$.
This completes the proof that
$\overline{\M_{p,+}}\setminus\M_{p,+}$
can't be an interior point of $\M(C'_{p, +})/\R$, and therefore must
be contained in the (SFT) boundary of $\M(C'_{p, +})/\R$.
As noted above, it follows from the main theorem of \cite{wendl2010-compactness}
that $\overline{\M_{p,+}}\setminus\M_{p,+}$
consists of stable, non-nodal, nicely-embedded pseudoholomorphic buildings.

Finally we show that any element of 
$\overline{\M_{p,+}}\setminus\M_{p,+}$
has at least two levels and at least two nontrivial components.
We first note that an element of $\overline{\M_{p,+}}\setminus\M_{p,+}$
is connected since it is the limit of connected curves.
Supposing such an element has only one level, then that level must have at least
two components (or else we would just have a curve in $\M(C'_{p, +})/\R$).
But, since a height-$1$ pseudoholomorphic building with at least two components
and no nodes must be disconnected, this is a contradiction.
Given that an element of $\overline{\M_{p,+}}\setminus\M_{p,+}$ has at least
two levels, stability implies that each level must have at least one nontrivial component.
\end{proof}

We next seek to show that the
sets
$\overline\M_{p, +}\setminus\M_{p, +}$ and $\overline\M_{p, -}\setminus\M_{p, -}$
each contain a single
non-nodal, nicely-embedded pseudoholomorphic building,
and that these buildings consist of
exactly one of the planes
$P^{\pm}$ and precisely one other nicely-embedded pseudoholomorphic curve $Z_{p}$
with $\gamma_{0}$
as a negative asymptotic limit.

\begin{proposition}\label{p:m-p-boundary-building}
Let $\Gamma_{p}^{-}$ denote the set of negative punctures of $C_{p, \pm}$ and
assume that at $z\in\Gamma_{p}^{-}$, $C_{p, +}$ is asymptotic to the orbit
$\gamma_{z}$.
Then, there exists a pseudoholomorphic curve $Z_{p}\in\M(M', \lambda')/\R$
so that 
\[
\overline{\M_{p, +}}\setminus\M_{p, +}
=Z_{p}\odot(P^{+}\amalg_{z\in\Gamma_{p}^{-}}\R\times \gamma_{z})
\]
and
\[
\overline{\M_{p, -}}\setminus\M_{p, -}
=Z_{p}\odot(P^{-}\amalg_{z\in\Gamma_{p}^{-}}\R\times \gamma_{z}).
\]
\end{proposition}
We note that since, as previously observed,
the main theorem of
\cite{wendl2010-compactness} (see Theorem \ref{t:spec-cpct} above)
implies that any curve in $\overline{\M_{p, \pm}}\setminus\M_{p, \pm}$
must be a nicely-embedded building, it follows immediately from this proposition that
$Z_{p}$ is embedded in $M$
and disjoint from $P^{\pm}$.

We prove the proposition in a series of lemmas.
In the following we let
\[
C_{\infty, +}\in\overline{\M_{p, +}}\setminus\M_{p, +}
\]
denote one of the nicely-embedded, non-nodal pseudoholomorphic buildings
given by Lemma \ref{l:non-compactness} above.

\begin{lemma}
Let $\gamma$ be a simple periodic orbit and assume that
the $m$-fold cover
$\gamma^{m}$ is a positive (resp.\ negative) asymptotic limit of $C'_{p, +}$.
Then there exists a nontrivial component
of $C_{\infty, +}$ with $\gamma^{m}$ as a positive (resp.\ negative) asymptotic limit.
\end{lemma}

\begin{proof}
For simplicity we assume that $\gamma^{m}$ is a positive asymptotic limit of $C'_{p, +}$.
The argument in the case that it is a negative asymptotic limit is identical.
Since $C_{\infty, +}$ can be written as the SFT-limit of a sequence of curves
$C_{k}\in\M_{p, +}$, each homotopic to
$C'_{p, +}$, it follows from the definition of SFT-convergence
(see Proposition \ref{p:compactness-local-C-infinity})
that the
positive asymptotic limits of the top-most level of $C_{\infty, +}$ agree with the
positive asymptotic limits of $C'_{p, +}$ and, similarly, the negative asymptotic limits
of the bottom-most level of $C_{\infty, +}$ agree with the
negative asymptotic limits of $C'_{p, +}$.
Thus there is some component of the top level
of $C_{\infty, +}$ with $\gamma^{m}$ as a positive asymptotic limit.
If this component is nontrivial there is no more to prove.  If not, then  the component with $\gamma^{m}$ as a positive asymptotic limit
must be a trivial cylinder and, thus, there is some component on the next level
down with
$\gamma^{m}$ as a positive asymptotic limit.
Repeating this argument we find either at least one component of $C_{\infty, +}$
with $\gamma^{m}$ as a positive puncture, or we can conclude that there is a trivial cylinder
over $\gamma^{m}$ on the lowest level of $C_{\infty, +}$ and thus that $\gamma^{m}$
is also negative asymptotic limits of $C'_{p, +}$ so that
$(\gamma, m, m)$ is a bidirectional orbit of $C'_{p, +}$.
Recalling that $C'_{p, +}$ satisfies $C'_{p, +}*C'_{p, +}=0$ and $\ind(C'_{p, +})=2$ it
follows Theorem \ref{t:foliating-curves} that all asymptotic limits of $C'_{p, +}$, and specifically
$\gamma^{m}$, have
odd Conley--Zehnder index.
But, since $C'_{p, +}$ is nicely embedded, it follows from Lemma \ref{l:bidirectional-orbits}
that $\gamma^{m}$ must have even Conley--Zehnder index if it were a bidirectional orbit of
$C'_{p, +}$.  This contradiction completes the proof.
\end{proof}

\begin{lemma}\label{l:compactness-step1}
The periodic orbit $(\gamma_{0}, 1, 1)$ is the only possible
bidirectional asymptotic limit (see Definition \ref{d:bidirectional-limit} above)
of $C_{\infty, +}$.  Moreover, if $\gamma_{0}$ appears as a limit
of a component of $C_{\infty, +}$, $\gamma_{0}$ is necessarily a bidirectional orbit.
\end{lemma}

\begin{proof}
Since $C_{\infty, +}$ is the limit of a sequence of curves contained entirely within the open
set $U_{p}'$, it follows from the definition of SFT-convergence
(see Proposition \ref{p:compactness-local-C-infinity})
that any periodic orbit appearing
as an asymptotic limit of a component of $C_{\infty, +}$ must be contained in
the closure $\bar U_{p}'$ of $U_{p}'$.
Moreover, Lemma \ref{l:no-orbits} tells us that there are no orbits contained within
$U_{p}'$ so any periodic orbit within the closure $\bar U_{p}'$ must touch the boundary.
However, by construction, the boundary of $U_{p}'$ consists of
the pseudoholomorphic curves
$C'_{p, \pm}$ and $P^{\pm}$
and the periodic orbits which are asymptotic limits of these curves,
and any flow line touching the curves $C'_{p, \pm}$ of $P^{\pm}$ necessarily passes through
points outside of $U_{p}'$.
We thus conclude that any orbit
contained in $\bar U_{p}'$ is contained within the boundary.
Since the periodic orbits in the boundary of $U_{p}'$ are the
asymptotic limits of $C'_{p, +}$ along with $\gamma_{0}$, and since
every asymptotic limit of $C'_{p, +}$ is odd and,
according to the previous lemma,
occurs as an asymptotic limit of 
some nontrivial component of
$C_{\infty, +}$ with the same covering numbers, it follows
from
Lemma \ref{l:bidirectional-orbits}
that the only possible
bidirectional limit is $\gamma_{0}$,
which can only occur simply covered since $\gamma_{0}$
is even.

Next, we argue that
if $\gamma_{0}$ (or in fact any orbit other than the limits of $C'_{p, +}$)
appears as an asymptotic limit of a component
of $C_{\infty, +}$, then it must be a bidirectional limit.
Recall that since $C_{\infty, +}$ is a limit of pseudoholomorphic spheres,
its structure can be modeled by a tree with one vertex for each component,
and an edge for each periodic orbit connecting adjacent levels (or, in general, 
for each node, but we know there are none here).
Moreover, the positive asymptotic limits of the top level of $C_{\infty, +}$
and the negative asymptotic limits of the bottom level of $C_{\infty, +}$
agree with those of $C'_{p, +}$.
Assuming then that $\gamma_{0}$
appears as an asymptotic limit of a component of $C_{\infty, +}$
but is not a bidirectional limit, 
we can conclude that either 
the only curves in $C_{\infty, +}$ having $\gamma_{0}$ as
a positive limit are trivial cylinders, or the only curves in $C_{\infty, +}$
having $\gamma_{0}$ as a negative limit are trivial cylinders.
In either case, we could, by following a path of
vertices corresponding to trivial cylinders,
conclude that $\gamma_{0}$ is 
either a positive asymptotic limit of the top level
or a negative asymptotic limit of the bottom level.
This contradicts the fact that the positive limits of the top level and the negative limits
of the bottom level agree with those of $C'_{p, +}$.
\end{proof}

\begin{lemma}\label{l:positive-puncture-hyperbolic}
All components of $C_{\infty, +}$ with $\gamma_{0}$ as a positive asymptotic limit must be equal to
either $P^{+}$ or $P^{-}$.
\end{lemma}

\begin{proof}
This follows from Theorem \ref{t:dir-of-approach},
\cite[Theorem 2.4]{siefring2011}/Theorem \ref{t:no-isect},
and Lemma \ref{l:building-isect}.
Indeed if $C_{k}$ is a sequence in $\M_{p, +}$ converging to $C_{\infty, +}$
then $P^{\pm}*C_{k}=P^{\pm}*C'_{p, +}=0$ by homotopy invariance of the intersection number.
Thus Lemma \ref{l:building-isect} allows us to conclude that any component
of $C_{\infty, +}$ is either identical with or disjoint from $P^{\pm}$.

Now, if there were a component $Z$ of $C_{\infty, +}$ distinct from $P^{\pm}$
with $\gamma_{0}$ as a positive puncture,
then $Z$ must approach a simple cover of $\gamma_{0}$ since,
according to the previous lemma,
$\gamma_{0}$
would have to be a bidirectional
limit of $C_{\infty, +}$.
Condition (3c) of Theorem \ref{t:no-isect} then
tells us that $Z$ must approach with the same winding as $P^{\pm}$.
But since $P^{+}$ and $P^{-}$ approach $\gamma_{0}$ in opposite directions
with extremal winding,
$Z$ necessarily approaches $\gamma_{0}$ with extremal winding and thus,
according to Lemma \ref{l:same-or-opposite},
in the same direction
as either $P^{+}$ or $P^{-}$.
Theorem \ref{t:dir-of-approach} then lets us conclude that $Z$ intersects either
$P^{+}$ or $P^{-}$, which contradicts the
fact from the previous paragraph
that such a $Z$ must be disjoint from $P^{+}$ and $P^{-}$.
\end{proof}

\begin{lemma}\label{l:unique-component}
There is precisely one nontrivial component of $C_{\infty, +}$ not equal to $P^{+}$ or $P^{-}$.
\end{lemma}

\begin{proof}
We first show that there is at most one nontrivial component of $C_{\infty, +}$ distinct from
$P^{+}$ and $P^{-}$.
Recall that the building $C_{\infty, +}$, being the limit of
spheres, can be modeled by a tree with vertices
corresponding to components of the building and edges corresponding to
periodic orbits connecting adjacent levels
(or, in general, nodes, but we know there are none in this case).
Moreover, we know that all components of the building are either nicely-embedded curves
or trivial cylinders.
Assuming there are two or more nontrivial components in $C_{\infty, +}$ which 
are distinct from $P^{+}$ and $P^{-}$,
we can find
a sequence of components $(C_{1}, \dots, C_{n})$ of
$C_{\infty, +}$ corresponding to distinct, adjacent vertices in the modeling tree,
with $C_{1}$ and $C_{n}$ nontrivial and not equal to $P^{+}$ or $P^{-}$. 
Since we assume all the $C_{i}$ correspond to distinct vertices,
we can conclude
none of the $C_{i}$ with $i\notin\br{1, n}$ are planes since planes correspond to
univalent vertices in the tree modeling $C_{\infty, +}$.
Thus all elements of the sequence
$(C_{1}, \dots, C_{n})$
are either trivial cylinders, or nontrivial components
distinct from $P^{+}$ and $P^{-}$.
Moreover, by truncating the sequence
if necessary, we can assume without loss of generality
that only $C_{1}$ and $C_{n}$ are nontrivial and that all
$C_{i}$ for $1<i<n$ are trivial cylinders.
However, since a sequence of adjacent trivial cylinders
in the holomorphic building must be cylinders over the same periodic orbit $\gamma$,
this allows us to conclude that
either $C_{1}$ has a positive puncture limiting to $\gamma$ and $C_{n}$ has a negative puncture limiting to
$\gamma$,
or that $C_{1}$ has a negative puncture limiting to $\gamma$ and $C_{n}$ has a positive
puncture limiting to $\gamma$.  Thus $\gamma$ is a bidirectional orbit, so
Lemma \ref{l:compactness-step1} tells us that $\gamma=\gamma_{0}$.  However, since
Lemma \ref{l:positive-puncture-hyperbolic} tells us that the planes
$P^{+}$ and $P^{-}$ are the only possible components of the building having $\gamma_{0}$
as a positive asymptotic limit,
this contradicts our assumption that $C_{1}$ and $C_{m}$ are
distinct from $P^{\pm}$.  This completes the argument that there is at most one component
of $C_{\infty, +}$ distinct from the $P^{\pm}$.

We next argue there is at least one nontrivial component of $C_{\infty, +}$
distinct from $P^{+}$ and $P^{-}$.
If there are no nontrivial components other than $P^{+}$ and $P^{-}$, then
every component of the building is either equal to $P^{+}$, $P^{-}$ or a trivial cylinder.
Since $C_{\infty, +}$ is, as the limit of connected curves, a connected building, this would then
let us conclude that $\gamma_{0}$ is the only asymptotic limit of components of the building
$C_{\infty, +}$
However, since
$C_{\infty, +}$ is a limit of curves in $\M(C'_{p, +})$, the properties of SFT-convergence
(see Proposition \ref{p:compactness-local-C-infinity}) allow us to conclude
that the asymptotic limits of the top and bottom levels of $C_{\infty, +}$
agree with those of $C'_{p, +}$.
Since $C'_{p, +}$ has only odd asymptotic limits and $\gamma_{0}$ is even,
this is a contradiction.
Thus there must be at least one nontrivial component in the building
not equal to $P^{+}$ or $P^{-}$.
\end{proof}

\begin{lemma}\label{l:flow-segment-intersection}
Let $C\in\M_{p, +}\setminus\br{C'_{p, +}}$.  Then the intersection numbers of the curve $C$
with the flow segments $\gamma_{p'_{\pm}}$ are well defined and given by
\[
\gamma_{p'_{+}}\cdot C=1 \qquad \gamma_{p'_{-}}\cdot C=0.
\]
\end{lemma}

\begin{proof}
To see this, we first recall the flow segments have boundary in
the boundary $U_{p}'$ and interior in the interior of $U_{p}'$.
Since curves $C\in\M_{p, +}\setminus\br{C'_{p, +}}$ limit at the punctures
to periodic orbits disjoint from $\gamma_{p'_{\pm}}$
and, by definition of $\M_{p, +}$,
are contained entirely within $U_{p}'$, it follows that
the intersection number of such a curve with the flow segments $\gamma_{p'_{\pm}}$ is well defined.
Moreover, since the image of $C'_{p, +}$ compactifies to a map that is disjoint from
$\gamma_{p'_{-}}$, it follows that curves in $\M_{p, +}$ nearby
to $C'_{p, +}$ are also disjoint from $\gamma_{p'_{-}}$ and thus
\[
\gamma_{p'_{-}}\cdot C=0
\]
for $C\in\M_{p, +}\setminus\br{C'_{p, +}}$.
On the other hand, it follows from Corollary \ref{c:flow-loc-diff}
---
specifically that a flow line passing through $C'_{p, +}$ gives a local diffeomorphism
with a neighborhood of $C'_{p, +}$ in $\M(C'_{p, +})$
---
that curves
in $\M_{p, +}$ nearby to $C'_{p, +}$ have a single transverse intersection with
$\gamma_{p'_{+}}$.
Thus
\[
\gamma_{p'_{+}}\cdot C=1
\]
for $C\in\M_{p, +}\setminus\br{C'_{p, +}}$ as claimed.
\end{proof}

We will denote by $Z_{p}$ the nontrivial component of $C_{\infty, +}$
guaranteed by Lemma \ref{l:unique-component}.

\begin{lemma}\label{l:compactness-step3}
The pseudoholomorphic building
$C_{\infty, +}$ is a height-$2$ building with
$Z_{p}$ on the top level, and $P^{+}$ and trivial cylinders over the negative orbits of
$C'_{p, \pm}$ on the bottom level, i.e.\
\[
C_{\infty, +}
=Z_{p}\odot(P^{+}\amalg_{z\in\Gamma_{p}^{-}}\R\times \gamma_{z})
\]
where $\Gamma_{p}^{-}$ is the set of negative punctures of $C'_{p, \pm}$
and $\gamma_{z}$ is the asymptotic limit of $C'_{p, \pm}$ at $z\in\Gamma_{p}^{-}$.
\end{lemma}

\begin{proof}
We argue that the number of times that $P^{\pm}$ appears as a component of the building
$C_{\infty, +}$ bounds 
the intersection number
$\gamma_{p'_{\pm}}\cdot C$ for $C\in\M_{p, +}\setminus\br{C'_{p, +}}$ from below.
This with the previous lemma will show that $P^{+}$ can appear at most once 
as a component of $C_{\infty, +}$, while $P^{-}$ can't appear.
Assume there
are components $D_{i}$
of $C_{\infty, +}$ with $D_{i}=P^{+}$ (mod the $\R$-action),
and choose a parametrization $u:\C\to M'$ of $P^{+}$ with $u$ mapping $0\in\C$ to the
intersection of $\gamma_{p'_{+}}$ with $P^{+}$.
Then according to Proposition \ref{p:compactness-local-C-infinity}, if
$C_{k}=[\Sigma_{k}, j_{k}, \Gamma_{k}, a_{k}, u_{k}]\in\M_{p, +}$
is a sequence which converges in the sense of \cite{behwz} to $C_{\infty, +}$, then there
is a sequence of holomorphic embeddings
\[
\phi_{k}=\amalg_{i}\phi_{k, i}:\amalg_{i}\DD\to \Sigma_{k},
\]
with $\DD=\br{z\in\C\,|\, |z|\le 1}$ the unit disk in $\C$,
so that $u_{k}\circ\phi_{k, i}$ converges in $C^{\infty}$ to
$u|_{\DD}$.  Since $u_{k}\circ\phi_{k}$ has image in $U_{p}'$ and $\gamma_{p'_{+}}$
meets $P^{+}$ transversely, it follows from the $C^{\infty}$ convergence
that the pseudoholomorphic disks
$u_{k}\circ\phi_{k, i}$ intersect $\gamma_{p'_{+}}$ transversely for sufficiently large $k$.
This shows that for large $k$, $C_{k}$ has at least one transverse intersection with
$\gamma_{p'_{+}}$ for each component of $C_{\infty, +}$ equal to $P^{+}$.
However, since $\R\times \gamma_{p'_{+}}$ is pseudoholomorphic, every intersection of
$\gamma_{p'_{+}}$ with a curve $C_{k}$ contributes positively to the intersection number
$\gamma_{p'_{+}}\cdot C_{k}$, so the number
of components of $C_{\infty, +}$ equal to $P^{+}$ is bounded above by
$\gamma_{p'_{+}}\cdot C_{k}=1$.
An analogous argument shows that the number of components of $C_{\infty, +}$ equal to
$P^{-}$ is bounded above by $\gamma_{p'_{-}}\cdot C_{k}=0$.
Thus $P^{+}$ can appear at most once as a component of $C_{\infty, +}$ and
$P^{-}$ does not appear.

Finally, we observe that we have shown in
Lemmas \ref{l:non-compactness} and \ref{l:compactness-step1}
that $C_{\infty, +}$ has at least two levels,
at least two nontrivial components, and precisely one nontrivial component $Z_{p}$ distinct from
$P^{+}$ and $P^{-}$.  This combined with the results of the previous paragraph then shows that
$C_{\infty, +}$ has precisely two nontrivial components: $Z_{p}$ and $P^{+}$.
Moreover, by stability, $C_{\infty, +}$ must be a
height-$2$ building with $Z_{p}$ on the top level and $P^{+}$ along with
cylinders over the other negative orbits of $Z_{p}$ on the bottom level.
\end{proof}

\begin{remark}
We remark that Theorem \ref{t:dir-of-approach}
can also be used to bound the number of planes appearing in $C_{\infty, +}$,
but an additional argument is then needed to show that the unique plane appearing
in $C_{\infty, +}$ is $P^{+}$ and not $P^{-}$.
Indeed,
if the total number of times $P^{+}$ and $P^{-}$ appeared as components of
$C_{\infty, +}$ were greater than one,
we could argue that
the unique nontrivial component of $C_{\infty, +}$
distinct from $P^{\pm}$
guaranteed from Lemma \ref{l:unique-component}
must have multiple negative ends approaching $\gamma_{0}$.  Moreover,
these ends would have to be disjoint in $M$ since all components of $C'_{p, +}$
are nicely embedded.  We could then argue, as in Lemma \ref{l:m-p-plus-building}
below, that these ends would have to approach $\gamma_{0}$ in the same direction.
Theorem \ref{t:dir-of-approach} would then yield a contradiction,
so we could conclude that there is at most a single plane in the building $C_{\infty, +}$.
\end{remark}

\begin{remark}
The work leading up to the proof of
Lemma \ref{l:compactness-step3}
can be simplified somewhat
if one applies \cite[Theorem 1.10]{hwz:prop3} to perturb $J'$ slightly so that
all moduli spaces of embedded curves are smooth manifolds of the appropriate
dimension as predicted by the index formula \eqref{e:index-formula}.
Since automatic transversality holds for the planes $P^{+}$ and $P^{-}$ 
such a perturbation of $J'$ could be carried out while maintaining the existence of planes with the
properties we need.
Given such a $J'$, it would follow immediately from
\cite[Theorem 2]{wendl2010-compactness}
that $C_{\infty, +}$ is a height-$2$ building with precisely one nontrivial component
on each level.
It would then remain to argue, as we have above, that the nontrivial component 
of the lower level is $P^{+}$.
However,
we do not need to assume that such a generic $J'$ has been chosen since
our argument shows that 
the theorem holds for
any $J'$ for which there exist planes $P^{+}$ and $P^{-}$ with the prescribed properties.
\end{remark}

Finally, we conclude that $\overline{\M_{p, +}}\setminus\M_{p, +}$ consists
of a single building, completing the proof of the first statement in
Proposition \ref{p:m-p-boundary-building}.

\begin{lemma}\label{l:m-p-plus-building}
With $C_{\infty, +}$, $Z_{p}$, $\Gamma_{p}^{-}$, and $\gamma_{z}$ as above, we have that
\[
\overline{\M_{p, +}}\setminus\M_{p, +}=
C_{\infty, +}
=Z_{p}\odot(P^{+}\amalg_{z\in\Gamma_{p}^{-}}\R\times \gamma_{z}).
\]
\end{lemma}

\begin{proof}
Let $C_{\infty, +}'\in \overline{\M_{p, +}}\setminus\M_{p, +}$.
We seek to show that $C_{\infty, +}'=C_{\infty, +}$.
The argument above applies to show that
$C_{\infty, +}=Z_{p}'\odot(P^{+}\amalg_{z\in\Gamma_{p}^{-}}\R\times \gamma_{z})$
for some nicely-embedded curve $Z_{p}'$.
It only remains to show that $Z_{p}'=Z_{p}$.
We first argue that $Z_{p}$ and $Z_{p}'$ must approach $\gamma_{0}$
in the same direction.
To see this, we first note that
since $C_{\infty, +}$ and $C_{\infty, +}'$ are nicely-embedded buildings
and thus $Z_{p}$ and $Z_{p}'$ are disjoint from $P^{\pm}$, it follows from
condition (3c) of Theorem \ref{t:no-isect} that
$Z_{p}'$ and $Z_{p}$
approach $\gamma_{0}$ with winding equal to that of $P^{\pm}$,
and thus equal to $\fl{\mu^{\Phi}(\gamma_{0})/2}=\mu^{\Phi}(\gamma_{0})/2$.
Thus, according to Lemma \ref{l:same-or-opposite},
$Z_{p}'$ and $Z_{p}$ approach $\gamma_{0}$ in either the same or the opposite direction,
with approach governed by a nonzero multiple
of an eigenvector $e_{+}$ of $\A_{\gamma_{0}, J}$
with smallest possible positive eigenvalue.
However, the boundary of $U_{p}'$ at $\gamma_{0}$,
being given nearby by the planes $P^{+}$ and $P^{-}$,
is tangent to the largest negative eigenspace $\vspan\br{e_{-}}$
of $\A_{\gamma_{0}, J}$, since eigenvectors with largest negative
eigenvalue govern the approach of $P^{+}$ and $P^{-}$.
Since,
according to \cite[Lemma 3.5]{hwz:prop2},
eigenvectors with the same winding and different eigenvalue are pointwise
linearly independent,
$e_{+}$ 
is nowhere tangent to the boundary of $U_{p}'$.
Thus if $e_{+}$ points into $U_{p}'$, $-e_{+}$ points into $U_{q}'$ and vice versa.
Since, $Z_{p}'$ and $Z_{p}$ both approach $\gamma_{0}$ from within
$U_{p}'$ we can thus conclude that $Z_{p}'$ and $Z_{p}$ approach $\gamma_{0}$
in the same direction.

We next
claim the fact that $Z_{p}'$ and $Z_{p}$ approach $\gamma_{0}$ in the same direction
leads to a contradiction unless $Z_{p}'=Z_{p}$.
Let $C_{k}$ and $C_{k}'$ be sequences in $\M_{p, +}$ converging
respectively to $C_{\infty, +}$ and $C_{\infty, +}'$.
Then, for any $j$ and $k$, we have that
\[
C_{j}*C_{k}'=C'_{p, +}*C'_{p, +}=0
\]
by homotopy invariance of the holomorphic intersection product.
It then follows from Lemma \ref{l:building-isect} that $Z_{p}'$ and $Z_{p}$
are either identical or disjoint.  However, since $Z_{p}'$ and $Z_{p}$ approach
$\gamma_{0}$ in the same direction, it follows from Theorem \ref{t:dir-of-approach}
that $Z_{p}'$ and $Z_{p}$ must intersect.  We thus arrive at a contradiction unless
$Z_{p}'=Z_{p}$.
\end{proof}

We now complete the proof of Proposition \ref{p:m-p-boundary-building}.

\begin{lemma}
The set
$\overline{\M_{p, -}}\setminus\M_{p, -}$
consists of a single height-$2$ building with
$Z_{p}$ on the top level, and $P^{-}$ and trivial cylinders over the negative orbits of $C'_{p, \pm}$ on the bottom
layer, i.e.\
\[
\overline{\M_{p, -}}\setminus\M_{p, -}
=Z_{p}\odot(P^{-}\amalg_{z\in\Gamma_{p}^{-}}\R\times \gamma_{z})
\]
where $\Gamma_{p}^{-}$ is the set of negative punctures of $C'_{p, \pm}$
and $\gamma_{z}$ is the asymptotic limit of $C'_{p, \pm}$ at $z\in\Gamma_{p}^{-}$.
\end{lemma}

\begin{proof}
Let $C_{\infty, -}\in \overline{\M_{p, -}}\setminus\M_{p, -}$.
An analogous argument
to that in Lemmas \ref{l:compactness-step1}-\ref{l:compactness-step3}
as above gives us that
$C_{\infty, -}=Z_{p}'\odot(P^{-}\amalg_{z\in\Gamma_{p}^{-}}\R\times \gamma_{z})$
for some nicely-embedded curve $Z_{p}'$.
It remain only to show that $Z_{p}'=Z_{p}$.
This follows from an argument analogous to that in Lemma \ref{l:m-p-plus-building} above.
Indeed, we can argue exactly as in Lemma \ref{l:m-p-plus-building} that
$Z_{p}'$ and $Z_{p}$ approach $\gamma_{0}$ in the same direction.
Then, with $C_{k}$ and $C_{k}'$ sequences respectively in $\M_{p, +}$ and $\M_{p, -}$
converging respectively to $C_{\infty, +}$ and $C_{\infty, -}$,
we have that
\[
C_{j}*C_{k}'=C'_{p, +}*C'_{p, -}=0
\]
by homotopy invariance of the holomorphic intersection product.
We can thus again apply Lemma \ref{l:building-isect} and Theorem \ref{t:dir-of-approach}
to arrive at a contradiction unless $Z_{p}'=Z_{p}$.
\end{proof}

Next, we define $\M_{q,\pm}$ analogously to $\M_{p, \pm}$,
that is, we define
$\M_{q, \pm}\subset\M(C_{q, \pm})/\R$
to be
the connected
component of
\[
\br{C_{q, \pm}}\cup \br{C\in\M(C_{q, \pm})/\R\,|\, C\subset U_{q}'}
\]
containing $C_{q, \pm}$.
An analogous argument to above shows the following.

\begin{proposition}\label{p:m-q-boundary-building}
There exists a pseudoholomorphic curve $Z_{q}\in\M(\lambda', J')/\R$
which is embedded in $M$ and disjoint from $P^{\pm}$
so that
\[
\overline{\M_{q, +}}\setminus\M_{q, +}
=Z_{q}\odot(P^{-}\amalg_{z\in\Gamma_{q}^{-}}\R\times \gamma_{z})
\]
and
\[
\overline{\M_{q, -}}\setminus\M_{q, -}
=Z_{q}\odot(P^{+}\amalg_{z\in\Gamma_{q}^{-}}\R\times \gamma_{z})
\]
where $\Gamma_{q}^{-}$ is the set of negative punctures of $C'_{q, \pm}$
and $\gamma_{z}$ is the asymptotic limit of $C'_{q, \pm}$ at $z\in\Gamma_{q}^{-}$.
\end{proposition}

Finally, to complete the argument, we need to show that the curves
from the old foliation together with the curves from the compactified
moduli spaces from above give a foliation of the surgered manifold
with the same energy as the original foliation.
We work with the definition of finite energy foliation given by
Corollary \ref{c:fol-alt-def}.
More precisely we
consider the collection of simple periodic orbits $B'\subset M'$ defined
by
\[
B'=i(B)\cup\br{\gamma_{0}}
\]
where $B$ is the set of periodic orbits with covers appearing as asymptotic limits of the
original foliation $\F$,
and we 
define a collection of curves $\F'\subset\M(\lambda', J')/\R$
by including:
\begin{itemize}
\item The curves in the moduli spaces $\M_{p, +}$, $\M_{p, -}$, $\M_{q, +}$, and $\M_{q, -}$.

\item The curves $\br{Z_{p},\, Z_{q},\, P^{+},\, P^{-}}$ constructed above.

\item The push forward via the inclusion $i:M\setminus\br{p, q}\to M'$
of any curve in the original foliation $\F$ which
lies in the closure of the complement of the regions
$U_{p}$ and $U_{q}$,
i.e.\ if $C=[\Sigma, j, \Gamma, a, u]\in\F$ and $C\subset M\setminus (U_{p}\cup U_{q})$
then we define $i(C)=[\Sigma, j, \Gamma, a, i\circ u]$.
Since $i^{*}J=J'$ outside of $U_{p}$ and $U_{q}$ it follows that $i(C)$
is a pseudoholomorphic curve in $M'$.
\end{itemize}

We now argue that $\F'$ so defined is a stable finite energy foliation for $M'$.
We need to show that there is a unique curve from $\F'$ through every point of
$M'\setminus B'$,
that the index of any nontrivial curve in $\F'$ is $1$ or $2$, that the intersection numbers
between any two nontrivial curves in $\F'$ vanish, and that
$E(\F')=E(\F)$.

We first address the fact that the energies of the two collections of curves are the same.

\begin{lemma}
With $\F'\subset\M(\lambda', J')/\R$ the collection of curves defined above,
we have that
\[
E(\F'):=\sup_{C\in\F'}E(C)=\sup_{C\in\F}E(C)=:E(\F).
\]
\end{lemma}

\begin{proof}
We recall from the proof of Lemma \ref{l:finite-index-0-and-1} that
the energy $E(C)$ of a curve
$C=[\Sigma, j, \Gamma, da, u]\in\M(\lambda, J)/\R$,
defined by \eqref{e:energy-def},
is given by the sum of the periods of the
orbits that are asymptotic limits of the positive punctures of $C$.
Since all curves $C\in\F$ have asymptotic limits in the region where
$i^{*}\lambda'=\lambda$ and, by construction, either
satisfy $i(C)\in\F'$ or are homotopic to a curve satisfying this,
every curve in $\F$ has energy equal to that of some curve in $\F'$.
We conclude $E(F)\le E(\F')$.

Conversely, every curve in $C'\in\F'\setminus\br{Z_{p}, Z_{q}, P^{+}, P^{-}}$
is, by construction, either of the form
$C'=i(C)$ for some $C\in\F$ or is homotopic to a curve of this form.
Thus every curve $C'\in\F'\setminus\br{Z_{p}, Z_{q}, P^{+}, P^{-}}$
has energy equal to that of some curve in $\F$ and we conclude that
$E(C')\le E(\F)$.  Moreover, by
Propositions \ref{p:m-p-boundary-building} and \ref{p:m-q-boundary-building},
$Z_{p}$ and $Z_{q}$ have positive
punctures identical respectively to
$C'_{p, \pm}=i(C_{p, \pm})$ and $C'_{q, \pm}=i(C_{q, \pm})$
and thus $E(Z_{p})\le  E(\F)$ and $E(Z_{q})\le E(\F)$ as well.
Finally, we recall from the proof of Lemma \ref{l:finite-index-0-and-1}
that the $d\lambda$-energy of a curve, defined by \eqref{e:dlambda-energy-def},
is always nonnegative and
is given by the difference between the sums of periods of the positive asymptotic limits
and those of the negative asymptotic limits.
Thus, if follows immediately from
Propositions \ref{p:m-p-boundary-building}
that $E(P^{\pm})\le E(Z_{p})\le E(\F)$.
We conclude that $E(F')\le E(\F)$ and, with the previous paragraph, this completes the proof.
\end{proof}

We next address the fact that all nontrivial curves in $\F'$ have index
$1$ or $2$.

\begin{lemma}
Let $C\in\F'$.  Then $\ind(C)\in\br{1, 2}$.
\end{lemma}

\begin{proof}
Except for
$C\in\br{Z_{p}, Z_{q}, P^{+}, P^{-}}$,
this is immediate
from the fact that all nontrivial curves in $\F$ have index $1$ or $2$.
To see that $\ind(P^{\pm})=1$ we use the fact
$P^{\pm}*P^{\pm}=0$.
Then according to Theorem \ref{t:self-gin-zero}
\begin{align*}
\ind(P^{\pm})
&=\chi(S^{2})-\#\Gamma_{even} \\
&=2-1=1.
\end{align*}
The fact that $\ind(Z_{p})=\ind(Z_{q})=1$ then follows from
Propositions \ref{p:m-p-boundary-building} and \ref{p:m-q-boundary-building}.
Indeed, since the pseudoholomorphic buildings
$\overline{\M_{p, \pm}}\setminus\M_{p, \pm}$
and
$\overline{\M_{q, \pm}}\setminus\M_{q, \pm}$
have no nodes, the sum of the indices of the nontrivial components must
add to the index of a curve in $\M_{p, \pm}$ or $\M_{q, \pm}$.
Since such curves have index $2$, and we have just shown that $P^{\pm}$ have index
$1$, it follows that
$\ind(Z_{p})=\ind(Z_{q})=1$ as claimed.
\end{proof}

We next address the intersection numbers.  We start by showing that any two
distinct curves in $\F'$ are disjoint.

\begin{lemma}
Let $C_{1}$, $C_{2}\in\F'$ be distinct curves.  Then
$C_{1}$ and $C_{2}$ are disjoint.
\end{lemma}

\begin{proof}
Since any two nontrivial curves in the original foliation $\F$
have vanishing intersection number,
it follows that any two nontrivial curves in
\[
C_{1},\, C_{2}\in\F'\setminus\br{Z_{p}, Z_{q}, P^{+}, P^{-}}
\]
have vanishing intersection number.  Thus any two such
distinct curves are disjoint in $M$.
Moreover, since the curves $Z_{p}$, $Z_{q}$, $P^{+}$, and $P^{-}$
occur as components of limiting buildings of sequences of curves in
$\F'\setminus\br{Z_{p}, Z_{q}, P^{+}, P^{-}}$ we can immediately conclude from
Lemma \ref{l:building-isect} that any two distinct, nontrivial curves in $\F'$
are disjoint.
\end{proof}

\begin{lemma}\label{l:isect-zero}
Let $C_{1}$, $C_{2}\in\F'$.  Then
$C_{1}*C_{2}=0$.
\end{lemma}

\begin{proof}
As noted in the proof of the above lemma, this is immediate for any two curves
\[
C_{1},\, C_{2}\in\F'\setminus\br{Z_{p}, Z_{q}, P^{+}, P^{-}}.
\]
It remains to show that $C_{1}*C_{2}=0$ when one
or both of $C_{1}$, $C_{2}$ is equal to one of
$Z_{p}$, $Z_{q}$, $P^{+}$, or $P^{-}$.
We first observe that
$\ind(Z_{p})=\ind(Z_{q})=1$ and $Z_{p}$ and $Z_{q}$ each have
precisely one puncture asymptotic to an even orbit.
We thus have for $C\in\br{Z_{p}, Z_{q}}$ that
\[
\ind(C)-\chi(S^{2})+\#\Gamma_{even}=1-2+1=0,
\]
so it then follows from facts in \cite{hwz:prop2} 
that the bound in 
inequality \eqref{e:wind-infinity-inequality}
is achieved at each puncture of $Z_{p}$ and $Z_{q}$
(see also discussion preceding Lemma 2.6 in \cite{abb-cie-hof}
and equation 5.1 in \cite{siefring2011}).
Meanwhile, we know that the 
bound in inequality 
\eqref{e:wind-infinity-inequality}
is achieved at each puncture of
 every other curve in $\F'$ from Theorem \ref{t:self-gin-zero}.
Thus, by Corollary \ref{c:no-isect-gin-zero} and the previous lemma, we can conclude that
$C_{1}*C_{2}=0$ for any distinct $C_{1}$, $C_{2}\in\F'$.
Since we already know that $P^{\pm}*P^{\pm}=0$ by Theorem \ref{t:contact-connect-sum},
it remains to show $Z_{p}*Z_{p}=Z_{q}*Z_{q}=0$.
However, since $Z_{p}$ and $Z_{q}$ are embedded in $M'$ and, as observed above,
have extremal winding at each puncture, Corollary \ref{c:embedding-gin-zero}
implies that $Z_{p}*Z_{p}=Z_{q}*Z_{q}=0$.
This completes the proof that $C_{1}*C_{2}=0$ for any two curves
$C_{1}$, $C_{2}\in\F'$.
\end{proof}

It remains to show that there is a curve of $\F'$ through every point of $M'\setminus B'$.

\begin{lemma}
For every $x\in M'\setminus B'$
there is a curve $C\in\F'$ passing through $x$.
\end{lemma}

\begin{proof}
By construction,
since $i^{*}J'=J$ outside of $U_{p}\cup U_{q}$ and since the boundaries
of $U_{p}$ and $U_{q}$ are made up of curves in the foliation $\F$,
it suffices to show there
is a curve through every point of $U_{p}'\setminus Z_{p}$ and
$U_{q}'\setminus Z_{q}$.  We will prove this for $U_{p}'\setminus Z_{p}$.  The argument for
$U_{q}'\setminus Z_{q}$ is identical.

We first define a subset
$U_{p, +}'$ (resp.\ $U_{p, -}'$)  of $U_{p}'\setminus Z_{p}$ to be
the sets of points in $U_{p}'\setminus Z_{p}$
having a curve of $\M_{p, +}$ (resp.\ $\M_{p, -}$) passing through them.
Then $U_{p, +}'$ and $U_{p, -}'$ are each nonempty (by construction of $\M_{p, \pm}$),
open (by Corollary \ref{c:mod-spaces-mod-r}),
and (relatively) closed (by compactness and Proposition \ref{p:m-p-boundary-building}).
Thus, $U_{p, +}'$ and $U_{p, -}'$ each form a connected component of $U_{p}'\setminus Z_{p}$.
Since $U_{p}'$ is connected and $Z_{p}$ is an embedded submanifold, 
$U_{p}'\setminus Z_{p}$ has at most two connected components.
Thus, the proof is completed unless $U_{p, +}'=U_{p, -}'$.
However, if $U_{p, +}'=U_{p, -}'$ there is a point
$x'\in U_{p}'\setminus Z_{p}$ with curves $C_{+}\in\M_{p, +}$ and $C_{-}\in\M_{p, -}$
passing through $x'$.  Since
$C_{+}*C_{-}=C'_{p, +}*C'_{p, -}=0$ by homotopy invariance of the intersection number
and Theorem \ref{t:intpositivity},
we must have that $C_{+}=C_{-}$, so we have found a curve
belonging to both $\M_{p, +}$ and $\M_{p, -}$.
But we've shown in Lemma \ref{l:flow-segment-intersection} that
the
intersection numbers $\gamma_{p'_{\pm}}\cdot C$ 
of curves $C\in\M_{p, +}$ with the flow segments $\gamma_{p'_{\pm}}$ are well defined
and satisfy
\[
\gamma_{p'_{+}}\cdot C=1 \qquad \gamma_{p'_{-}}\cdot C=0,
\]
while an analogous argument shows that for curves $C\in\M_{p, -}$ the intersection numbers
$\gamma_{p'_{\pm}}\cdot C$ are well defined
and satisfy
\[
\gamma_{p'_{+}}\cdot C=0 \qquad \gamma_{p'_{-}}\cdot C=1.
\]
Thus, a curve belonging to both $\M_{p, +}$ and $\M_{p, -}$
would lead to a contradiction, and this completes the proof that there is at least
one curve of $\F'$ through every point of $M'\setminus B'$.
\end{proof}

\appendix
\section{Additional details for Section \ref{s:connect-sum}}\label{a:details}

In this section we collect some of the more straightforward but tedious computations
supporting claims made in Section \ref{s:connect-sum}.

\begin{lemma}\label{l:smoothstuff}
Consider $S^{2}$ equipped with polar coordinate
$\phi\in\R/2\pi\Z$ and azimuthal coordinate $\theta\in[0, \pi]$.
The following define smooth tensor fields on $S^{2}$:
\begin{enumerate}
\item $\sin^{2}\theta$
\item $\cos\theta$
\item $\sin\theta\,d\theta$
\item $\dph$ (and $\dph$ vanishes for $\theta\in\br{0, \pi}$)
\item $\sin^{2}\theta\,d\phi$
\item $\sin\theta\,\dth$ (and $\sin\theta\,\dth$ vanishes for  $\theta\in\br{0, \pi}$)
\end{enumerate}
\end{lemma}

\begin{proof}
Considering $S^{2}$ embedded as the unit sphere in $\R^{3}$, the smooth change of coordinates
on the upper and lower hemispheres obtained by projecting onto the
$xy$-plane is given by
\begin{align*}
x&=\sin\theta\cos\phi \\
y&=\sin\theta\sin\phi.
\end{align*}
In these coordinates we have
\[
\sin^{2}\theta=x^{2}+y^{2}
\]
which is clearly smooth.
Meanwhile
\[
\cos\theta=\pm\sqrt{1-\sin^{2}\theta}=\pm\sqrt{1-x^{2}-y^{2}}
\]
which is also smooth near $(x, y)=0$,
and hence the $1$-form
\[
-d(\cos\theta)=\sin\theta\,d\theta
\]
is also smooth.

Next we have that
\begin{align*}
\dph
&=x_{\phi}\,\dx+y_{\phi}\,\dy \\
&=-y\,\dx+x\,\dy
\end{align*}
which is smooth and vanishes when $x=y=0$
(i.e.\ when $\theta\in\br{0, \pi}$).
Meanwhile, using that the standard round metric $g$
on $S^{2}$ with total area $4\pi$ is given by
\[
g=d\theta\otimes d\theta+\sin^{2}\theta\,d\phi\otimes d\phi
\]
we have that
\[
g(\dph, \cdot)=\sin^{2}\theta\,d\phi
\]
so $\sin^{2}\theta\,d\phi$ is smooth since it is dual to a smooth vector field.
Similarly
\[
g(\sin\theta\,\dth, \cdot)=\sin\theta\,d\theta
\]
so $\sin\theta\,\dth$ is smooth since it is dual to a smooth $1$-form.
Moreover,
\[
g(\sin\theta\,\dth, \sin\theta\,\dth)=\sin^{2}\theta=x^{2}+y^{2}
\]
shows that $\sin\theta\,\dth$ vanishes when $x=y=0$
i.e.\ when $\theta\in\br{0, \pi}$.
\end{proof}

\begin{lemma}[Lemma \ref{l:connect-sum-pullback}]\label{l:connect-sum-pullback-app}
Consider the maps $\Phi_{\pm}:\R^{\pm}\times S^{2}\to \R^{3}\setminus\br{0}$
defined by
\[
\Phi_{\pm}(\rho, \phi, \theta)
=\pm(\rho\sin\theta\cos\phi, \rho\sin\theta\sin\phi, \rho^{3}\cos\theta).
\]
Then $\Phi_{+}$ and $\Phi_{-}$ are smooth diffeomorphisms satisfying
\[
\Phi_{\pm}^{*}\lambda_{\pm}=\rho^{2}\lambda_{1}
\]
with $\lambda_{+}$, $\lambda_{-}$, and $\lambda_{1}$ as defined
in Section \ref{s:connect-sum}.
\end{lemma}

\begin{proof}
We first claim that $\Phi_{\pm}$ is bijective.
Since the two maps differ by negation on $\R^{3}\setminus\br{0}$
it suffices to show that $\Phi_{+}$ is bijective.
We first observe that $\Phi_{+}$ maps the set
$\R^{+}\times\br{\theta=0, \pi}$
bijectively to complement of the origin on the $z$-axis.
We thus consider a point $p_{0}=(x_{0}, y_{0}, z_{0})\in\R^{3}\setminus\br{0}$ not in the $z$-axis,
and seek to find a unique solution to
\begin{align}
x_{0}&=\rho \sin\theta\cos\phi \label{e:phi-1} \\
y_{0}&=\rho \sin\theta\sin\phi \label{e:phi-2} \\
z_{0}&=\rho^{3}\cos\theta \label{e:phi-3}
\end{align}
with $(\rho, \phi, \theta)\in\R^{+}\times \R/2\pi\Z\times (0, \pi)$.
Squaring and summing \eqref{e:phi-1}-\eqref{e:phi-2} gives
\begin{equation}\label{e:phi-4}
x_{0}^{2}+y_{0}^{2}=\rho^{2}\sin^{2}\theta
\end{equation}
and the assumption that $p_{0}$ is not in the $z$-axis implies that $x_{0}^{2}+y_{0}^{2}>0$.
Combining \eqref{e:phi-4} and \eqref{e:phi-3} leads to
\[
\cot\theta\csc^{2}\theta=\frac{z_{0}}{(x_{0}^{2}+y_{0}^{2})^{3/2}}
\]
which has a unique solution with $\theta_{0}\in (0, \pi)$
since the derivative of $\cot\theta\csc^{2}\theta$ is everywhere negative
and $\lim_{\theta\to k\pi^{\pm}}=\pm\infty$ for $k\in\Z$.
Substituting this $\theta_{0}$ in \eqref{e:phi-3}
gives a unique $\rho_{0}>0$.
Finally substituting these values of $\rho_{0}$ and $\theta_{0}$
into equations \eqref{e:phi-1}-\eqref{e:phi-2} gives a
unique value of $\phi_{0}\in\R/2\pi\Z$ for which \eqref{e:phi-1}-\eqref{e:phi-3} are satisfied.
This completes the proof that $\Phi_{+}$ and $\Phi_{-}$ are bijective.

To show the $\Phi_{\pm}$ are diffeomorphisms, it
remains
to show that $\Phi_{\pm}$ are immersions.
Again it suffices to show this for $\Phi_{+}$.
For $\theta\notin\br{0, \pi}$ we have that
\[
D\Phi_{+}(\rho, \phi, \theta)=
\begin{bmatrix}
\sin\theta\cos\phi & -\rho \sin\theta\sin\phi & \rho \cos\theta\cos\phi \\
\sin\theta\sin\phi  & \rho \sin\theta\cos\phi  & \rho \cos\theta\sin\phi \\
3\rho^{2}\cos\theta & 0 & -\rho^{3}\sin\theta
\end{bmatrix}
\]
from which we can compute that
\[
\det D\Phi_{+}(\rho, \phi, \theta)
=-\rho^{4}\sin\theta(1+2\cos^{2}\theta) 
\]
which is nonzero for $\theta\notin\br{0, \pi}$.
Meanwhile, in a neighborhood of
$\theta\in\br{0, \pi}$ we can make the change of coordinates
\begin{align*}
X&=\sin\theta\cos\phi \\
Y&=\sin\theta\sin\phi
\end{align*}
to write
\[
\Phi_{+}(\rho, X, Y)=(\rho X, \rho Y, \rho^{3}\sqrt{1-X^{2}-Y^{2}}).
\]
Thus,
\[
D\Phi_{+}(\rho, X, Y)=
\begin{bmatrix}
X & \rho & 0 \\
Y & 0 & \rho \\
3\rho^{2}\sqrt{1-X^{2}-Y^{2}} & -\rho^{3}\frac{X}{\sqrt{1-X^{2}-Y^{2}}} &  -\rho^{3}\frac{Y}{\sqrt{1-X^{2}-Y^{2}}}
\end{bmatrix}
\]
and
\[
\det D\Phi_{+}(\rho, X, Y)
=
\frac{\rho^{4}(3-2X^{2}-2Y^{2})}{\sqrt{1-X^{2}-Y^{2}}}
\]
which is nonzero along the set $X=Y=0$ as required.
Thus $\Phi_{\pm}$ are immersions.

Finally recall that $\lambda_{\pm}$ were defined by
\[
\lambda_{\pm}=\pm dz+\frac{1}{2}(x\,dy-y\,dx).
\]
The 
maps $\Phi_{\pm}:(\rho, \phi, \theta)\in \R^{\pm}\times S^{2}\to(x, y, z)\in\R^{3}$ defined by
\begin{align*}
x&=\pm\rho \sin\theta\cos\phi &
dx&=\pm\bp{\sin\theta\cos\phi \,d\rho+\rho \cos\theta\cos\phi\,d\theta-\rho\sin\theta\sin\phi\,d\phi} \\
y&=\pm\rho \sin\theta\sin\phi &
dy&=\pm\bp{\sin\theta\sin\phi\,d\rho+\rho\cos\theta\sin\phi\,d\theta+\rho\sin\theta\cos\phi\,d\phi} \\
z&=\pm \rho^{3}\cos\theta &
dz&=\pm\bp{3\rho^{2}\cos\theta\,d\rho-\rho^{3}\sin\theta\,d\theta}
\end{align*}
so
\begin{align*}
y\,dx&=\rho\sin^{2}\theta\sin\phi\cos\phi\,d\rho+\rho^{2}\sin\theta\cos\theta\sin\phi\cos\phi\,d\theta-\rho^{2}\sin^{2}\theta\sin^{2}\phi\,d\phi \\
x\,dy&=\rho\sin^{2}\theta\sin\phi\cos\phi\,d\rho+\rho^{2}\sin\theta\cos\theta\sin\phi\cos\phi\,d\theta+\rho^{2}\sin^{2}\theta\cos^{2}\phi\,d\phi
\end{align*}
and hence
\[
x\,dy-y\,dx=\rho^{2}\sin^{2}\theta\,d\phi.
\]
We then find that
\begin{align*}
\Phi^{*}_{\pm}\lambda_{\pm}
&=3\rho^{2}\cos\theta\,d\rho-\rho^{3}\sin\theta\,d\theta+\frac{1}{2}\rho^{2}\sin^{2}\theta\,d\phi \\
&=\rho^{2}\bp{3\cos\theta\,d\rho-\rho\sin\theta\,d\theta+\frac{1}{2}\sin^{2}\theta\,d\phi} \\
&=\rho^{2}\lambda_{1}
\end{align*}
as required.
\end{proof}

We next compute the Reeb vector field of the contact form
$\lambda_{f}=f\lambda_{1}$ on $\R\times S^{2}$.
We start with a general lemma.

\begin{lemma}\label{l:reeb-lemma-app}
Let $(M, \lambda)$ be a contact $3$-manifold, let $f$ be a smooth positive function, and let
$\tl X_{f}$ be the unique section of $\xi=\ker\lambda$ satisfying
\begin{equation}\label{e:x-f}
i_{\Xf}d\lambda=df-df(X_{\lambda})\lambda
\end{equation}
where $X_{\lambda}$ is the Reeb vector field.  Then the Reeb vector field of
the contact form associated to $f\lambda$ is given by
\[
X_{f\lambda}=\tfrac{1}{f}X_{\lambda}+\tfrac{1}{f^{2}}\tl X_{f}.
\]
Moreover, if $\br{v_{1}, v_{2}}$ is a basis for $\xi$, then we have that
\begin{equation}\label{e:lem-basis-formula}
\tl X_{f}=[d\lambda(v_{1}, v_{2})]^{-1}(df(v_{2})v_{1}-df(v_{1})v_{2}).
\end{equation}
\end{lemma}

\begin{proof}
Let $\tl X_{f}$ be the unique section of $\xi$ satisfying $i_{\tl X_{f}}d\lambda=df-df(X_{\lambda})\lambda$.
Then
\[
df(\tl X_{f})=(i_{\tl X_{f}})^{2}d\lambda+f(X_{\lambda})i_{\tl X_{f}}\lambda=0
\]
since $\tl X_{f}\in\xi=\ker\lambda$.
We then find that
\begin{align*}
i_{X_{\lambda}}d(f\lambda)
&=i_{X_{\lambda}}df\wedge\lambda +i_{X_{\lambda}}(fd\lambda) \\
&=i_{X_{\lambda}}df\wedge\lambda  \\
&=df(X_{\lambda})\lambda-df
\end{align*}
while
\begin{align*}
i_{\tl X_{f}}d(f\lambda)
&=i_{\tl X_{f}}df\wedge\lambda +i_{\tl X_{f}}(fd\lambda) \\
&=i_{\tl X_{f}}(fd\lambda) \\
&=f\,i_{\tl X_{f}}d\lambda
\end{align*}
so
\begin{align*}
i_{\tfrac{1}{f}X_{\lambda}+\tfrac{1}{f^{2}}\tl X_{f}}d(f\lambda)
&=\tfrac{1}{f}i_{X_{\lambda}}d(f\lambda)+\tfrac{1}{f^{2}}i_{\tl X_{f}}d(f\lambda) \\
&=\tfrac{1}{f}[df(X_{\lambda})\lambda-df]+\tfrac{1}{f^{2}}[f\,i_{\tl X_{f}}d\lambda] \\
&=\tfrac{1}{f}[df(X_{\lambda})\lambda-df+\,i_{\tl X_{f}}d\lambda] \\
&=0
\end{align*}
by definition of $\tl X_{f}$.
Furthermore
\begin{align*}
i_{\tfrac{1}{f}X_{\lambda}+\tfrac{1}{f^{2}}\tl X_{f}}(f\lambda)
&=\tfrac{1}{f}i_{X_{\lambda}}(f\lambda)+\tfrac{1}{f^{2}}i_{\tl X_{f}}(f\lambda) \\
&=i_{X_{\lambda}}\lambda+\tfrac{1}{f}i_{\tl X_{f}}\lambda \\
&=1+0=1.
\end{align*}
Thus $\tfrac{1}{f}X_{\lambda}+\tfrac{1}{f^{2}}\tl X_{f}$ is the Reeb vector field of $f\lambda$ as claimed.

Next, to verify the second claim, we define $\tl X_{f}$ by \eqref{e:lem-basis-formula} and verify that this
$\tl X_{f}$ satisfies \eqref{e:x-f}.  Computing, we have that
\[
i_{\tl X_{f}}d\lambda-df-df(X_{\lambda})\lambda
=[d\lambda(v_{1}, v_{2})]^{-1}[df(v_{2})d\lambda(v_{1}, \cdot)-df(v_{1})d\lambda(v_{2}, \cdot)]
-df-df(X_{\lambda})\lambda.
\]
Using that $i_{v_{i}}\lambda=0$, $i_{X_{\lambda}}\lambda=1$ and $i_{X_{\lambda}}d\lambda=0$
we see that each of the vectors $X_{\lambda}$, $v_{1}$, and $v_{2}$ yields $0$ when evaluated
on the right hand side of this equation above.  Since these vectors form a basis for the tangent space,
it follows that this quantity vanishes on $TM$ and thus \eqref{e:x-f} is satisfied.
\end{proof}

\begin{lemma}[Lemma \ref{l:lambda-f-properties}]\label{l:lambda-f-properties-app}
Recalling the definition
\[
g(\theta)=2\cos^{2}\theta+1=3\cos^{2}\theta+\sin^{2}\theta
\]
from \eqref{e:g-def},
we have for $\theta\notin\br{0, \pi}$:
\begin{itemize}
\item The set
\begin{align*}
\B_{(\rho, \theta, \phi)}
&=\br{(fg)^{-1}(-3\cot\theta\,\dph+\frac{1}{2}\sin\theta\,\drh), 2\rho\csc\theta\,\dph+\dth} \\
&=:\br{v_{1}(\rho, \theta, \phi), v_{2}(\rho, \theta, \phi)}
\end{align*}
is a symplectic basis for
$(\xi_{1}, d\lambda_{f})$.

\item The Reeb vector field $\Xf$ of the contact form $\lambda_{f}$ is given by
\begin{align*}
\Xf=
[gf^{2}]^{-1}\Big[&
(-\rho f_{\rho}-3f_{\theta}\cot\theta+2f)\,\dph  \\
&
+(3\cot\theta f_{\phi}-\frac{1}{2}\sin\theta f_{\rho})\,\dth  \\
&
+(\rho f_{\phi}+\frac{1}{2}\sin\theta f_{\theta}+f\cos\theta)\,\drh
\Big].
\end{align*}
\end{itemize}
\end{lemma}

\begin{proof}
The vectors $v_{1}$ and $v_{2}$ are clearly linearly independent for
$\theta\notin\br{0, \pi}$ since
$v_{1}$ has nonzero $\drh$ component and vanishing $\dth$ component, while the opposite 
is true of $v_{2}$.
Recalling that 
\begin{align*}
\lambda_{1}&=3\cos\theta\,d\rho-\rho\sin\theta\,d\theta+\frac{1}{2}\sin^{2}\theta\,d\phi \\
d\lambda_{1}&=(2\sin\theta\,d\rho-\cos\theta\sin\theta\,d\phi)\wedge d\theta
\end{align*}
we immediately find that
\[
\lambda_{1}(v_{1})=(fg)^{-1}\bbr{(3\cos\theta)(\tfrac{1}{2}\sin\theta)+(\tfrac{1}{2}\sin^{2}\theta)(-3\cot\theta)}=0
\]
and
\[
\lambda_{1}(v_{2})=-\rho\sin\theta+(\tfrac{1}{2}\sin^{2}\theta)(2\rho\csc\theta)=0
\]
so
$v_{1}$ and $v_{2}\in\xi_{1}$.
Finally, we have that
\begin{align*}
i_{v_{1}}d\lambda_{1}
&=(fg)^{-1}(\sin^{2}\theta+3\cot\theta\cos\theta\sin\theta)\,d\theta \\
&=(fg)^{-1}(\sin^{2}\theta+3\cos^{2}\theta)\,d\theta \\
&=1/f\,d\theta \\
\end{align*}
and thus
\begin{equation}\label{e:d-lambda-1-w1-w2}
\begin{aligned}
d\lambda_{1}(v_{1}, v_{2})
&=i_{v_{2}}(i_{v_{1}}\lambda_{1}) \\
&=i_{v_{2}}(1/f\,d\theta)\\
&=1/f.
\end{aligned} 
\end{equation}
Since $\lambda_{1}(v_{1})=\lambda_{1}(v_{2})=0$, if follows that
\[
d\lambda_{f}(v_{1}, v_{2})=fd\lambda_{1}(v_{1}, v_{2})=1
\]
and thus
$\br{v_{1}, v_{2}}$ is a symplectic basis for $(\xi_{1}, d\lambda_{f})$ for $\theta\notin\br{0, \pi}$.

Next to compute the Reeb vector field of $\lambda_{f}$ we first observe that
the vector field $X_{1}$ defined by
\[
X_{1}=g(\theta)^{-1}\bp{\cos\theta\,\drh+2\,\dph}.
\]
satisfies
\[
\lambda_{1}(X_{1})=g(\theta)^{-1}\bbr{3\cos^{2}\theta+\sin^{2}\theta}=1
\]
and
\[
i_{X_{1}}d\lambda_{1}
=g(\theta)^{-1}\bbr{2\sin\theta\cos\theta-2\sin\theta\cos\theta}\,d\theta =0
\]
so
$X_{1}$ is the Reeb vector field of $\lambda_{1}$.
According to Lemma \ref{l:reeb-lemma-app}, the Reeb vector field of
$\Xf$ of $\lambda_{f}$ is then given by
\[
\Xf=\tfrac{1}{f} X_{1}+\tfrac{1}{f^{2}}\tl X_{f}
\]
with
\begin{align*}
\tl X_{f}
&=[d\lambda_{1}(v_{1}, v_{2})]^{-1}(df(v_{2})v_{1}-df(v_{1})v_{2}) \\
&=f(df(v_{2})v_{1}-df(v_{1})v_{2})
\end{align*}
where we've applied \eqref{e:d-lambda-1-w1-w2} in the second line.
Computing, we have that
\begin{align*}
f\,df(v_{2})v_{1}
&=g(\theta)^{-1}\bp{2\rho\csc\theta f_{\phi}+f_{\theta}}\bp{-3\cot\theta\,\dph+\frac{1}{2}\sin\theta\,\drh} \\
&=g(\theta)^{-1}\bbr{\bp{-6\rho\csc\theta\cot\theta f_{\phi}-3\cot\theta f_{\theta}}\dph+\bp{\rho f_{\phi}+\frac{1}{2}\sin\theta f_{\theta}}\drh}
\intertext{and}
f\,df(v_{1})v_{2}
&=g(\theta)^{-1}\bp{-3\cot\theta f_{\phi}+\frac{1}{2}\sin\theta f_{\rho}}\bp{2\rho\csc\theta\,\dph+\dth} \\
&=g(\theta)^{-1}\bbr{\bp{-6\rho\csc\theta\cot\theta f_{\phi} +\rho f_{\rho}}\dph+\bp{-3\cot\theta f_{\phi}+\frac{1}{2}\sin\theta f_{\rho}}\dth}.
\end{align*}
Combining the above we conclude that
\begin{align*}
\Xf=
[gf^{2}]^{-1}\Big[&
(-\rho f_{\rho}-3f_{\theta}\cot\theta+2f)\,\dph  \\
&
+(3\cot\theta f_{\phi}-\frac{1}{2}\sin\theta f_{\rho})\,\dth  \\
&
+(\rho f_{\phi}+\frac{1}{2}\sin\theta f_{\theta}+f\cos\theta)\,\drh
\Big]
\end{align*}
as claimed.
\end{proof}

\begin{lemma}\label{l:conley-zehnder-app} 
Let $C\in Sp(1)$ be a symplectic matrix
and let $k\ne 0$ be a constant.
Then the path
$\Psi_{C, k}:[0, 1]\to Sp(1)$ defined by
\[
\Psi_{C, k}(t)=C
\begin{bmatrix}
e^{kt} & 0 \\ 0 & e^{-kt}
\end{bmatrix}
C^{-1}
\]
has Conley--Zehnder index $\pathcz(\Psi_{C, k})=0$.
\end{lemma}

\begin{proof}
We first consider the case
$C=I$, i.e.\ the path given by
\[
\Psi_{I, k}=
\begin{bmatrix}
e^{kt} & 0 \\ 0 & e^{-kt}
\end{bmatrix}.
\]
The path $\Psi_{I, k}$ is easily seen to be homotopic within
$\Sigma(1)$ to its inverse
\[
\Psi_{I, k}^{-1}=
\Psi_{I, -k}=
\begin{bmatrix}
e^{-kt} & 0 \\ 0 & e^{kt}
\end{bmatrix}.
\]
via the homotopy
\[
\Psi_{s}(t)
:=
R(s\pi/2)
\Psi_{I, k}(t)
R(s\pi/2)^{-1}
=
R(s\pi/2)
\Psi_{I, k}(t)
R(-s\pi/2)
\]
where $R(\theta)$ is the rotation matrix
\[
R(\theta)
=
\begin{bmatrix}
\cos\theta & -\sin\theta \\ \sin\theta & \cos\theta
\end{bmatrix}.
\]
The homotopy axiom from Theorem \ref{t:conley-zehnder-axioms}
then implies that
\[
\pathcz(\Psi_{I, k})=\pathcz(\Psi_{I, k}^{-1})
\]
while the inverse axiom implies that
\[
\pathcz(\Psi_{I, k})=-\pathcz(\Psi_{I, k}^{-1}).
\]
We conclude that
\[
\pathcz(\Psi_{I, k})=0.
\]

For $C\ne I$, we can use the fact that the symplectic group is path connected
to find a path $C_{s}\in Sp(1)$ with $C_{0}=C$ and $C_{1}=I$.
We then construct a homotopy $\Psi_{s}\in\Sigma(1)$ defined by
\[
\Psi_{s}(t):=\Psi_{C_{s}, k}(t)=
C_{s}
\begin{bmatrix}
e^{kt} & 0 \\ 0 &  e^{-kt}
\end{bmatrix}
C_{s}^{-1}
\]
with $\Psi_{0}=\Psi_{C, k}$ and $\Psi_{1}=\Psi_{I, k}$.
The homotopy invariance axiom of Theorem \ref{t:conley-zehnder-axioms}
and the result of the previous paragraph then imply
\[
\pathcz(\Psi_{C, k})=\pathcz(\Psi_{I, k})=0
\]
as claimed.
\end{proof}

\begin{lemma}\label{l:function-construction}
For any $\ep>0$, there exists a smooth, positive function $f_{\ep}:\R\to\R^{+}$ satisfying:
\begin{enumerate}
\item $f_{\ep}(x)=x^{2}$ whenever $\abs{x}\ge\ep$,
\item $xf_{\ep}'(x)>0$ for all $x\ne 0$, and
\item $f_{\ep}''(0)>0$.
\end{enumerate}
\end{lemma}

\begin{proof}
Given $\ep>0$ choose a smooth function
$\beta_{\ep}:\R\to [0,1]$  satisfying
\begin{itemize}
\item $\beta_{\ep}(x)=1$ for $\abs{x}\ge \ep$,
\item $\beta_{\ep}(x)=0$ for $\abs{x}\le \ep/2$, and
\item $x \beta_{\ep}'(x)\ge 0$ for $\abs{x}\in(\ep/2, \ep)$ (and thus all $x\in\R$).
\end{itemize}
Since $\beta_{\ep}'$ is compactly supported
and vanishes for $\abs{x}\le \ep/2$
it follows that the function 
$x\mapsto x^{-1}\beta'(x)$ is smooth,
compactly supported, and thus bounded.
\begin{figure}
\scalebox{.75}{\includegraphics{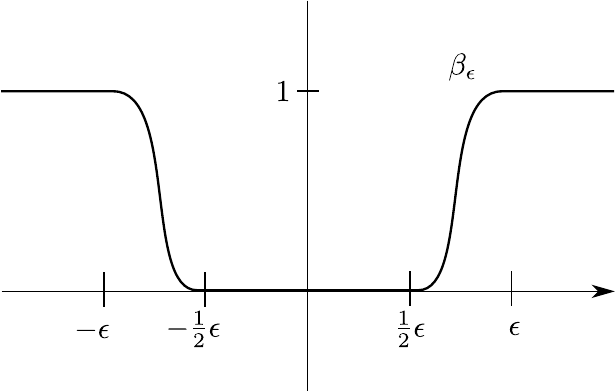}}
\end{figure}
\begin{figure}
\scalebox{.75}{\includegraphics{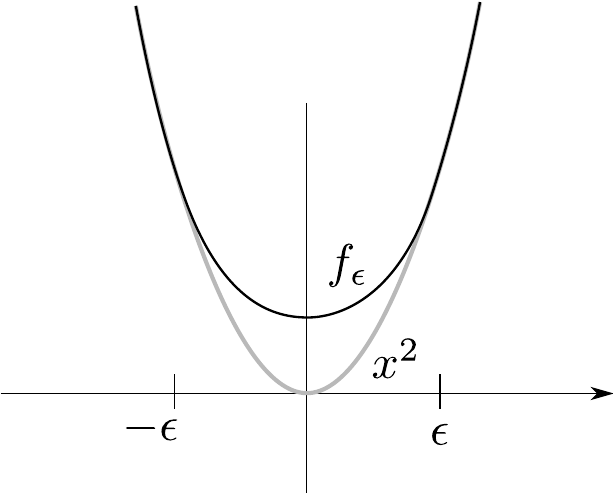}}
\end{figure}
We can thus find a $c_{\ep}>0$ so that
\[
\frac{1}{c_{\ep}}\ge 2 \max \bp{x^{-1}\beta_{\ep}'(x)}
\]
and hence
\begin{equation}\label{e:c-epsilon}
\begin{aligned}
1-x^{-1}\beta_{\ep}'(x) c_{\ep}
&\ge 1-\bp{\max x^{-1}\beta_{\ep}'(x)} c_{\ep} \\
&\ge 1-\frac{1}{2}=\frac{1}{2}
\end{aligned}
\end{equation}
for all $x\in\R$.
Defining
\begin{align*}
f_{\ep}(x)
&=\beta_{\ep}(x)x^{2}+(1-\beta_{\ep}(x))(\tfrac{1}{2} x^{2}+c_{\ep})  \\
&=\tfrac{1}{2}(\beta_{\ep}(x)+1) x^{2}+(1-\beta_{\ep}(x))c_{\ep}
\end{align*}
it's immediately clear that $f_{\ep}$
is smooth, positive, and 
satisfies
the first and third conditions listed in the lemma.

To check
that $f_{\ep}$ satisfies the second condition
we compute using
$x\beta_{\ep}'(x)\ge 0$, $\beta_{\ep}(x)\ge 0$, and \eqref{e:c-epsilon}, and find
\begin{align*}
x f_{\ep}'(x)
&=x\bbr{\tfrac{1}{2}\beta_{\ep}'(x) x^{2}+(\beta_{\ep}(x)+1) x-\beta_{\ep}'(x) c_{\ep}} \\
&=\tfrac{1}{2}[x\beta_{\ep}'(x)] x^{2}+(\beta_{\ep}(x)+1) x^{2}-x\beta_{\ep}'(x) c_{\ep} \\
&\ge x^{2}-x\beta_{\ep}'(x) c_{\ep}  \\
&=x^{2}(1-x^{-1}\beta_{\ep}'(x) c_{\ep}) \\
&\ge \tfrac{1}{2}x^{2}.
\end{align*}
Since $\tfrac{1}{2}x^{2}>0$ for $x\ne 0$, $xf_{\ep}'(x)>0$ for all $x\ne 0$,
and this completes the proof.
\end{proof}

\bibliographystyle{../../hplain5}
\bibliography{../fol-bibdata}

\end{document}